\documentclass[12pt]{article}

\usepackage{rotating}
\usepackage{graphicx}													
\usepackage{amsmath}
\usepackage{amsthm}						
\usepackage{amsfonts}								
\usepackage{amssymb}
\usepackage{tabu}
\usepackage{enumerate}
\usepackage{adjustbox}

\newcommand\cF{{\cal F}}
\newcommand\F{\mathbb{F}}
\newcommand\la{\langle}
\newcommand\ra{\rangle}
\newcommand\dla{\la\!\la}
\newcommand\dra{\ra\!\ra}
\newcommand\Gm{\Gamma}
\newcommand\Aut{\operatorname{Aut}}

\newtheorem{lemma}{Lemma}[section]
\newtheorem*{remark}{Remark}
\newtheorem{theorem}[lemma]{Theorem}
\newtheorem{definition}[lemma]{Definition}

\newtheorem{proposition}[lemma]{Proposition}
\newtheorem{question}[lemma]{Question}

\title{Axial algebras of Monster type $(2\eta, \eta)$ for $D$ diagrams. I}
\date{}
\author{Andrey Mamontov and Alexey Staroletov\footnote{The work is supported by Mathematical Center in Akademgorodok under agreement No.075-15-2022-282 with the Ministry of Science and Higher Education of the Russian
Federation}}

\begin{document}
\maketitle
\newcommand{\Addresses}{{
		\bigskip\noindent
		\footnotesize
		Andrey~Mamontov, \textsc{Sobolev Institute of Mathematics, Novosibirsk, Russia;}\\\nopagebreak
		\textsc{Novosibirsk State University, Novosibirsk, Russia;}\\\nopagebreak
		\textit{E-mail address: } \texttt{andreysmamontov@gmail.com}
		
		\medskip\noindent
		Alexey~Staroletov, \textsc{Sobolev Institute of Mathematics, Novosibirsk, Russia;}\\\nopagebreak
		\textsc{Novosibirsk State University, Novosibirsk, Russia;}\\\nopagebreak
		\textit{E-mail address: } \texttt{staroletov@math.nsc.ru}
  	
		\medskip
}}

\begin{abstract}
Axial algebras are a class of commutative algebras generated by idempotents, with adjoint action semisimple and satisfying a prescribed fusion law. Axial algebras were introduced
by Hall, Rehren, and Shpectorov in 2015 as a broad generalization of Majorana algebras of Ivanov, whose axioms were derived from the properties of the Griess algebra for the Monster group. 

The class of Matsuo algebras was introduced by Matsuo and later generalized by
Hall, Rehren, and Shpectorov. A Matsuo algebra $M$ is built by a set of 3-transpositions $D$.
Elements of $D$ are idempotents in $M$ and called axes. In particular, $M$ is an example of an axial algebra. It is known that double axes, i.e., sums of two orthogonal axes in a Matsuo algebra, satisfy the fusion law of Monster type. 
This observation shows that a set consisting of axes and double axes can generate a subalgebra of Monster type in the Matsuo algebra. Subalgebras corresponding to various series of 3-transposition groups are extensively studied by many authors. 

In this paper, we study primitive subalgebras generated by a single axis and two double axes. 
We classify all such subalgebras in seven out of nine possible cases for a diagram on 3-transpositions that are involved in the generating elements. We also construct several infinite series of axial algebras of Monster type generalizing our 3-generated algebras.

\end{abstract}

\section{Introduction}
\label{Introduction}

{\it Axial algebras} are commutative non-associative algebras generated by a set of idempotents, called axes, satisfying a prescribed fusion law. 
This class of algebras was introduced by Hall, Rehren, and Shpectorov~\cite{hrs} as a broad generalization of the class of Majorana algebras defined by Ivanov in the framework of Majorana theory~\cite{ivan}. The main inspiration for both theories is the Griess algebra, which
is a real commutative nonassociative algebra of dimension 196884 that has the Monster group as its automorphism group.
More information on axial algebras 
and related topics can be found in the recent survey~\cite{survey}.

The class of {\it Matsuo algebras} was introduced
by Matsuo~\cite{Matsuo} and later generalized in~\cite{hrs1}.
Recall that a group $G$ is a  {\it 3-transposition group} if it is generated by a normal set $D$ of involutions such that the order of the product of any pair of these involutions is not greater than three. If $\mathbb{F}$ is a field of characteristic not two
and $\eta\in\mathbb{F}\setminus\{0,1\}$, then
the Matsuo algebra $M_\eta(G,D)$ has $D$ as its basis, where each element of $D$ is an idempotent. Moreover, if $c,d\in D$ 
and $|cd|=2$, then their product in $M_\eta(G,D)$ equals 0,
and if $|cd|=3$, then their product equals $\frac{\eta}{2}(c+d-c^d)$. It turned out that the adjoint operators of elements $d\in D$, considered as idempotents of $M_\eta(G,D)$, have the minimal polynomial dividing $(x-1)x(x-\eta)$ and restrictive multiplication rules on their eigenvectors. These properties generalize the
Peirce decomposition for idempotents in Jordan algebras, where 
$1/2$ is replaced with $\eta$. Matsuo algebras are examples of so-called axial algebras of Jordan type $\eta$, generating axes are exactly involutions of $D$. In fact, each element $d\in D$ is a {\it primitive axis} \cite{hrs1}, that is the equality $dx=x$ in the algebra implies that $x$ is a scalar multiple of $d$.

The double axis construction was introduced by Joshi in \cite{Joshi-mas,Joshi-phd} and developed in \cite{gjmss}. Suppose that $M=M_\eta(G,D)$ is a Matsuo algebra, where $\eta\neq\frac{1}{2}$. 
We say that the axes $a, b\in D$ are orthogonal if $ab = 0$, 
or equivalently $a$ and $b$ commute as involutions in $G$. If $a$ and $b$ are orthogonal, then $x := a+b$ is an idempotent in $M$ called a {\it double axis}. It is proved in \cite{Joshi-phd,gjmss} that every double axis obeys the fusion law of Monster type $\mathcal{M}(2\eta,\eta)$ (see Fig.~\ref{M}).
Thus, this construction links three different objects:  3-transposition groups and axial algebras of Jordan and Monster types. We note that similar consideration was very useful in finite groups \cite{norton}, where the product of any two commuting 3-transpositions is a 6-transposition: an obvious example is provided by conjugacy classes of  permutations $(1,2)$ and $(1,2)(3,4)$ in a symmetric group of degree at least four.

In~\cite{Joshi-phd,gjmss}, authors started to classify all primitive 2- and 3-generated subalgebras $M$ of Matsuo algebras, where each generating element is either a single axis or a double axis.
Following~\cite{gjmss}, we say $M$ has type $A$ if it is generated by a single axis and a double axis; $M$ has type $B$ if it is generated by two double axes. 
Similarly, for the 3-generated case we consider three possibilities: type $C$ for a generating  set comprising of two single axes and one double axis; type $D$ for one single axis and two double axes; and type $E$ for three double axes. All possible primitive algebras of types $A$, $B$, and $C$ were classified and described in \cite{Joshi-phd,gjmss}. In this work we study primitive algebras of $D$ type. The classification for type $E$ is an ongoing project.

Suppose that an algebra $A$ is generated by a single axis $a$ and double axes $b+c$ and $d + e$, where $a, b, c, d, e$ are distinct elements of
a set of 3-transpositions $D$. 
The diagram on the support set $\{a, b, c, d, e\}$ is the graph on this set in which two points are connected by an edge iff the corresponding involutions do not commute in the group. 
It is proved in~\cite{gjmss} that there are nine possible diagrams (see Fig.~\ref{Type D}), numbering by $D_1,D_2,\ldots,D_9$. In this paper, we classify all primitive algebras of $D$ type for first seven diagrams. We consider here only seven diagrams, since in these cases the resulting groups and algebras are not too large, which allows us to reproduce almost all  calculations by hand. The other two diagrams contain a large number of subcases, the consideration of which depends on computer calculations. 
Our main result is the description of algebras of $D$ type, we formulate it shortly as follows.
 
\begin{theorem} Suppose that $M=M_\eta(G, D)$ is a Matsuo algebra, where $\eta\neq\frac{1}{2}$.
Suppose that $A$ is a primitive axial subalgebra in $M$ generated by a single axis and two double axes. Denote the diagram on the support set of the generators of $A$ by $T$.
\begin{enumerate}[(i)]
\item If $T$ has type $D_4$, then $\dim A\in\{7,13,20\}$.
\item If $T$ has type $D_5$, then $A$ is isomorphic to the 9-dimensional algebra $Q_3(\eta)$.
\item If $T$ has type $D_7$ and $\eta\neq2$, then $\dim A\in\{9,12,39,42,30,90\}$,
and if $T$ has type $D_7$ and $\eta=2$, then $\dim A\in\{9,12,39,42,29,89\}$.
\end{enumerate}
\end{theorem}

In this theorem, we only list the dimensions of algebras corresponding to connected diagrams.
If the diagram is disconnected, then any suitable algebra decomposes as a direct sum of smaller known algebras. We consider all algebras including decomposable ones in detail in Section~\ref{s:main}.
We provide a complete multiplication table for a basis of the algebra for small dimensions.

%

Authors in \cite{gjmss} introduced the general flip  construction
which helps to produce a rich variety of axial algebras of Monster type. 
Flip subalgebras corresponding to various series of 3-transposition groups are extensively studied by many authors. Joshi classified all flip subalgebras for $G=S_n$ and $G=Sp_{2n}(2)$ \cite{Joshi-mas,Joshi-phd}. Alsaeedi completed the $2^{n-1}:S_n$ case \cite{alsaeedi-phd,alsaeedi-paper}. Shi~\cite{shi}  analyzed the case of $O^\pm_{2n}(2)$. The case of $G=SU_n(2)$ was looked
at by Hoffman, Rodrigues, and Shpectorov (yet unpublished work). Generalizing the 3-generated algebras from the theorem, where possible, we construct an infinite series of algebras corresponding to a series of 3-transposition groups to which the group generated by $a$, $b$, $c$, $d$, and $e$ belongs. The resulting infinite series are described in Propositions~\ref{p:tau:W2(Dn)}, \ref{p:fixed:Wr(A_4,n)}, \ref{p:fixed2:Wr(A_4,n)}, \ref{p:fixed:Wr(3^1+2,n)}, and \ref{p:fixed:Wr(3^2,n)}. 

The paper is organized as follows. In Section 2, we
 discuss general facts and definitions on axial algebras,
Matsuo algebras and their fixed and flip subalgebras.
In Section 3, we describe Fischer spaces of some 3-transposition groups.
Finally, in Section 4, we classify all primitive subalgebras for diagrams $D_1$--$D_7$.

\section{Preliminaries}
In this section, we provide the necessary background on axial algebras, then
introduce notion of Matsuo algebras and their fixed and flip subalgebras. 

\subsection{Axial algebras}
Fix a field $\mathbb{F}$. 

\begin{definition} A \emph{fusion law} $\mathcal{F}$ over $\mathbb{F}$ is a pair $(X,\ast)$,
where $X$ is a finite subset of $\F$ and $\ast:X\times X\to 2^X$ is a symmetric map from $X^2$ to the set of all subsets of $X$.
\end{definition}

Consider a commutative $\mathbb{F}$-algebra $A$. 
If $a\in A$, then $ad_a : A\to A$ stands for the adjoint map defined as $u\mapsto au$. For $\lambda\in\mathbb{F}$, denote $A_{\lambda}(a)=\{u\in A~|~au=\lambda u\}$.
If $A_{\lambda}(a)\neq\{0\}$, then it coincides with the $\lambda$-eigenspace of $\operatorname{ad}_a$. For $L\subseteq\mathbb{F}$, set $A_L(a):=\oplus_{\lambda\in L}A_{\lambda}(a)$. 

Suppose that $\mathcal{F}=(X,\ast)$ is a fusion law over $\mathbb{F}$. 

\begin{definition}
A non-zero idempotent $a\in A$ is an \emph{($\mathcal{F}$-)axis} if 
\begin{enumerate}
	\item[(a)] $A=A_X(a)$; and 
	\item[(b)] $A_{\lambda}(a)A_\mu(a)\subseteq A_{\lambda\ast\mu}(a)$ for all 
	           $\lambda,\mu\in X$.
\end{enumerate}
\end{definition}

Condition (a) implies that $\operatorname{ad}_a$ is diagonalizable over $\F$. Since $a\in A_1(a)$, we always assume that $1\in X$. 

\begin{definition}
An axis $a$ is \emph{primitive} if $A_1(a)$ is $1$-dimensional; that is, 
$A_1(a)=\langle a\rangle$.
\end{definition}

If $a$ is a primitive axis and $\lambda\neq0$, then $A_1(a)A_0(a)=0$ and $A_1(a)A_{\lambda}(a)=\langle a\rangle A_{\lambda}(a)=A_{\lambda}(a)$.
This forces us to assume that $1\ast 0=\emptyset$ (if $0\in X$) and $1\ast\lambda=\{\lambda\}$ for all $0\neq\lambda\in X$. 

\begin{definition}
An algebra $A$ over $\F$ is an \emph{($\cF$-)axial} algebra if it is 
generated as algebra by a set of $\cF$-axes. The algebra $A$ is a 
\emph{primitive} \emph{($\cF$-)axial} algebra if it is generated as algebra by a set of primitive $\cF$-axes.
\end{definition}

In this paper, we focus on two types of fusion laws. 
Consider elements $\alpha \not = \beta$ in $\F$ distinct from 
$1$ and $0$. Table \ref{M} shows the fusion law $\mathcal{M}(\alpha,\beta)$.
\begin{table}[h]
\begin{center}
\begin{tabular}{|c||c|c|c|c|}
\hline
$\ast$&$1$&$0$&$\alpha$&$\beta$\\
\hline\hline
$1$&$1$&&$\alpha$&$\beta$\\
\hline
$0$&&$0$&$\alpha$&$\beta$\\
\hline
$\alpha$&$\alpha$&$\alpha$&$1,0$&$\beta$\\
\hline
$\beta$&$\beta$&$\beta$&$\beta$&$1,0,\alpha$\\
\hline
\end{tabular}
\end{center}
\caption{Fusion law $\mathcal{M}(\alpha,\beta)$}\label{M}
\end{table}
Each cell of the table lists the elements of the corresponding 
set $\lambda\ast\mu$. For example, $1\ast 0=\emptyset$ and $\alpha\ast\alpha=\{1,0\}$.
Axial algebras with this fusion law are called \emph{algebras of 
Monster type $(\alpha,\beta)$}. The name is explained by the fact that
these algebras generalize the Majorana algebras of Ivanov~\cite{ivan} and the Griess algebra that are examples of axial algebras of Monster type $(\frac{1}{4},\frac{1}{32})$ having strong connections with the Monster group.

Consider $\eta\in\F$ distinct from $1$ and $0$. Table \ref{Jordan} describes the fusion law $\mathcal{J}(\eta)$. Note that $\mathcal{J}(\eta)$ is a particular case of the Monster fusion laws $\mathcal{M}(\alpha,\beta)$, namely, it corresponds to the case $\beta=\eta$ and $\alpha$-eigenspace is trivial.

\begin{table}[h]
\begin{center}
\begin{tabular}{|c||c|c|c|}
\hline
$\ast$&$1$&$0$&$\eta$\\
\hline\hline
$1$&$1$&&$\eta$\\
\hline
$0$&&$0$&$\eta$\\
\hline
$\eta$&$\eta$&$\eta$&$1,0$\\
\hline
\end{tabular}
\end{center}
\caption{Fusion law $\mathcal{J}(\eta)$}\label{Jordan}
\end{table}

Axial algebras with this fusion law are called \emph{algebras of Jordan type $\eta$}. Observe that for $\eta=1/2$ the fusion law  $\mathcal{J}(\frac{1}{2})$ imitates the Peirce decomposition for an idempotent in a Jordan algebra \cite{Ja68}. 

\subsection{3-transposition groups and Matsuo algebras}

Recall that a \emph{$3$-transposition group} is a pair $(G,D)$, where $G$ is a group generated by a normal subset $D$ of its involutions such that
$|cd|\leq 3$ for every two elements $c,d\in D$. 

The symmetric group $G=S_n$ of degree $n$ together with the set of all transpositions $D=(1,2)^G$ is an example of a 3-transposition group. Finite $3$-transposition groups were classified, under some
restrictions, by Fischer \cite{fischer} and, in complete generality, by Cuypers and Hall \cite{cuypers_hall}.

In many cases, $D$ is unique for a given $G$, so 
we often write $G$ instead of $(G,D)$.

The \emph{Fischer space} of a 3-transposition group 
$(G,D)$ is a point-line geometry $\Gm=\Gm(G,D)$, whose point set is $D$ 
and where distinct points $c$ and $d$ are collinear if and only if $|cd|=3$. 
Observe that any two collinear points $c$ and $d$ lie in a 
unique common line, which consists of $c$, $d$, and the third point 
$e=c^d=d^c$. It follows form the definition that the connected components of the Fischer space $\Gm$ coincide with the conjugacy classes of $G$ contained in $D$. In particular, the Fischer space is connected if and only if $D$ is a single conjugacy class of~$G$.

The $3$-transposition group $(G,D)$ can be recovered from $\Gm$ up to the center of $G$. We say that 
two $3$-transposition groups $(G_1,D_1)$ and $(G_2,D_2)$ have the same central type if the 3-transposition groups $(\overline{G}_i,\overline{D}_i)$ are isomorphic, where $\overline{G}_i=G_i/Z(G_i)$. It is known~\cite{hs} that for each 3-transposition group $(G,D)$, there exists a unique centrally universal 3-transposition group $(\tilde{G},\tilde{D})$ which has all 3-transposition groups with central type that of $(G,D)$ as homomorphic images. We will often identify two 3-transposition groups that have the same central type. 
Following~\cite{hs}, we denote by $F(k,n)$ a centrally universal $k$-generator 3-transposition group whose 3-transposition class has size $n$. Note that for small values of $n$ the parameters $n$ and $k$ define a group $F(k,n)$ uniquely if there exists at least one such group \cite{hs}.

Consider a field $\F$ of characteristic not two and suppose that $\eta\in\F$, $\eta\neq0,1$. 

\begin{definition}
The \emph{Matsuo algebra} $M_\eta(\Gm)$ (or $M_\eta(G,D)$) over $\F$, corresponding to $\Gm$ 
and $\eta$, has the point set $D$ as its basis. Multiplication is defined 
on the basis as follows:
$$c\cdot d=\left\{
\begin{array}{rl}
c,&\mbox{if }c=d;\\
0,&\mbox{if $c$ and $d$ are non-collinear};\\
\frac{\eta}{2}(c+d-e),&\mbox{if $c$ and $d$ lie on a line $\{c,d,e\}$}.
\end{array}
\right.$$ 
\end{definition}
We use the dot for the algebra product to distinguish it from 
the multiplication in the group $G$. 
It follows from \cite[Theorem~6.4]{hrs} that 
any Matsuo algebra $M_\eta(G,D)$ is a primitive axial algebra of Jordan type $\eta$.
The following result shows that this is a basic example if $\eta\neq\frac{1}{2}$.
\begin{proposition}[\cite{hrs,hss}] \label{characterization}
Any primitive axial algebra of Jordan type $\eta\neq\frac{1}{2}$ is a factor algebra of a Matsuo algebra.
\end{proposition} 

The Matsuo algebra admits a bilinear symmetric form $(\cdot,\cdot)$ that associates with the algebra product, this means $(u\cdot v,w)=(u,v\cdot w)$ for arbitrary algebra elements $u$, $v$, and $w$.
This form is given on the basis $D$  by the following:
$$
(c,d)=\left\{
\begin{array}{rl}
1,&\mbox{ if }c=d;\\
0,&\mbox{ if }|cd|=2;\\
\frac{\eta}{2},&\mbox{ if }|cd|=3.
\end{array}\right.
$$
Such forms play an important role in the study of axial algebras, for example the radical of the form often contains all ideals of the corresponding axial algebra \cite{kms}. In this regard, for a given group of 3-transpositions $(G,D)$, an important question is to find the values of the parameter $\eta$ for which the Gram matrix of the form is degenerate. 
We call the corresponding elements of the field {\it critical values} for $\eta$.
Note that for a given 3-transposition group the critical values can be found using the recent paper \cite{hs2}.

Consider a Matsuo algebra $M_\eta(\Gm)$.
Since non-collinear axes $a$ and $b$ are orthogonal, that is $a\cdot{b}=0$, we see that $x=a+b$ is an idempotent. 
Following~\cite{gjmss}, call such $x$ a {\it double axis}. 
It turned out that double axes obey a nice fusion law.

\begin{proposition}\cite[Theorem~1.1]{gjmss}
If $\eta\neq1/2$, then the double axis $x=a+b$ satisfies the fusion law $\mathcal{M}(2\eta,\eta)$.
\end{proposition}

The concept of double axes is opposed to single axes, which are the elements of $D$. As was mentioned above single axes also satisfy the above law, with the $2\eta$-eigenspace trivial. Thus, any subalgebra in $M_\eta(\Gm)$ generated by a collection of single and double axes is an axial algebra of Monster type $(2\eta,\eta)$. Observe that the double axis $x=a+b$ is not primitive in the whole Matsuo algebra, because its 1-eigenspace is 2-dimensional, equal to $\langle a, b\rangle$. However, $x$ can be primitive in a proper subalgebra of $M_\eta(\Gm)$.

In~\cite{gjmss}, authors started to classify all primitive 2- and 3-generated subalgebras $M$ of Matsuo algebras, where each generating element is either a single axis or a double axis.
Following~\cite{gjmss}, we say that $M$ has type $D$ if it is generated by a single axis $a$ and two  double axes $b+c$ and $d+e$. The diagram on the support set $\{a, b, c, d, e\}$ is the graph on this set in which two distinct points are connected by an edge when they are collinear in the Fischer space $\Gm$. It is proved in~\cite{gjmss} that the diagram is one of nine depicted in~Figure~\ref{Type D} (up to permutation of points).

\begin{figure}
\begin{center}
\includegraphics[scale=0.08]{
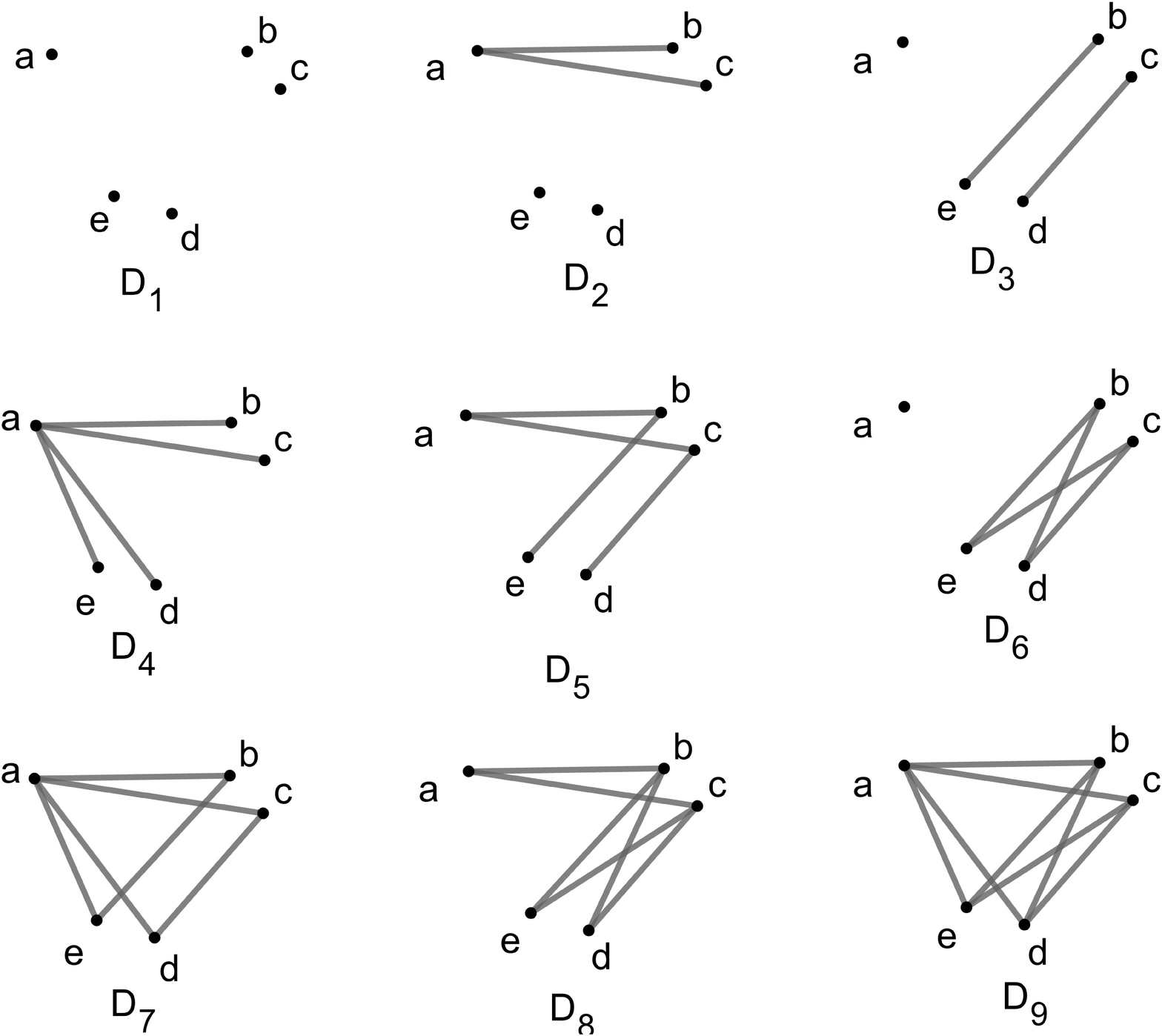}
\caption{Diagrams for algebras of $D$-type}\label{Type D}
\end{center}
\end{figure}

\subsection{Grading and the Miyamoto group}

It turns out that axial algebras of Monster and Jordan types are graded algebras.
\begin{definition} Suppose $\mathcal{F}=(X,\ast)$ is a fusion law. A grading of $\mathcal{F}$ by an abelian group $T$ is a partition $\{X_t~|~t\in T\}$ of the set $X$ (parts $X_t$ are allowed to be empty), such that, for all $\lambda, \mu\in X$, if $\lambda\in X_t$ and $\mu\in X_{t'}$ then $\lambda\ast\mu$ is contained in $X_{tt'}$.
\end{definition}

Observe that the fusion laws $\mathcal{M}(\alpha,\beta)$ and $J(\eta)$ are $C_2$-graded, where $C_2=\{1,-1\}$. When $\mathcal{F}=\mathcal{M}(\alpha,\beta)$ or $\mathcal{F}=\mathcal{J}(\eta)$, we have $X_1=\{1,0, \alpha\}$ and $X_{-1}=\{\beta\}$ in the first case and $X_1=\{1,0\}$ and $X_{-1}=\{\eta\}$ in the second case.

If $a$ is an $\mathcal{F}$-axis in an algebra $A$, where $\mathcal{F}$ is $C_2$-graded, then define a linear 
mapping $\tau_a:A\rightarrow A$ in the following way: $\tau_a|_{A_{X_1}(a)}=1_{A_{X_1}(a)}$ and 
$\tau_a|_{A_{X_{-1}}(a)}=-1_{A_{X_{-1}}(a)}$. The grading implies that $\tau_a$ is an involutive automorphism of $A$ known as the {\it Miyamoto involution}.
The {\it Miyamoto group} is the subgroup of $\Aut(A)$ generated by all Miyamoto involutions $\tau_a$, where $a$ runs over the set of $\mathcal{F}$-axes of $A$. 

We want to emphasize an important property of Miyaomoto involutions, which follows from the definition: if an $\mathcal{F}$-axis $a$ of an axial algebra $A$ is also an $\mathcal{F}$-axis in an axial subalgebra $B$, then $B$ is invariant under the action of $\tau_a$.

There are ways to associate an automorphism with an axis for an arbitrary abelian group $T$ (see \cite{gjmss}).

Consider a Matsuo algebra $M_\eta(G,D)$ and an axis $a\in D$. The action of the Miyamoto involution $\tau_a$ on $M_\eta(G,D)$ agrees with the conjugation group action of $a$ on $D$. Therefore, the Miyamoto group of $M_\eta(G,D)$ is isomorphic to the factor group $G/Z(G)$. The action of the Miyamoto involution $\tau_a$ can be interpreted in terms of the Fischer space $\Gamma=\Gamma(G,D)$. Namely, $\tau_a$ fixes $a$ and all points non-collinear with $a$ and it switches the two points other than $a$ on any line through $a$. 

Since the fusion law $\mathcal{M}(2\eta, \eta)$ is $C_2$-graded, each double axis $x$ in a Matsuo algebra defines the Miyamoto involution $\tau_x$. 

\begin{proposition}\cite[Proposition~4.10]{gjmss}
If $x=a+b$ is a double axis, then $\tau_x=\tau_a\tau_b$.
\end{proposition}


\subsection{Fixed and flip subalgebras}
In this subsection, we define fixed and flip subalgebras in Matsuo algebras. These concepts were  introduced in~\cite{gjmss}.

Let $(G,D)$ be a $3$-transposition group and $\Gm$ be its Fischer space. Consider a Matsuo algebra $M=M_\eta(\Gamma)$ over a field $\F$. The group $\Aut\Gamma$ acts on $M$ and it is the full group of 
automorphisms of $M$ preserving the basis $D$. 

\begin{definition} If $H\leq\Aut\Gamma$, then
the \emph{fixed subalgebra} of $H$ in $M$ is defined as
$$M_H:=\{u\in M\mid u^h=u\mbox{ for all }h\in H\}.$$
\end{definition}

Each fixed subalgebra $M_H$ has a natural basis.

\begin{proposition}\cite[Proposition~7.3]{gjmss}\label{fixed-basis}
Let $O_1,\ldots,O_k$ be all the $H$-orbits on $D$. Then the orbit vectors 
$u_i:=\sum_{c\in O_i}c$ form a basis of the fixed subalgebra $M_H$.
\end{proposition}

Since we are interested in single and double axes, we will focus on the case when the order of $H$ is two.

\begin{definition}
Consider an element $\tau\in\Aut\Gm$ with $|\tau|=2$. Denote $H=\la\tau\ra$. Then 
the $H$-orbits $O_1,\ldots,O_k$ have each length $1$ or $2$. Let us classify 
the orbit vectors $u_i$ into three groups:
\begin{itemize}
\item $u_i=a\in D$, corresponding to orbits $O_i=\{a\}$ of length $1$;
\item $u_i=a+b$, corresponding to orbits $O_i=\{a,b\}$ with $ab=0$ (orthogonal 
orbits);
\item $u_i=a+b$, with $O_i=\{a,b\}$ satisfying $ab\neq 0$ (non-orthogonal 
orbits). 
\end{itemize}

We call the $u_i$ of the first kind \emph{singles}, of the second kind 
\emph{doubles}, and of the third kind \emph{extras}. 

\end{definition}

Recall that double axes in $M$ are not primitive.
On the other hand, if a double axis lies in $M_H$, then it is primitive in $M_H$.
\begin{proposition}\cite[Prop.~7.4]{gjmss}
Each double $u_i=a+b$ is a double axis and it is 
primitive in $M_H$.  In fact, the set of doubles consists of all double axes that are contained in $M_H$ and are primitive in it.
\end{proposition}

Since singles and doubles are primitive axes in $M_H$, they generate axial subalgebras of Monster type $(2\eta,\eta)$ inside $M_H$. This leads to the following definition. 

\begin{definition} If $\tau$ is an involution in $\Aut\Gm$ and $H=\la\tau\ra$, then the subalgebra of $M_H$ generated by all singles and all doubles is called a flip subalgebra.
\end{definition} 

In general, a flip subalgebra can be smaller than the corresponding fixed subalgebra.

\section{Some 3-transposition groups and their Fischer spaces}\label{sec:3}
In this section, we describe Fischer spaces of some 3-transposition groups. First, we introduce 3-transposition groups that play an important role throughout the paper.

Recall that the spherical Coxeter groups with simply laced diagrams $A_n$ and  $D_n$ are examples of 3-transposition groups. They are denoted by $W(A_n)\simeq S_{n+1}$ and $W(D_n)\simeq 2^{n-1}:S_n$, respectively.
Denote by $\Lambda=\mathbb{Z}^n$ the root lattice. Then the affine groups $W(\tilde{A}_n)$ and $W(\tilde{D}_n)$ are the extensions of $\Lambda$ by the corresponding spherical group. Denote $W_p(\tilde{\Phi}_n)=W(\tilde{\Phi}_n)/p\Lambda$ for $p=2,3$ and $\Phi=A,D$. Then the resulting groups $W_2(\tilde{A}_n)\simeq2^n:S_{n+1}\simeq W(D_{n+1})$, $W_3(\tilde{A}_n)\simeq3^n:S_{n+1}$, and $W_p(\tilde{D}_n)\simeq p^n:2^{n-1}:S_n$ are 3-transposition groups \cite{cuypers_hall}. There is a common approach to all these groups.

Denote the base group of the restricted wreathed product $G= T\operatorname{wr}S_n$ by $B$, that is $B=T^n$.  The natural injection $\iota_i$ of $T$ as the $i$-th direct factor $T_i$ of $B$ is given by $\iota_i(t) =t_i$, where $1\leq i\leq n$. The projection $\pi_i$ of $B$ onto $T$
induced by the $i$-th factor is given by $\pi_i(b)=b(i)$. We identify $S_n$ with the complement to $B$ in
$G$ which acts naturally on the indices from $\{1,\ldots,n\}$. Let $Wr(T,n)$ be the subgroup $\langle d^G\rangle$ of $G$, where $d$ is a transposition of the complement to $B$. Note that in general situation $Wr(T,n)$ can be a proper subgroup of $G$. The following statement describes when $d^G$ is a class of 3-transpositions.
\begin{proposition}\cite[Theorem~6]{zara88},
\cite[Prop.~8.1]{hall93}\label{p:w(k,n)}
Suppose that $T$ is a group and $G=T\operatorname{wr}S_n$. Fix a transposition $d$ of $S_n$. Then $d^G$ is a class of 3-transpositions 
if and only if each element of $T$ has order 1, 2, or 3. 
\end{proposition}

Note that the groups $T$ with restrictions as in the proposition were classified in \cite{Neumann}. 

This wreathed product construction can be applied to the 3-transposition groups mentioned above: $W(D_n)\simeq Wr(2,n)$, $W_3(\tilde{A}_{n-1})\simeq Wr(3,n)$, $ W_3(\tilde{D}_n)\simeq Wr(S_3,n)$, and $W_2(\tilde{D}_n)$ has the same central type as $Wr(2^2,n)$~\cite{hs}. 
Therefore, it is useful to have a description of the corresponding class of 3-transpositions in the group $Wr(T,n)$.
The next two lemmas are well known, we have included their proofs for completeness and to show how we deal with points and lines in our notation.

\begin{lemma}\label{l:3trans}
Consider the restricted wreathed product  $G=T\operatorname{wr}{S}_n$ and a transposition $d\in S_n$. Then $d^G$ consists of elements $t_it^{-1}_j(i,j)$, where $t\in T$ and $1\leq i< j\leq n$.
\end{lemma}
\begin{proof}
Denote $B=T^n$. For each $i\in\{1,\ldots,n\}$, consider the group $T_i$ and maps $\iota_i:T\rightarrow T_i$, $\pi_i:B\rightarrow T_i$, as above.
An arbitrary element $g$ of $T\operatorname{wr}{S}_n$
can be written as $g=b\sigma$, where $b\in B$ and $\sigma\in S_n$. Then 
$$(1,2)^{g}=\sigma^{-1}b^{-1}(1,2)b\sigma=(b^{-1})^\sigma b^{(1,2)\sigma}(1,2)^{\sigma}.$$
Denote $i=1^\sigma$ and $j=2^\sigma$. If $k\neq1,2$
then $k^\sigma=k^{(1,2)\sigma}$ and hence $((b^{-1})^{\sigma})_{k^\sigma}$ and 
$(b^{(1,2)\sigma})_{k^\sigma}$ are elements of $T_{k^{\sigma}}$ whose product equals 1. Since  $1^{(1,2)\sigma}=j$ and $2^{(1,2)\sigma}=i$, we find that
 $(1,2)^g=(b(1)^{-1})_i(b(2)^{-1})_jb(2)_ib(1)_j(i,j)$.
Therefore, if $t=b(1)^{-1}b(2)$, then $(1,2)^g=t_it_j^{-1}(i,j)$. Clearly, for every $t\in T$ we can find $b\in B$ such that $t=b(1)^{-1}b(2)$ and hence $(1,2)^G=\{t_it_j^{-1}(i,j)~|~t\in T\}$.
\end{proof}

\noindent{\bf Notation.} We write $t.(i,j)$ for a 3-transposition element $t_it_j^{-1}(i,j)$ from Lemma~\ref{l:3trans}. Since $t.(i,j)=t^{-1}.(j,i)$, we will usually assume that $i<j$. 

\begin{lemma}\label{l:lines} Suppose that each element of $T$ has order 1, 2, or 3.
Then each line of the Fischer space of $Wr(T,n)$ coincides with one of the following sets.
\begin{enumerate}[(i)]
\item $\{t.(i,j), s.(j,k), ts.(i,k)\}$, where $s,t\in T$ and $1\leq i<j<k\leq n$;
\item $\{t.(i,j), s.(i,j), st^{-1}s.(i,j)\}$, where $s,t\in T$, $|st^{-1}|=3$, and $1\leq i<j\leq n$.
\end{enumerate}
\end{lemma}
\begin{proof}
Denote $B=T^n$. Let $x=t.(i,j)$ and $y=s.(i',j')$ be collinear points. Note that order of $xy$ is divisible by the order of $(i,j)(i',j')$.
Consequently, if $|xy|=3$, then either $|(i,j)(i',j')|=3$ or $|(i,j)(i',j')|=1$.
In the first case, we find that $|\{i,j\}\cap\{i',j'\}|=1$. Then Proposition~\ref{p:w(k,n)} implies that $|xy|=3$. 
We can assume that $i'=j$, $j'=k$. Now we find that
\begin{multline*}
x^y=(j,k)s_ks_j^{-1}t_it^{-1}_j(i,j)s_js^{-1}_k(j,k)\\=(s^{-1}t^{-1})_ks_jt_i(j,k)(i,j)s_js^{-1}_k(j,k)=(s^{-1}t^{-1})_k(ts)_i(k,i)=ts.(i,k).
\end{multline*}

Suppose that $|(i,j)(i',j')|=1$. Then $i_1=i$ and $j_1=j$. Hence $xy=(ts^{-1})_i(st^{-1})_j$. Clearly, $|xy|=|ts^{-1}|$. Then
$$x^y=(i,j)s_i^{-1}s_jt_it^{-1}_j(i,j)s_is^{-1}_j(i,j)\\=(st^{-1}s)_i(s^{-1}ts^{-1})_j(i,j),$$
as claimed.
\end{proof}

\subsection{3-transposition groups $W_p(\tilde{A}_{n-1})$ and their Fischer spaces}\label{Wp(An)}

In this subsection, we focus on 3-transposition groups $W(A_{n-1})$ and $W_p(\tilde{A}_{n-1})$, where $p=2,3$. We discuss a description of the Fischer spaces of these groups that is convenient for further calculations. As was mentioned above, these groups can be considered as $Wr(p,n)$, where $p\in\{1,2,3\}$.

Let $n$ be an integer and $n\geq3$. 
For $p\in\{2,3\}$ consider the $n$-dimensional permutational module $V$ of $S_n$ over $\F_p$. Let $e_i$, 
$i\in\{1,\ldots,n\}$, be a basis of $V$ permuted by $S_n$.
Then the natural semi-direct product $V\rtimes S_n$ is isomorphic to $p\operatorname{wr}S_n$. Denote the $(n-1)$-dimensional `sum-zero' submodule of $V$ by $U$. 
Then $Wr(p,n)$ is isomorphic to the natural semidirect product $U\rtimes S_n$. Note that, for $p=2$ and even $n$, $U$ contains a $1$-dimensional `all-one' submodule $D$, which is the center of $Wr(2,n)$. When $p=3$, $U$ is irreducible. In both cases, $U$ is the unique minimal non-central normal subgroup of $Wr(p,n)$ and $Wr(p,n)/U\cong S_n$. Since $S_n$ does not have proper factor groups containing commuting involutions, we conclude that, up to the center, groups $Wr(p,n)$ have no other factors that are $3$-transposition groups. Now we describe the Fischer spaces of these groups.

If $p=1$, then the Fischer space of $W(A_{n-1})=S_n$ consists of 
$n(n-1)/2$ points corresponding to transpositions $b_{i,j}=(i,j)$,
where each line is of shape $\{b_{i,j},b_{i,k},b_{j,k}\}$ with $1\leq i<j<k\leq n$.

Assume that $p=2$. It follows from Lemmas~\ref{l:3trans} and \ref{l:lines} that the Fischer space of $Wr(2,n)=U:S_n$ consists of 
$n(n-1)$ points: $b_{i,j}=(i,j)$ and $c_{i,j}=(e_i+e_j)(i,j)$, for $1\leq 
i<j\leq n$; and $n^2$ lines, where each `b' line 
$\{b_{i,j},b_{i,k},b_{j,k}\}$, $1\leq i<j<k\leq n$, is complemented by three 
`bc' lines $\{b_{i,j},c_{i,k},c_{j,k}\}$, $\{b_{i,k},c_{i,j},c_{j,k}\}$, 
and $\{b_{j,k},c_{i,j},c_{i,k}\}$. 

Assume that $p=3$. By Lemma~\ref{l:3trans}, for each pair $i$ and $j$ with $1\leq i<j\leq n$ we have three points: $b_{i,j}=(i,j)=b_{j,i}$, $c_{i,j}=(e_i-e_j)(i,j)$ and $c_{j,i}=(e_j-e_i)(i,j)$. Consequently, the Fischer space has $\frac{3n(n-1)}{2}$ points.
By Lemma~\ref{l:lines}, the lines are of several types. First, for each $1\leq i<j\leq n$, the triple (1) $\{b_{i,j},c_{i,j},c_{j,i}\}$ is a line. Secondly, for distinct $i$, $j$, and $k$ in $\{1,\ldots,n\}$, the triples (2) $\{b_{i,j},b_{i,k},b_{j,k}\}$, (3) $\{b_{i,j},c_{i,k},c_{j,k}\}$, (4) $\{b_{j,k},c_{i,j},c_{i,k}\}$, and (5) $\{c_{i,j},c_{j,k},c_{k,i}\}$ are lines.

In \cite{gjmss} authors introduced the following family of fixed subalgebras for the groups $Wr(2,n)$ with even $n$.
\begin{proposition}\cite[Prop.~9.1]{gjmss}\label{a:2qketa}
Suppose that $n=2k$ and $M=M(Wr(2,n))$. Consider an involution $\tau=b_{1,2}b_{3,4}\cdots b_{n-1,n}\in Wr(2,n)$, a permutation $\pi=(1,2)(3,4)\cdots(2k-1,2k)$, and $H=\langle\tau\rangle$. Then the fixed subalgebra $M_H$ coincides with the flip subalgebra $A(\tau)$ and has dimension $2k^2$. Among the orbit vectors, there are $2k$ singles: $b_{2i-1,2i}$ and $c_{2i-1,2i}$, where $1\leq i\leq k$; $2(k^2-k)$ doubles: $b_{2i-1,j}+b_{2i,j^\pi}$ and $c_{2i-1,j}+c_{2i,j^\pi}$, where $1\leq i\leq k-1$, $2i+1\leq j\leq 2k$; and no extras.
\end{proposition}

Following \cite{gjmss}, we denote this algebra by $2Q_k(\eta)$.
We also need the fixed subalgebra for the same flip in the Matsuo algebra of $Wr(3,n)$.

\begin{proposition}\label{basis}
Suppose that $n=2k$ and $M=M(Wr(3,n))$. Consider an involution $\tau=b_{1,2}b_{3,4}\cdots b_{n-1,n}\in Wr(2,n)$, a permutation $\pi=(1,2)(3,4)\cdots(2k-1,2k)$, and $H=\langle\tau\rangle$. Then the fixed subalgebra $M_H$ has dimension $3k^2-k$. Among the orbit vectors, there are $k$ singles: $b_{2i-1,2i}$, where $1\leq i\leq k$; $3(k^2-k)$ doubles: $b_{2i-1,j}+b_{2i,j^\pi}$, $c_{2i-1,j}+c_{2i,j^\pi}$, and $c_{j,2i-1}+c_{j^\pi, 2i}$, where $1\leq i\leq k-1$, $2i+1\leq j\leq 2k$; and $k$ extras $c_{2i-1,2i}+c_{2i,2i-1}$, where $1\leq i\leq k$.
\end{proposition}
\begin{proof} This is a consequence of the description of lines in the Fischer space of $Wr(3,n)$.
\end{proof}

\begin{question}
What is the dimension of the flip algebra $A(\tau)$ in Proposition~\ref{basis}?
\end{question}
\begin{remark}
In Subsection~\ref{sec:D7}, we find that if $k=2$, then $\dim A(\tau)=9=\dim M_H-1$.
\end{remark}

\subsection{3-transposition groups $W_p(\tilde{D}_{n})$ and their Fischer spaces}\label{Wp(Dn)}

In this subsection, we focus on the 3-transposition groups $W_p(\tilde{D}_{n})$, where $p=2,3$.
As was mentioned above, it is true that $W_3(\tilde{D}_{n})\simeq Wr(S_3,n)$ and $W_2(\tilde{D}_{n})$ has the same central type as $Wr(2^2,n)$.
Now we consider the groups $Wr(2^2,n)$ and $Wr(S_3,n)$ in detail.

Clearly, we can identify the base group of $2^2\operatorname{wr}S_n$
with $V_1\times V_2$, where $V_1$ and $V_2$ are vector spaces of dimension $n$ over a field of order $2$. Namely, $S_n$ permutes bases $e_1,\ldots,e_n$ and $f_1,\ldots,f_n$ of $V_1$ and $V_2$, respectively: $e_i^\pi=e_{i^\pi}$ and $f_i^\pi=f_{i^\pi}$ for any $\pi\in S_n$.
Hyperplanes $U_1=\langle e_i+e_{i+1}~|~1\leq i\leq n-1\rangle$ and 
$U_2=\langle f_i+f_{i+1}~|~1\leq i\leq n-1\rangle$ are invariant under this $S_n$-action. Therefore, the group $(U_1\times U_2)\rtimes S_n$ is a subgroup of $(V_1\times V_2)\rtimes S_n$. 

\begin{proposition}\label{p:W2(Dn)} The following statements hold.
\begin{enumerate}[(i)]
\item The conjugacy class of $(1,2)$ in $(V_1\times V_2)\rtimes S_n$
comprises of $2n(n-1)$ elements $b_{i,j}=b_{j,i}=(i,j)$, $c_{i,j}=c_{j,i}=(e_i+e_j)(i,j)$, $d_{i,j}=d_{j,i}=(f_i+f_j)(i,j)$, and $e_{i,j}=e_{j,i}=(e_i+e_j)(f_i+f_j)(i,j)$, where $1\leq i<j\leq n$.
\item The conjugacy class of $(1,2)$ in $(V_1\times V_2)\rtimes S_n$ generates $(U_1\times U_2)\rtimes S_n$ and satisfies the 3-transposition property; in particular, $(U_1\times U_2)\rtimes S_n\simeq Wr(2^2,n)$.
\item Each point of the Fischer space $\Gamma((U_1\times U_2)\rtimes S_n)$ lies exactly on $4(n-2)$ lines. The lines
are $\{b_{i,j}, b_{i,k}, b_{j,k}\}$, $\{b_{i,j}, c_{i,k}, c_{j,k}\}$, $\{b_{i,j}, d_{i,k}, d_{j,k}\}$,
$\{b_{i,j}, e_{i,k}, e_{j,k}\}$, $\{c_{i,j}, d_{i,k}, e_{j,k}\}$,
 where $i$, $j$, and $k$ are distinct elements of $\{1,\ldots,n\}$.
\end{enumerate}
\end{proposition}
\begin{proof} 
The base group of $2^2\operatorname{wr}S_n=(V_1\times V_2)\rtimes S_n$ is the direct product of groups $T_i=\{1,e_i,e_if_i, f_i\}$ for $1\leq i\leq n$. By Lemma~\ref{l:3trans}, there are four types of points for each pair $i<j$:
$b_{i,j}=(i,j)$, $c_{i,j}=(e_i+e_j)(i,j)$, $d_{i,j}=(f_i+f_j)(i,j)$, and $e_{i,j}=(e_i+e_j)(f_i+f_j)(i,j)$.
Since there are no elements of order three in the Klein four-group, Lemma~\ref{l:lines} implies that each line goes through two points $x_{i,j}$,
$y_{j,k}$, where $1\leq i<j<k\leq n$ and $x,y\in\{b,c,d,e\}$. A straightforward calculation shows that $b_{i,j}^{b_{i,k}}=b_{j,k}$, $b_{i,j}^{c_{i,k}}=c_{j,k}$, $b_{i,j}^{d_{i,k}}=d_{j,k}$, $b_{i,j}^{e_{i,k}}=e_{j,k}$, and $c_{i,j}^{d_{i,k}}=e_{j,k}$.
\end{proof}

Now we introduce a flip subalgebra in the Matsuo algebra of the group $(U_1\times U_2)\rtimes S_n$ for every even $n\geq 4$.

\begin{proposition}\label{p:tau:W2(Dn)} Suppose that $n=2k$, where $k\geq2$ is an integer.
Denote $G=(U_1\times U_2)\rtimes S_n$. 
Define a map $\tau_1:G\rightarrow G$
by rules $$b_{i,j}^{\tau_1}=b_{i,j}, c_{i,j}^{\tau_1}=d_{i,j}, d_{i,j}^{\tau_1}=c_{i,j}, e_{i,j}^{\tau_1}=e_{i,j},$$ 
where $1\leq i<j\leq n$. Let $\tau_2$ be an inner automorphism of $G$ corresponding to the element $\sigma=(1,2)(3,4)\ldots(2k-1,2k)\in S_n$. Denote $\tau=\tau_1\tau_2$. Then the following statements hold.
\begin{enumerate}[(i)]
\item $\tau_1$ and $\tau$ are involuntary automorphisms of $G$;
\item if $H=\langle\tau\rangle$, then the flip subalgebra $A=A(\tau)$ coincides with the fixed subalgebra $M_H$ and has dimension
$4k^2-k$. Among the orbit vectors, there are $2k$ singles
and $4k^2-3k$ doubles.
\end{enumerate}
\end{proposition}
\begin{proof}
Clearly, the group $\Aut(G)$ includes a subgroup isomorphic to $S_3$ whose elements fix $S_n$ pointwise and permute subgroups $U_1$, $U_2$, and $U_1U_2$. Since elements $b_{i,j}$, $c_{i,j}$, and $d_{i,j}$ generate $G$, we see that $\tau_1$ coincides with the automorphism that switches $U_1$ and $U_2$. By definition, we find that $b_{i,j}^{\tau_1\tau_2}=b_{i^\sigma,j^\sigma}=b_{i,j}^{\tau_2\tau_1}$, $c_{i,j}^{\tau_1\tau_2}=d_{i^\sigma,j^\sigma}=c_{i,j}^{\tau_2\tau_1}$, and $d_{i,j}^{\tau_1\tau_2}=c_{i^\sigma,j^\sigma}=d_{i,j}^{\tau_2\tau_1}$. Therefore, $\tau$ has order two and its orbits on $b_{1,2}^G$ of length one are 
$b_{2i-1,2i}$ and $e_{2i-1,2i}$, where $1\leq i\leq k$.
Since there are $8k(2k-1)$ elements in $b_{1,2}^G$, the dimension of $M_H$
equals $2k+\frac{4k(2k-1)-2k}{2}=2k+4k^2-3k=4k^2-k$.
If $x_{i,j}\in b_{1,2}^G$, where $x=b,c,d,e$, then $x_{i,j}^\tau=y_{i^\sigma,j^\sigma}$, where $y=b,d,c,e$, respectively.
Since the sets $\{i,j\}$ and $\{i^\sigma,j^\sigma\}$ are either equal or disjoint, the order of $x_{i,j}x_{i,j}^\tau$ cannot be equal to three.
Thus, there are no extras among orbit vectors in the basis of $M_H$ and hence $A(\tau)=M_H$.
\end{proof}

We now turn to the group $Wr(S_3,n)$.
The base group $B$ of $S_3\operatorname{wr}S_n$
is the direct product of groups $T_i=\langle f_i, e_i~|~f_i^3=s_i^2=(s_if_i)^2=1 \rangle$, where $1\leq i\leq n$. Moreover, we have $e_i^\pi=e_{i^\pi}$ and $f_i^\pi=f_{i^\pi}$ for all $\pi\in S_n$ and $i\in\{1,\ldots,n\}$.
Denote $V_1=\langle f_1,f_2,\ldots,f_n\rangle$ and 
$V_2=\langle e_1,e_2,\ldots,e_n\rangle$. Then $B=V_1\rtimes V_2$.
We consider $V_1$ and $V_2$ as vector spaces of dimension $n$ over fields of order 3 and 2, respectively. Denote $U=\langle e_i+e_{i+1}~|~1\leq i\leq n-1\rangle$, which is a hyperplane in $V_2$.
Since $U$ is an $S_n$-invariant subgroup in $S_3\operatorname{wr}S_n$, index of $V_1\rtimes(U\rtimes S_n)$ in $S_3\operatorname{wr}S_n$ equals two.

\begin{proposition}\label{p:tau:W3(Dn)} The following statements hold.
\begin{enumerate}[(i)]
    \item The semidirect product $G=V_1\rtimes(U\rtimes S_n)$
    is isomorphic to $Wr(S_3, n)$.
    
    \item Denote $b_{i,j}=(i,j)$, $c_{i,j}=(e_i+e_j)(i,j)$, $f_{i,j}=(-f_i-f_j)(e_i+e_j)(i,j)$, $e_{i,j}=(f_i+f_j)(e_i+e_j)(i,j)$, where $1\leq i<j\leq n$, and 
    $d_{i,j}=(f_i-f_j)(i,j)$,  where  
    $1\leq i\neq j\leq n$. Then $b_{i,j}$, $c_{i,j}$, $d_{i,j}$, $e_{i,j}$, and $f_{i,j}$ are all $3n(n-1)$ points of the Fischer space of $G$.
    
    \item There are $6(n-2)+1$ lines through each point of the Fischer space. More precisely, $x_{i,j}$ is collinear with $2n-4$ out of $\frac{n(n-1)}{2}$ points in each of the six groups of points $b_{k,l},c_{k,l},d_{k,l},d_{l,k},e_{k,l},f_{k,l}$, where  $1\leq k<l\leq n$, and sets $\{i,j\}$ and $\{k,l\}$ have exactly one common element, these lines are presented in Table~\ref{t:lines-Wr(S3,n)}. In addition, each point lies exactly on one of the following $n(n-1)$ lines: $\{b_{i,j},d_{i,j},d_{j,i}\}$ and $\{c_{i,j},e_{i,j},f_{i,j}\}$. 
\end{enumerate}
\end{proposition}
\begin{proof}
Fix generators $e$ and $f$ of $S_3$ such that $e^2=f^3=(ef)^2=1$.
By Lemma~\ref{l:3trans}, the conjugacy class of a transposition of $S_n$ in $S_3\operatorname{wr}S_n$
consists of elements $t_it_j^{-1}(i,j)$,
where $1\leq i<j\leq n$ and
$$(t_i,t_j)\in\{(1,1),(f_i,f_j),(-f_i,-f_j),(e_i,e_j), (f_ie_i,f_je_j),((-f_i)e_i,(-f_j)e_j)\}.$$
Since $(f_je_j)^{-1}=f_je_j$ and $((-f_j)e_j)^{-1}=(-f_j)e_j$,
corresponding elements are
$(i,j)=b_{i,j}$, $(f_i-f_j)(i,j)=d_{i,j}$, $(f_j-f_i)(i,j)=d_{j,i}$, $(e_i+e_j)(i,j)=c_{i,j}$, $(f_i+f_j)(e_i+e_j)(i,j)=e_{i,j}$, and $(-f_i-f_j)(e_i+e_j)(i,j)=f_{i,j}$.

By Lemma~\ref{l:lines}, two points $x_{i,j}$ and $y_{k,l}$ are collinear if the sets $\{i,j\}$ and $\{k,l\}$ have exactly one common element. Therefore, there are $6(2n-4)/2=6(n-2)$ distinct lines of such a type through a fixed point.
We apply Lemma~\ref{l:lines} to find the third point on the line through $x_{i,j}$ and $y_{j,k}$ with $1\leq i<j<k\leq n$. For example, take $c_{i,j}$ and $e_{i,j}$.
Since $c_{i,j}=e.(i,j)$ and $e_{j,k}=(fe).(j,k)$, Lemma~\ref{l:lines} implies that
$c_{i,j}^{e_{j,k}}=(efe).(i,k)=(f^{-1}).(i,k)=g_{i,k}$. Other cases are considered in a similar manner, we list the corresponding third points in Table~\ref{t:lines-Wr(S3,n)}.

Lines of the second type go through points $t.(i,j)$ and $s.(i,j)$, where $|ts^{-1}|=3$. If $t_i=1$, then $s_i=f_i$ or $s_i=-f_i$ and if $t_i=e_i$ then $s_i=e_if_i$ or $s_i=e_i(-f_i)$. Fixing $i$ and $j$, we obtain lines
$\{b_{i,j},d_{i,j},d_{j,i}\}$ and $\{c_{i,j},e_{i,j},f_{i,j}\}$, respectively.
\end{proof}

\begin{table}[h]
\begin{center}
\begin{tabular}{|c||c|c|c|c|c|c|c|}
\hline
 & $b_{i,j}$ & $c_{i,j}$ & $d_{i,j}$ & $g_{i,j}$ & $e_{i,j}$ & $f_{i,j}$\\
\hline\hline
$b_{j,k}$ & $b_{i,k}$ & $c_{i,k}$ & $d_{i,k}$ & $g_{i,k}$ & $e_{i,k}$ & $f_{i,k}$ \\
$c_{j,k}$ & $c_{i,k}$ & $b_{i,k}$ & $e_{i,k}$ & $f_{i,k}$ & $d_{i,k}$ & $g_{i,k}$ \\
$d_{j,k}$ & $d_{i,k}$ & $f_{i,k}$ & $g_{i,k}$ & $b_{i,k}$ & $c_{i,k}$ & $e_{i,k}$ \\
$g_{j,k}$ & $g_{i,k}$ & $e_{i,k}$ & $b_{i,k}$ & $d_{i,k}$ & $f_{i,k}$ & $c_{i,k}$ \\
$e_{j,k}$ & $e_{i,k}$ & $g_{i,k}$ & $f_{i,k}$ & $c_{i,k}$ & $b_{i,k}$ & $d_{i,k}$ \\
$f_{j,k}$ & $f_{i,k}$ & $d_{i,k}$ & $c_{i,k}$ & $e_{i,k}$ & $g_{i,k}$ & $b_{i,k}$ \\
\hline
\end{tabular}
\end{center}
\caption{Lines in $\Gamma(S_3\operatorname{wr}S_n)$ with $1\leq i<j<k\leq n$}\label{t:lines-Wr(S3,n)}
\end{table}

\subsection{3-transposition groups $Wr(A_4,n)$ and their Fischer spaces}\label{s:Wr(A_4,n)}

In this subsection, we focus on the 3-transposition groups $Wr(A_4,n)$, where $n\geq2$.
Recall that $Wr(A_4,n)$ is the subgroup of the wreath product $A_4\operatorname{wr}S_n$ generated by the conjugacy class of a transposition.

\begin{lemma}\label{l:Wr(A4,n)} 
The Fischer space of $Wr(A_4,n)$ comprises of $6n(n-1)$ points.
Each of these points lies on exactly $12n-20$ lines.
\end{lemma}
\begin{proof}
It follows from Lemma~\ref{l:3trans} that each point of the Fischer space
is of shape $t.(i,j)$, where $t\in A_4$ and $1\leq i<j\leq n$. Therefore, the number of points equals $12n(n-1)/2=6n(n-1)$. Fix a point $d=(1,2)$.
By Lemma~\ref{l:lines}, each line through $d$
is either $\{d, s.(2,k), s.(1,k)\}$, where $2<k\leq n$ and $s\in A_4$,
or $\{d, s.(1,2), s^2.(1,2)\}$, where $s\in A_4$ and the order of $s$ equals 3.
Consequently, there are $12(n-2)+4=12n-20$ lines that contain $d$.
\end{proof}

Suppose that $n=2k$, where $k$ is an integer. Consider permutations $\pi=(1,2)(3,4)\ldots(2k-1,2k)\in S_n$
and $\sigma=(1,2)(3,4)\in A_4$. Then the inner automorphism $\tau=(\sigma,\sigma,\ldots,\sigma).\pi$ of $Wr(A_4,n)$ has order 2. For $i\in\{1,\ldots,2n\}$ denote $\overline{i}=i^\pi$.
Let $H=\langle\tau\rangle$.
\begin{proposition}\label{p:fixed:Wr(A_4,n)} If $n=2k$, then the fixed subalgebra $M_H$ is of dimension $12k^2-4k$. Among the orbit vectors, there are 
\begin{itemize}
\item $4k$ singles $t.(2i-1, 2i)$, where $t\in K_4$ and $1\leq i\leq k$;
\item $12k(k-1)$ doubles $t.(i,j)+t^\pi.(\overline{i},\overline{j})$, where $t\in A_4$ and $\overline{i}\neq j$;
\item $4k$ extras $t.(2i-1,2i)+t^\pi.(2i-1,2i)$, where $t$ is a 3-cycle and $1\leq i\leq k$.
\end{itemize}
\end{proposition}
\begin{proof}
Consider a point $t.(i,j)\in\Gamma(Wr(A_4,n))$, where $1\leq i<j\leq n$ and $t\in A_4$.
It is easy to see that $(t.(i,j))^\tau=t^\pi.(\overline{i}, \overline{j})$.
Therefore, $t.(i,j)$ is fixed by $\tau$ if and only if $i=2l-1$, $j=2l$ with $1\leq l\leq k$ and $t\in K_4=\{(), (1,2)(3,4), (1,3)(2,4), (1,4)(2,3) \}$. Hence, the number of singles equals $4k$. Now we find the number of extras. Suppose that an orbit vector $t.(i,j)+t^\pi.(\overline{i},\overline{j})$ is an extra. If $\overline{i}\neq j$ then the pairs $\{i,j\}$ and $\{\overline{i},\overline{j}\}$ are disjoint and hence  
$t.(i,j)$ and $t^\pi.(\overline{i},\overline{j})$ are not colinear. If $\overline{i}=j$ then 
$\overline{j}=i$. We know that $t.(i,j)$ and $t^{\pi}.(j,i)=(t^{-1})^\pi.(i,j)$ are colinear if and only if $tt^{\pi}$ has order 3 which is equivalent to $t$ is a 3-cycle in $A_4$.
\end{proof}
\begin{question}
What is the dimension of the flip algebra $A(\tau)$ in Propositon~\ref{p:fixed:Wr(A_4,n)}?
\end{question}

Consider now another inovolutive automorphism of $Wr(A_4,n)$.
Suppose  that $n=2k$,  where $k$ is a positive integer, and $\pi=(1,2)(3,4)\ldots(2k-1,2k)$ as above. 
Denote $$\tau=((3,4),(3,4),\ldots,(3,4)).\pi$$ is an element of order 2 in $S_4\operatorname{wr}S_n$.
If $\pi_1\in A_4$ and $\pi_2$ is a transposition in $S_n$, then for a 3-tranposition $\pi_1.\pi_2$ of $Wr(A_4,n)$ we find that $(\pi_1.\pi_2)^\tau=\pi_1^{(3,4)}.\pi_2^{\pi}$.
Therefore, $\tau$ induces an automorphism of the Fischer space of $Wr(A_4,n)$.
Denote $H=\langle\tau \rangle$.
\begin{proposition}\label{p:fixed2:Wr(A_4,n)}
If $n=2k$, then the fixed subalgebra $M_H$
coincides with the flip subalgebra $A(\tau)$ and has dimension $12k^2-3k$. Among the $\tau$-orbit vectors, 
there are $6k$ singles $t.(2i-1, 2i)$, where $$t\in\{(),(1,2)(3,4), (1,3,4), (2,3,4), (1,4,3), (2,4,3)\}$$ and $1\leq i\leq k$, and the remaining $12k^2-9k$ vectors are doubles.
\end{proposition}
\begin{proof} 
Suppose that $\pi_1.\pi_2$ is a point of the Fischer space of $Wr(A_4,n)$.
If $(\pi_1.\pi_2)^\tau=\pi_1.\pi_2$, then $\pi_2=(2i-1,2i)$ for some $i\in\{1,\ldots, k\}$.
Now $(\pi_1.(2i-1,2i)^\tau=\pi^{(3,4)}.(2i,2i-1)=(\pi_1^{-1})^{(3,4)}.(2i-1,2i)$.
It is easy to see that $\pi_1^{-1})^{(3,4)}=\pi_1$ if and only if 
$\pi_1\in\{(),(1,2)(3,4), (1,3,4), (2,3,4), (1,4,3), (2,4,3)\}$. So $M_H$ includes exactly $6k$ singles.

Now we show that all other orbit vectors are doubles. If $\pi_2\neq(2i-1,2i)$ for some
 $i\in\{1,\ldots, k\}$, then $\pi_2$ and $\pi_2^\pi$ are distinct commuting transpositions and hence
 the orbit vector $\pi_1.\pi_2+\pi_1^{(3,4)}.\pi_2^\pi$ is a double. Assume that $\pi_2=(2i-1,2i)$.
 Since the orbit vector is not a single, we have $\pi_1\in\{(1,2,4), (1,4,2), (1,2,3), (1,3,2), (1,3)(2,4), (1,4)(2,3)\}$. The corresponding orbit vectors are $(1,2,4).\pi_2+(1,3,2).\pi_2$, 
 $(1,2,3).\pi_2+(1,4,2).\pi_2$, and $(1,3)(2,4).\pi_2+(1,4)(2,3).\pi_2$.
 Since $(1,2,4)(1,3,2)^{-1}=(1,3)(2,4)$ and $(1,2,3)(1,4,2)^{-1}=(1,4)(2,3)$, Lemma~\ref{l:lines} implies that these orbit vectors are doubles. The Fischer space comprises $12\cdot k(2k-1)$ points, 
 so there are $\frac{1}{2}(24k^2-12k-6k)=12k^2-9k$ doubles.
\end{proof}

\subsection{3-transposition groups $Wr(3^{1+2},n)$ and their Fischer spaces}\label{s:Wr(3^1+2,n)}

In this subsection, we focus on the 3-transposition groups $Wr(3^{1+2},n)$, where $n\geq2$.
Fix generators $u$ and $v$ of the extraspecial group $P$ of order 27 and exponent 3, so that $w=[u,v]=u^{-1}v^{-1}uv$ has order three and commutes with $u$ and $v$.
Then each element of $P$ has a unique presentation in the form $u^{r}v^{s}w^t$ where $0\leq r,s,t\leq2$.
So the Fischer space has $27\frac{n(n-1)}{2}$ points of the shape $u^{r}v^{s}w^t.(i,j)$ in the terms of Lemma~\ref{l:lines}.
Suppose that $n=2k$ and $\pi=(1,2)(3,4)\ldots(2k-1,2k)$.
The group $P$ has an involutive automorphism $\sigma$ which $u$ sends to $v$ and vice versa. Using $\pi$ and $\sigma$, define an automorphsim $\tau$ as follows: 
$$((u^{r}v^sw^t).(i,j))^\tau:=((u^{r}v^sw^t)^\sigma.(i,j)^\pi=v^ru^sw^{-t}.(i^\pi,j^\pi)=u^sv^rw^{-t-sr}.(i^\pi,j^\pi).$$
For $i\in\{1,\ldots,2n\}$ denote $\overline{i}=i^\pi$.
Consider the fixed subalgebra $M_H$, where $H=\langle\tau\rangle$.
\begin{proposition}\label{p:fixed:Wr(3^1+2,n)} 
If $n=2k$, then the fixed subalgebra $M_H$ is of dimension $27k^2-9k$. Among the orbit vectors, there are 
\begin{itemize}
\item $9k$ singles $u^rv^sw^t.(2i-1, 2i)$, where $0\leq t\leq2$ and $r+s\equiv0\pmod{3}$;
\item $27k(k-1)$ doubles $g.(i,j)+g^\tau.(\overline{i},\overline{j})$, where $g\in 3^{1+2}$ and $\overline{i}\neq j$;
\item $9k$ extras $u^rv^sw^t.(2i-1,2i)+u^{2s}v^{2r}w^t.(2i-1,2i)$, where 
$r+s\not\equiv3\pmod{3}$ and $1\leq i\leq k$.
\end{itemize}
\end{proposition}
\begin{proof}
Consider a point $x=u^rv^sw^t.(i,j)$ of the Fischer space.
Since $x^\tau=u^sv^rw^{-t-sr}.(\overline{i},\overline{j})$,
we equality $x^\tau=x$ implies that $\overline{i}=j$.
Now $$u^rv^sw^t.(i,\overline{i})^\tau=u^sv^rw^{-t-sr}.(\overline{i},i)=(u^sv^rw^{-t-sr})^{-1}.(i,\overline{i})=u^{2s}v^{2r}w^t.(i,\overline{i}).$$ 
Therefore, $x$ is a single in $M_H$ if and only if 
$\overline{i}=j$ and $r+s\equiv0\pmod{3}$.
Observe that there $k$ ways to choose the transposition $(i,j)$
such that $j=\overline{i}$ and $9$ ways to choose the triple $(r,s,k)$ with $r+s\equiv3\pmod3$ and $0\leq k\leq 2$.
So there are 9k singles in $M_H$.
If $\overline{i}=j$ and $r+s\not\equiv0\pmod{3}$, then
we obtain an extra $u^rv^sw^t.(i,j)+u^{2s}v^{2r}w^t.(i,j)$.
Under this assumption, there are $k$ ways to choose the transposition $(i,j)$
and $18$ ways for the triple $(r,s,t)$, so we find $9k$ extras.
Finally, if $\overline{i}\neq j$, then $x+x^{\tau}$
is a double. There are 27 ways to choose the triple $(r,s,t)$ and $k(2k-1)-k=k(2k-2)$ ways to choose the transposition $(i,j)$. So we get $27k(k-1)$ doubles.
As a result, we infer that the dimension of $M_H$ equals
$9k+9k+27k(k-1)=27k^2-9k$.
\end{proof}
\begin{question}
What is the dimension of the flip subalgebra $A(\tau)$ in Proposition~\ref{p:fixed:Wr(3^1+2,n)}?
\end{question}
\begin{remark}
In Subsection~\ref{sec:D7}, we prove that if $k=2$ and $\eta\neq2$, then $A(\tau)=M_H$, and
if $k=2$ and $\eta=2$, then $\dim A(\tau)=\dim M_H-1$.
\end{remark}

\subsection{3-transposition groups $Wr(3^2,n)$ and their Fischer spaces}\label{s:Wr(3^2,n)}
In this subsection, we focus on the 3-transposition groups $Wr(3^2,n)$, where $n\geq2$.
Fix generators $u$ and $v$ of the elementary abelian group $P$ of order 9, 
in particular we have $[u,v]=1$ and $|u|=|v|=3$.
Each element of $P$ has a unique presentation in the form $u^{r}v^{s}$ where $0\leq r,s\leq2$.
So the Fischer space has $9\frac{n(n-1)}{2}$ points of the shape $u^{r}v^{s}.(i,j)$ in the terms of Lemma~\ref{l:lines}.
Suppose that $n=2k$ and $\pi=(1,2)(3,4)\ldots(2k-1,2k)$.
The group $P$ has an involutive automorphism $\sigma$ which $u$ sends to $v$ and vice versa. Using $\pi$ and $\sigma$, define an automorphsim $\tau$ as follows: 
$$((u^{r}v^s).(i,j))^\tau:=((u^{r}v^s)^\sigma.(i,j)^\pi=v^ru^s.(i^\pi,j^\pi)=u^sv^r.(i^\pi,j^\pi).$$
For $i\in\{1,\ldots,2n\}$ denote $\overline{i}=i^\pi$.
Consider the fixed subalgebra $M_H$, where $H=\langle\tau\rangle$.
\begin{proposition}\label{p:fixed:Wr(3^2,n)} 
If $n=2k$, then the fixed subalgebra $M_H$ is of dimension $9k^2-3k$. Among the orbit vectors, there are 
\begin{itemize}
\item $3k$ singles $u^rv^s.(2i-1, 2i)$, where $0\leq t\leq2$ and $r+s\equiv0\pmod{3}$;
\item $9k(k-1)$ doubles $t.(i,j)+t^\sigma.(\overline{i},\overline{j})$, where $t\in P$ and $\overline{i}\neq j$;
\item $3k$ extras $u^rv^s.(2i-1,2i)+u^{2s}v^{2r}.(2i-1,2i)$, where 
$r+s\not\equiv3\pmod{3}$ and $1\leq i\leq k$.
\end{itemize}
\end{proposition}
\begin{proof}
Consider a point $x=u^rv^s.(i,j)$ in the Fischer space.
Since $x^\tau=u^sv^r.(\overline{i},\overline{j})$,
we equality $x^\tau=x$ implies that $\overline{i}=j$.
Now $u^rv^s.(i,\overline{i})^\tau=u^sv^r.(\overline{i},i)=(u^sv^r)^{-1}.(i,\overline{i})=u^{2s}v^{2r}.(i,\overline{i})$. Therefore, $x$ is a single in $M_H$ if and only if 
$\overline{i}=j$ and $r+s\equiv0\pmod{3}$.
Observe that there are $k$ ways to choose the transposition $(i,j)$
such that $j=\overline{i}$ and $3$ ways to choose the pair $(r,s)$ with $r+s\equiv3\pmod3$. So there are $3k$ singles in $M_H$.
If $\overline{i}=j$ and $r+s\not\equiv0\pmod{3}$ then
we obtain an extra $u^rv^s.(i,j)+u^{2s}v^{2r}.(i,j)$.
Under this assumption, there are $k$ ways to choose the transposition $(i,j)$
and $6$ ways for the pair $(r,s)$, so we find $9k$ extras.
Finally, if $\overline{i}\neq j$ then $x+x^{\tau}$
is a double. There are nine ways to choose the pair $(r,s)$ and $k(2k-1)-k=k(2k-2)$ ways to choose the transposition $(i,j)$. So we get $9k(k-1)$ doubles.
As a result, we infer that the dimension of $M_H$ equals
$3k+3k+9k(k-1)=9k^2-3k$.
\end{proof}

\begin{question}
What is the dimension of the flip subalgebra $A(\tau)$ in Proposition~\ref{p:fixed:Wr(3^2,n)}?
\end{question}
\begin{remark}
In Subsection~\ref{sec:D7}, we prove that if $k=2$ and $\eta\neq2$, then $A(\tau)=M_H$, and
if $k=2$ and $\eta=2$, then $\dim A(\tau)=\dim M_H-1$.
\end{remark}

\section{Primitive 3-generated subalgebras for $D$-diagrams}\label{s:main}

In this section, we classify groups and primitive axial algebras corresponding to the first seven $D$-diagrams in Figure~\ref{Type D}.  Throughout, we assume that $M$ is a Matsuo algebra and $A$ is its subalgebra generated by an axis $x$ and two double axes $y$ and $z$. Recall that $x$ and $y$ are not primitive in $M$,
however, we assume that they are primitive in $A$. By definition of single and double axes, there exist five distinct  
3-transpositions of $M$ such that $x=a$, $y=b+c$, and $z=d+e$,
in particular $|bc|=|de|=2$. Denote $G=\langle a,b,c,d,e\rangle$ and note that $G$ belongs to a short list of 5-generated 3-transposition groups which were classified in~\cite{hs}. Note that 
3-transposition groups with small number of generators are also investigated in~\cite{sozutov}.
Note that recently, 3-transposition groups generated by at most six transpositions have been classified~\cite{hall22}, which opens the way to the classification of 3-generated subalgebras of type $E$ mentioned in Introduction.

The diagram on the set $Z=\{a,b,c,d,e\}$ represents a Coxeter group $\hat G$ given by the presentation encoded in the diagram. Clearly, $G$ is a factor group of $\hat G$ and, in 
many cases, $G=\hat G$. These are easy cases, and in the remaining 
cases we will have to find additional relations identifying $G$ as a 
factor group of the Coxeter group $\hat G$. These relations come from the 3-transposition property and the ones used in~\cite{hs}.

It follows form \cite{kms} that if the diagram on $Z$ is connected, then the radical of the Frobenius form of $A$ includes all of its ideals. This implies that if the radical is zero, then $A$ is simple. Therefore, for all algebras corresponding to connected diagrams, we find critical values for $\eta$ to determine when obtained algebras are simple.

First, we consider disconnected diagrams $D_1$, $D_2$, and $D_3$.
Then other cases are treated separately.
In what follows, we use double angular brackets $\langle\!\langle~
\rangle\!\rangle$ to indicate subalgebra generation, leaving single brackets for the linear span.

\subsection{Disconnected diagrams $D_1$, $D_2$, $D_3$, and $D_6$}

In this subsection, we deal with disconnected diagrams on $Z$. These are cases $D_1$, $D_2$, and $D_3$. Note that for these types,
the algebra $A$ is a direct sum of 1-generated and 2-generated algebras. Recall that 2-generated primitive algebras were classified in~\cite{gjmss,Joshi-phd}. Therefore, in all cases we obtain a sum of the 1-dimensional algebra $\F$ plus one of ``known'' algebras.

\bigskip\noindent
{\bf Diagram $\mathbf{D_1}$:} Clearly, $\hat G$ is an elementary 
abelian group of order $2^5$ and $G$ is either $\hat G$ or a factor group of $\hat G$ of order divisible by $2^3$, in which the images of the five generators of $\hat G$ remain 
distinct. Furthermore, we find that $xy=0$, $xz=0$, and $yz=0$. This implies that $A$ is 
isomorphic to $\F\oplus\F\oplus\F=\F^3$.

\bigskip\noindent
{\bf Diagram $\mathbf{D_2}$:} In this case, $\hat G=\la a,b,c\ra\times\la 
d\ra\times\la e\ra \cong S_4\times 2\times 2$. Since $Z(S_4)=1$, we infer that $G$ can only be the full $\hat G$. The algebra $A$ decomposes as $\dla x,y\dra\oplus\dla z\dra\cong 
Q_2(\eta)\oplus\F$, where $Q_2(\eta)$ is the 4-dimensional algebra generated by $x$ and $y$ and corresponding to the diagram $A_3$ in~\cite{gjmss}. 

\bigskip\noindent
{\bf Diagram $\mathbf{D_3}$:} In this case, $\hat G\cong\la a\ra\times\la b,e\ra\times 
\la c,d\ra\cong 2\times S_3\times S_3$. It is easy to see that $G=\hat G$, 
as it cannot be any of the proper factor groups. Since $xy=0$ and $xz=0$, the 
algebra $A=\dla x,y,z\dra$ decomposes as the direct sum $\dla x\dra\oplus 
\dla y, z\dra\cong\F\oplus3C(\eta)$, where $3C(\eta)$ is
the 3-dimensional algebra generated by $y$ and $z$ and corresponding to the diagram $B_4$ in~\cite{gjmss}. Note that 
$3C(\eta)$ is also isomorphic to the Matsuo algebra of $S_3$.

\bigskip\noindent
{\bf Diagram $\mathbf{D_6}$:} 
Clearly, $a$ is in the center of $\hat G$. Other elements form subdiagram  of type $B_6$ in terms of 2-generated subalgebras (see Subsection 5.4 in \cite{gjmss}). 
Therefore, $G=G(p) = \langle a \rangle \times \hat H / \langle (b^dc^e)^p=1 \rangle $, where $\hat{H}$ is the Coxeter group defined by the subdiagram with support set $\{b,c,d,e \}$ and $p \in \{1,2,3\}$.

Note that $x$ cannot lie in the subalgebra $\dla y,z\dra$, since $a$ and $b$ are not conjugated in $G$. 
Therefore, $A$ decomposes as $\dla x\dra\oplus\dla y,z\dra$. Using classification of algebras for type $B_6$, we find that $A=\mathbb{F}\oplus A_p$, where $p\in\{1,2,3\}$ and $A_1$ is the 3-dimensional Matsuo algebra $3C(2\eta)$,
$A_2$ is the 5-dimensional algebra \cite[Table 7]{gjmss},
$A_3$ is the 8-dimensional algebra \cite[Table 8]{gjmss}.

\subsection{Diagram $D_4$}
\bigskip\noindent
In this subsection, we assume that the diagram on $Z$ is $D_4$, that is $a$ does not commute with $b$, $c$, $d$, and $e$, while these four involutions are pairwise commuting.

Since $G$ is a 3-transposition group, we have $|a^{bc}a^{de}|\in\{1,2,3\}$.
We consider each of these possibilities.

{\bf Case~$1$.} Suppose that $a^{bc}a^{de}=1$ and hence $a^{bc}=a^{de}$. The involutions $d$, $a$, and, $e$ generate a subgroup isomorphic to $S_4$. We can identify them with transpositions $(1,2)$, $(2,3)$, and $(3,4)$, respectively. Then we see that $a^{de}=(1,4)=e^{ad}$. Therefore, we have $e=a^{bcda}\in H=\langle a,b,c,d\rangle$. Consider the diagram on $a$, $b$, $c$, $d$. We see that it  
coincides with the Coxeter diagram $D_4$ and hence $H$ is a homomorphic image of the Weyl group 
$\hat{H}=W(D_4)\cong W_2(\tilde{A}_3)\cong 2^3:S_4$ of order 192. It has been proven in detail in \cite{gjmss} that $H=\hat H$ or 
$\hat H/Z(\hat H)$, where this group has arisen for the diagram $C_4$. Thus, we may assume that $H=\hat H$. We use the description of the Fischer space of $H$ from Subsection~\ref{Wp(An)}.
Recall that this Fischer space consists of $2\cdot 
6=12$ points: $b_{i,j}=(i,j)$ and $c_{i,j}=(e_i+e_j)(i,j)$, for $1\leq i<j\leq 
4$; and $4\cdot 4=16$ lines $\{b_{i,j},b_{i,k},b_{j,k}\}$, 
$\{b_{i,j},c_{i,k},c_{j,k}\}$, $\{b_{i,k},c_{i,j},c_{j,k}\}$, 
and $\{b_{j,k},c_{i,j},c_{i,k}\}$, for $1\leq i<j<k\leq 4$.

We can identify $a$, $b$, $c$, and, $d$ with $b_{1,2}=(1,2)$, $b_{1,3}=(1,3)$, $b_{2,4}=(2,4)$, and $c_{1,3}=(e_1+e_3)(1,3)$, respectively.
Then $e=a^{bcda}=c_{2,4}=(e_2+e_4)(2,4)$ and
$a^{bc}=b_{3,4}=a^{de}$. Therefore, our identification extends to an isomorphism of the groups $G$ and $W(D_4)$.

Recall that $A=\dla x,y,z\dra$ is invariant under $\tau_x=\tau_a$, 
$\tau_y=\tau_b\tau_c$, and $\tau_z=\tau_d\tau_e$.
Now we see that $A=\dla x,y,z\dra$ contains the following elements:
\begin{center}
$\begin{aligned}
s_1:=&\,x=a=b_{1,2}, \quad s_2:=\,x^{\tau_y}=x^{\tau_z}=b_{3,4},\\
d_1:=&\,y=b_{1,3}+b_{2,4}, \quad d_2:=\,z=c_{1,3}+c_{2,4},\\
d_3:=&\,y^{\tau_x}=b_{2,3}+b_{1,4}, \quad d_4:=\,z^{\tau_x}=c_{2,3}+c_{1,4},\\
d_5:=&\,(d_2+d_3-\frac{1}{\eta}d_2d_3)=\\
    &\,d_2+d_3-\frac{1}{\eta}(\eta(d_2+d_3)-\eta(c_{1,2}+c_{3,4}))=c_{1,2}+c_{3,4}.
\end{aligned}$
\end{center}
Observe that $s_1$ and $s_2$ are single axes and 
$d_1,\ldots,d_5$ are double axes. We claim that 
these seven elements form a basis of $A$. The supports of the axes are disjoint and so they are linearly independent. It remains to verify that the subspace spanned by them is closed under the algebra product. We use the description of lines above and utilize the substantial symmetry we have. This is a sample calculation involving double axes:
\begin{multline*}
d_1d_3=(b_{1,3}+b_{2,4})(b_{2,3}+b_{1,4})=
\frac{\eta}{2}(b_{1,3}+b_{2,3}-b_{1,2})+\frac{\eta}{2}(b_{1,3}+b_{1,4}-b_{3,4})\\+
\frac{\eta}{2}(b_{2,4}+b_{2,3}-b_{3,4})+\frac{\eta}{2}(b_{2,4}+b_{1,4}-b_{1,2})=\eta(d_1+d_3-s_1-s_2).    
\end{multline*}

The complete table of products is presented in Table~\ref{t:algebra D4-1}. 

\begin{table}[ht]
\begin{center}
\begingroup
\setlength{\tabcolsep}{6pt}
\scalebox{.60}{
$\begin{tabu}[ht]{|c||c|c|c|c|c|c|c|}
\hline
 & s_1 & s_2 & d_1 & d_2 & d_3 & d_4 & d_5 \\ \hline
s_1 & s_1 & 0 & \frac{\eta}{2}(2s_1+d_1-d_3) & \frac{\eta}{2}(2s_1+d_2-d_4) & \frac{\eta}{2}(2s_1-d_1+d_3) & \frac{\eta}{2}(2s_1-d_2+d_4) & 0 \\ \hline
s_2 & 0 & s_2 & \frac{\eta}{2}(2s_2+d_1-d_3) & \frac{\eta}{2}(2s_2+d_2-d_4)  & \frac{\eta}{2}(2s_2-d_1+d_3) & \frac{\eta}{2}(2s_2-d_2+d_4)  & 0  \\ \hline
d_1 & \frac{\eta}{2}(2s_1+d_1-d_3)   & \frac{\eta}{2}(2s_2+d_1-d_3)  & d_1  & 0  & \eta(d_1+d_3-s_1-s_2) & \eta(d_1+d_4-d_5) & \eta(d_1-d_4+d_5) \\ \hline
d_2 & \frac{\eta}{2}(2s_1+d_2-d_4)  & \frac{\eta}{2}(2s_2+d_2-d_4) & 0   & d_2  & \eta(d_2+d_3-d_5)   & \eta(d_2+d_4-s_1-s_2)  & \eta(d_2-d_3+d_5)  \\ \hline
d_3 & \frac{\eta}{2}(2s_1-d_1+d_3)  & \frac{\eta}{2}(2s_2-d_1+d_3) & \eta(d_1+d_3-s_1-s_2)  & \eta(d_2+d_3-d_5)   & d_3  & 0 & \eta(d_3+d_5-d_2)  \\ \hline
d_4 & \frac{\eta}{2}(2s_1-d_2+d_4) & \frac{\eta}{2}(2s_2-d_2+d_4) & \eta(d_1+d_4-d_5)  & \eta(d_2+d_4-s_1-s_2) & 0 & d_4 & \eta(d_4+d_5-d_1)  \\ \hline
d_5  & 0   & 0   & \eta(d_1-d_4+d_5)   & \eta(d_2-d_3+d_5)   & \eta(d_3+d_5-d_2)   & \eta(d_4+d_5-d_1)   & d_5  \\ \hline
\end{tabu}$}
\caption{The new $7$-dimensional algebra}\label{t:algebra D4-1}
\endgroup
\end{center}
\end{table}

The Gram matrix for the Frobenius form with 
respect to the basis $$\{s_1,s_2,d_1,d_2,d_3,d_4,d_5\}$$ 
can be easily found and it is as follows: 
\[ \left(\begin{matrix}
1 & 0 & \eta & \eta & \eta & \eta & 0 \\ 
0 & 1 & \eta & \eta & \eta & \eta & 0 \\
\eta & \eta & 2 & 0 & 2\eta & 2\eta & 2\eta \\
\eta & \eta & 0 & 2 & 2\eta & 2\eta & 2\eta \\
\eta & \eta & 2\eta & 2\eta & 2 & 0 & 2\eta \\
\eta & \eta & 2\eta & 2\eta & 0 & 2 & 2\eta \\
0 & 0 & 2\eta & 2\eta & 2\eta & 2\eta & 2
\end{matrix} \right).
\]
The determinant of the Gram matrix is $512(\eta-\frac{1}{2})^2(\eta+\frac{1}{4})$. So $A$ is simple if and only if $\eta\neq-1/4$ in $\F$.
Finally, observe that this algebra is a subalgebra of the 8-dimensional algebra $2Q_2(\eta)$ from Proposition~\ref{a:2qketa}, whose basis includes $c_{1,3}$ and $c_{2,4}$ instead of $c_{1,3}+c_{2,4}$.

{\bf Case~$2$.} Suppose that $|a^{bc}a^{de}|=2$.
Since $a^{de}=a^{ed}=e^{ad}$, we have a relation 
$(a^{bc}e^{ad})^2=1$ which is equivalent to $(a^{bcda}e)^2=1$. By the presentation A13 from the Appendix of~\cite{hs}, this relation together with the relations of the diagram $D_4$
defines the finite group $F(5,24)\simeq W_2(\tilde{D}_4)$.
This group has the center of order 8, so $G$ isomorphic to
a factor group of $\hat{G}$ by a subgroup of $Z(\hat{G})$.
We mentioned in Section~\ref{sec:3} that $W_2(\Tilde{D}_4)$   has the same central type as $Wr(2^2,4)$. It is convenient to consider the Fischer space of this group for calculations.
We follow the notation of Proposition~\ref{p:W2(Dn)} for $n=4$.

Identify the generating elements of $G$ with the following elements of $Wr(2^2, S_4)$:
$a=b_{1,2}$, $b=b_{1,3}$, $c=b_{2,4}$, $d=c_{1,3}$, and $e=d_{2,4}$. Then $a^{bc}=b_{3,4}$ and $a^{de}=e_{3,4}$. 
It is easy to see that $a$, $b$, $c$, $d$, and $e$
satisfy all relations of this case. Note that these elements generate $Wr(2^2, S_4)$, so we can assume that $G=\langle a,b,c,d,e \rangle$.

The following elements belong to $A=\dla x,y,z\dra$:
\begin{center}
$\begin{aligned}
s_1:=&\,x=a=b_{1,2},\\
d_1:=&\,y=b+c=b_{1,3}+b_{2,4},\\
d_2:=&\,d_1^{\tau_x}=b_{1,3}^{b_{1,2}}+b_{2,4}^{b_{1,2}}=b_{2,3}+b_{1,4}=b_{1,4}+b_{2,3},\\ 
d_3:=&\,z=d+e=c_{1,3}+d_{2,4},\\
d_4:=&\,d_3^{\tau_x}=c_{1,3}^{b_{1,2}}+d_{2,4}^{b_{1,2}}=c_{2,3}+d_{1,4}=d_{1,4}+c_{2,3},\\ 
s_2:=&\,x^{\tau_y}=b_{1,2}^{b_{1,3}b_{2,4}}=b_{3,4},\\ 
s_3:=&\,x^{\tau_z}=b_{1,2}^{c_{1,3}d_{2,4}}=e_{3,4},\\ 
s_4:=&\,s_2^{\tau_z}=b_{3,4}^{c_{1,3}d_{2,4}}=e_{1,2},\\  
d_5:=&\,d_2^{\tau_z}=b_{2,3}^{c_{1,3}d_{2,4}}+b_{1,4}^{c_{1,3}d_{2,4}}=e_{1,4}+e_{2,3},\\ 
d_6:=&\,d_5^{\tau_x}=e_{1,4}^{b_{1,2}}+e_{2,3}^{b_{1,2}}=e_{2,4}+e_{1,3}=e_{1,3}+e_{2,4},\\ 
d_7:=&\,d_4^{\tau_y}=c_{2,3}^{b_{1,3}b_{2,4}}+d_{1,4}^{b_{1,3}b_{2,4}}=c_{1,4}+d_{2,3},\\ 
d_8:=&\,d_7^{\tau_x}=c_{1,4}^{b_{1,2}}+d_{2,3}^{b_{1,2}}=c_{2,4}+d_{1,3}=d_{1,3}+c_{2,4},\\ 
u:=&\,\frac{2}{\eta}(\eta(d_1+d_4)-d_1\cdot d_4)=c_{1,2}+c_{3,4}+d_{1,2}+d_{3,4}. 
\end{aligned}$
\end{center}

Most computations are straightforward, and we write the details only for $u$:
\begin{multline*}
d_1 \cdot d_4=(b_{1,3}+b_{2,4})\cdot(c_{2,3}+d_{1,4})=\frac{\eta}{2}(b_{1,3}+c_{2,3}-c_{1,2})+
\frac{\eta}{2}(b_{1,3}+d_{1,4}-d_{3,4})\\+\frac{\eta}{2}(b_{2,4}+c_{2,3}-c_{3,4})+
\frac{\eta}{2}(b_{2,4}+d_{1,4}-d_{1,2})=\frac{\eta}{2}(2(d_1+d_4)-c_{1,2}-d_{3,4}-c_{3,4}-d_{1,2}). 
\end{multline*}

These 13 elements include all 24 points of the Fischer space
and have disjoint supports, so they are linearly independent. 
The products are shown in Table~\ref{t:algebra D4-2}, in particular we obtain a basis of $A$. 
We provide below two examples of computations:
\begin{multline*}
d_5\cdot d_6=\frac{\eta}{2}\big((e_{1,3}+e_{1,4}-e_{1,3}^{e_{1,4}})+(e_{1,3}+e_{2,3}-e_{1,3}^{e_{2,3}})+(e_{2,4}+e_{1,4}-e_{2,4}^{e_{1,4}})\\+(e_{2,4}+e_{2,3}-e_{2,4}^{e_{2,3}})\big)=
\frac{\eta}{2}\big((e_{1,3}+e_{1,4}-b_{3,4})+(e_{1,3}+e_{2,3}-b_{1,2})\\+(e_{2,4}+e_{1,4}-b_{1,2})+(e_{2,4}+e_{2,3}-b_{3,4})\big)
=\eta(d_5+d_6-s_1-s_2);
\end{multline*}
\begin{multline*}
d_3\cdot d_7=\frac{\eta}{2}\big((c_{1,3}+c_{1,4}-c_{1,3}^{c_{1,4}})+(c_{1,3}+d_{2,3}-c_{1,3}^{d_{2,3}})+
(d_{2,4}+c_{1,4}-d_{2,4}^{c_{1,4}})\\+(d_{2,4}+d_{2,3}-d_{2,4}^{d_{2,3}})\big)=\eta(d_3+d_7-s_2-s_4).
\end{multline*}

As a result, we find that $A$ is a 13-dimensional algebra spanned by the above 13 elements. 
The double axes $y$ and $z$ are primitive in $A$ since it does not contain any of $b$, $c$, $d$, and $e$.

Now we give another proof that our elements form a basis without finding their pairwise products. Consider the involutive automorphism $\tau$ of $G$ as in Proposition~\ref{p:tau:W2(Dn)}.
By definition, we find that $a^\tau=a$, $b^\tau=c$, and $d^\tau=e$. Denote $H=\langle\tau\rangle$. Then 
$A$ is a subalgebra in the 14-dimensional fixed subalgebra $M_H$.
Note that our set of 13 elements comprises of 12 orbit vectors for $\tau$ and the sum of two orbit vectors $c_{1,2}+d_{3,4}$ and $c_{3,4}+d_{1,2}$. It is easy to see that the relations of the diagram on $Z$ are preserved under the permutation $\rho$ which fixes $a$, $d$, and $e$, and switches $b$ and $c$.
Moreover, $\rho$ preserves the relation $(a^{bc}a^{de})^2=1$, so $\rho$ is an automorphism of order two of the group $G$. Hence $\rho$ induces an automorphism of the Fischer space. Since $x^\rho=x$, $y^\rho=y$, and $z^\rho=z$, the algebra $A$ lies in the fixed subalgebra of $H_2=\langle\rho\rangle$. Since $c_{1,2}=c_{1,3}^{b_{2,3}}=d^{b_{2,3}}$ and $b_{2,3}=b_{1,2}^{b_{1,3}}=a^b$,
we find that $b_{2,3}^\rho=a^c=b_{1,4}$ and $c_{1,2}^\rho=d^{b_{1,4}}=c_{3,4}$.
Similarly, we see that $d_{1,2}=d_{2,4}^{b_{1,4}}=e^{b_{1,4}}$ and $b_{1,4}=b_{1,2}^{b_{2,4}}=a^c$, so $d_{1,2}^\rho=d_{2,4}^{b_{2,3}}=d_{3,4}$.
Now $(c_{1,2}+d_{3,4})^\rho=c_{3,4}+d_{1,2}$ and hence $c_{1,2}+d_{3,4}\in M_H\setminus A$.
Therefore, $A$ has dimension less than 14, so it is 13-dimensional. 
Denote $v=c_{1,2}+d_{3,4}-c_{3,4}-d_{1,2}$. Then $v\in M_H$ and $v^\rho=-v$.
If $u\in A$, then $(u,v)=(u^\rho, v^\rho)=(u,-v)$. Since $(v,v)=4\neq 0$, we infer that
$A$ coincides with the orthogonal complement to $v$ in $M_H$. Finally, we note that $A$
is the 1-eigenspace of the action of $\rho$ on $M_H$ and $\dla v\dra$ is the (-1)-eigenspace.

Using GAP, it is easy to find the determinant of the Gram matrix for $A$ with respect to the basis $s_1,\ldots,s_4,d_1,\ldots,d_8$, and $u$. 
According to calculations, the determinant is
$131072\eta^3-49152\eta^2+1024=131072(\eta-\frac{1}{4})^2(\eta+\frac{1}{8})$. Hence $A$ is simple unless $\eta=\frac{1}{4}$ or $-\frac{1}{8}$ in~$\mathbb{F}$.


{\bf Case~$3$.} In this case $|a^{bc}a^{de}|=3$.
Since $a^{de}=a^{ed}=e^{ad}$, we have a relation 
$(a^{bc}e^{ad})^3=1$ which is equivalent to $(a^{bcda}e)^3=1$. By the presentation A15a from the Appendix of~\cite{hs}, this relation together with the relations of the diagram $D_4$
defines the finite group $F(5,36)\simeq W_3(\tilde{D}_4)$.
We use that $W_3(\tilde{D}_4)\simeq Wr(S_3,4)$ and Proposition~\ref{p:tau:W3(Dn)} for a description of the Fischer space of $Wr(S_3,4)$. Since there are no other factors of 
$Wr(S_3,4)$ satisfying the required equalities for $a,b,c,d,$ and $e$, we infer that $G=\hat{G}\simeq Wr(S_3,4)$.

Fix the following elements: $a=b_{1,2}$, $b=b_{1,3}$, $c=b_{2,4}$, $d=c_{1,3}$, and $e=e_{2,4}$. 
Then $a^{bc}=b_{1,2}^{b_{1,3}b_{2,4}}=b_{3,4}$
and $a^{de}=b_{1,2}^{c_{1,3}e_{2,4}}=c_{2,3}^{e_{2,4}}=g_{3,4}$. 
By Proposition~\ref{p:tau:W3(Dn)}, we see that $a^{bc}$ and $a^{de}$ lie on a line $\{b_{3,4}, d_{3,4}, g_{3,4}\}$ and hence $|a^{bc}a^{de}|=3$. Therefore, elements $a$, $b$, $c$, $d$, and $e$ satisfy the relations of this case. It is easy to see that they generate $Wr(S_3,4)$.
Now use them to find a basis and the dimension of $A$. The following elements belong to $A$:
\begin{center}
$\begin{aligned}
s_1:=&\,a = b_{1,2}, \quad s_4:=a^{\tau_y}=b_{1,2}^{b_{1,3}b_{2,4}}=b_{3,4},\\ 
s_2:=&\,s_4^{\tau_z}=b_{3,4}^{c_{1,3}e_{2,4}}=g_{1,2}, \quad s_3:=\,s_2^{\tau_x}=g_{1,2}^{b_{1,2}}=d_{1,2},\\ 
s_5:=&\,a^{\tau_z}=b_{1,2}^{c_{1,3}e_{2,4}}=g_{3,4},\quad 
s_6:=s_3^{\tau_y}=d_{1,2}^{b_{1,3}b_{2,4}}=d_{3,4},\\ 
d_1:=&\,y=b+c=b_{1,3}+b_{2,4}, \quad d_2:=y^{\tau_x}=b_{1,3}^{b_{1,2}}+b_{2,4}^{b_{1,2}}=b_{2,3}+b_{1,4},\\ 
d_3:=&\,z=d+e=c_{1,3}+e_{2,4}, \quad
d_4:=z^{\tau_x}=c_{1,3}^{b_{1,2}}+e_{2,4}^{b_{1,2}}=c_{2,3}+e_{1,4},\\ 
d_5:=&\,y^{\tau_{s_5}}=b_{1,3}^{g_{3,4}}+b_{2,4}^{g_{3,4}}=g_{1,4}+d_{2,3}, \quad d_6:=d_2^{\tau_{s_5}}=b_{2,3}^{g_{3,4}}+b_{1,4}^{g_{3,4}}=g_{2,4}+d_{1,3},\\
d_7:=&\,z^{\tau_{s_4}}=c_{1,3}^{b_{3,4}}+e_{2,4}^{b_{3,4}}=c_{1,4}+e_{2,3}, \quad d_8:=d_4^{\tau_{s_4}}=c_{2,3}^{b_{3,4}}+e_{1,4}^{b_{3,4}}=c_{2,4}+e_{1,3},\\
d_9:=&\,d_2^{\tau_{s_2}}=b_{2,3}^{g_{1,2}}+b_{1,4}^{g_{1,2}}=g_{1,3}+d_{2,4},\quad d_{10}:=y^{\tau_{s_6}}=b_{1,3}^{d_{3,4}}+b_{2,4}^{d_{3,4}}=d_{1,4}+g_{2,3},\\
d_{11}:=&\,d_4^{\tau_{s_2}}=c_{2,3}^{g_{1,2}}+e_{1,4}^{g_{1,2}}=f_{1,3}+f_{2,4}, \quad d_{12}:=d_{3}^{\tau_{s_6}}=c_{1,3}^{d_{3,4}}+e_{2,4}^{d_{3,4}}=f_{1,4}+f_{2,3},\\
d_{13}:=&\, \frac{1}{\eta}(\eta(d_1+d_{12})-d_1\cdot d_{12}) = f_{1,2}+f_{3,4} \\ 
u:=&\,\frac{2}{\eta}(\eta(d_1+d_4)-d_1\cdot d_4)=c_{1,2}+c_{3,4}+e_{1,2}+e_{3,4}. 
\end{aligned}$
\end{center}

The corresponding calculations for $d_{13}$:
\begin{multline*}
d_1 \cdot d_{12}=(b_{1,3}+b_{2,4}) \cdot (f_{1,4}+f_{2,3})=
\frac{\eta}{2}\big((b_{1,3}+f_{1,4}-b_{1,3}^{f_{1,4}})+(b_{1,3}+f_{2,3}-b_{1,3}^{f_{2,3}})\\+
(b_{2,4}+f_{1,4}-b_{2,4}^{f_{1,4}})+(b_{2,4}+f_{2,3}-b_{2,4}^{f_{2,3}})\big)=
\frac{\eta}{2}\big((b_{1,3}+f_{1,4}-f_{3,4})+(b_{1,3}+f_{2,3}-f_{1,2})\\+
(b_{2,4}+f_{1,4}-f_{1,2})+(b_{2,4}+f_{2,3}-f_{3,4})\big)=
\eta(d_1+d_{12}-(f_{1,2}+f_{3,4})).
\end{multline*}

These 20 elements $s_1,\ldots,s_6,d_1,\ldots d_{13},$ and $u$ include all 36 points of the Fischer space
and have disjoint supports, so they are linearly independent. 
The products are shown in Table~\ref{t:algebra D4-3}. 
Now we prove that these 20 elements form a basis of $A$ without appealing to the table.
Denote $\tau=f_{1,2}f_{3,4}$ which is, clearly, an involution in $G$.
Then $a^\tau=b_{1,2}^{f_{1,2}f_{3,4}}=a$, $b^\tau=b_{1,3}^{f_{1,2}f_{3,4}}=f_{2,3}^{f_{3,4}}=b_{2,4}=c$, and
$d^\tau=c_{1,3}^{f_{1,2}f_{3,4}}=g_{2,3}^{f_{3,4}}=e_{2,4}=e$.
Therefore, $A$ is a subalgebra in the fixed subalgebra of $H=\langle\tau\rangle$. 
Note that $f_{1,2}^\tau=f_{1,2}$, $f_{3,4}^\tau=f_{3,4}$, 
$c_{1,2}^\tau=e_{1,2}^{f_{3,4}}=e_{1,2}$, and $c_{3,4}^\tau=c_{3,4}^{f_{3,4}}=e_{3,4}$.
Inspecting the 20 written above elements, we find that $s_1,\ldots, s_6$, $d_1,\ldots, d_{12}$ are $\tau$-orbit vectors. Therefore, the algebra $M_H$ is 22-dimensional with a basis $s_1,\ldots, s_6$, $d_1,\ldots, d_{12}$, $f_{1,2}$, $f_{3,4}$, $c_{1,2}+e_{1,2}$, 
$c_{3,4}+e_{3,4}$. 

Similarly to Case~2, consider involutive automorphism $\rho$ of the Fischer space such that
$a^\rho=a$, $b^\rho=c$, $d^\rho=d$, and $e^\rho=e$. Denote $H_2=\langle\rho\rangle$.
Then $A$ is a subalgebra in the fixed subalgebra $M_{H_2}$. Using the definition of $\rho$,
we see that
$$e_{1,2}^\rho=(b_{1,4}e_{2,4}b_{1,4})^\rho=(a^c)^\rho e^\rho(a^c)^\rho=a^b e a^b=b_{2,3}e_{2,4}b_{2,3}=e_{3,4}, $$
$$c_{1,2}^\rho=(c_{1,3}^{b_{2,3}})^\rho=(a^b)^\rho d^\rho(a^b)^\rho=a^cda^c=b_{1,4}c_{1,3}b_{1,4}=c_{3,4},$$
$$f_{1,2}=(c_{1,2}^{e_{1,2}})^\rho=c_{3,4}^{e_{3,4}}=f_{3,4}.$$

Therefore, the subspace $M_H$ is invariant under the action of $\rho$, 
$A$ is a subspace of the 1-eigenspace of $\rho$, and $U=\dla f_{1,2} - f_{3,4}, c_{1,2}+e_{1,2}-c_{3,4}-e_{3,4}\dra$ is a subspace of the (-1)-eigenspace of $\rho$.
Since $M_H$ is 22-dimensional, we infer that $A$ is 20-dimensional and $M_H$ is the orthogonal direct sum of $A$ and~$U$.

Using GAP, we find the determinant of the Gram matrix with respect to the chosen basis equals $-1300(\eta-2)^{13}(\eta-\frac{1}{5})^2(\eta+\frac{1}{13})(\eta+1)^4$.  Hence the algebra is simple unless $\eta\in\{-1,-\frac{1}{13},2,\frac{1}{5}\}$.

\begin{sidewaystable}  
\begin{center}
\begingroup
\setlength{\tabcolsep}{10pt} %
\renewcommand{\arraystretch}{4}
\scalebox{.40}{
$\begin{tabu}[h!]{|c||c|c|c|c|c|c|c|c|c|c|c|c|c|}
\hline
& s_1 & s_2 & s_3 & s_4 & d_1 & d_2 & d_3 & d_4 & d_5 & d_6 & d_7 & d_8 & u \\ \hline 
s_1 & s_1 & 0 & 0 & 0 & \frac{\eta}{2}(2s_1+d_1-d_2) & \frac{\eta}{2}(2s_1-d_1+d_2) & \frac{\eta}{2}(2s_1+d_3-d_4) & \frac{\eta}{2}(2s_1-d_3+d_4) & \frac{\eta}{2}(2s_1+d_5-d_6) & \frac{\eta}{2}(2s_1-d_5+d_6) & \frac{\eta}{2}(2s_1+d_7-d_8) & \frac{\eta}{2}(2s_1-d_7+d_8) & 0 \\ \hline
s_2 & 0 & s_2 & 0 & 0 & \frac{\eta}{2}(2s_2+d_1-d_2) & \frac{\eta}{2}(2s_2-d_1+d_2) & \frac{\eta}{2}(2s_2+d_3-d_7) & \frac{\eta}{2}(2s_2+d_4-d_8) & \frac{\eta}{2}(2s_2+d_5-d_6) & \frac{\eta}{2}(2s_2-d_5+d_6) & \frac{\eta}{2}(2s_2-d_3+d_7) & \frac{\eta}{2}(2s_2-d_4+d_8) & 0 \\ \hline
s_3 & 0 & 0 & s_3 & 0 & \frac{\eta}{2}(2s_3+d_1-d_5) & \frac{\eta}{2}(2s_3+d_2-d_6) & \frac{\eta}{2}(2s_3+d_3-d_4) & \frac{\eta}{2}(2s_3-d_3+d_4) & \frac{\eta}{2}(2s_3-d_1+d_5) & \frac{\eta}{2}(2s_3-d_2+d_6) & \frac{\eta}{2}(2s_3+d_7-d_8) & \frac{\eta}{2}(2s_3-d_7+d_8) & 0 \\ \hline
s_4 & 0 & 0 & 0 & s_4 & \frac{\eta}{2}(2s_4+d_1-d_5) & \frac{\eta}{2}(2s_4+d_2-d_6) & \frac{\eta}{2}(2s_4+d_3-d_7) & \frac{\eta}{2}(2s_4+d_4-d_8) & \frac{\eta}{2}(2s_4-d_1+d_5) & \frac{\eta}{2}(2s_4-d_2+d_6) & \frac{\eta}{2}(2s_4-d_3+d_7) & \frac{\eta}{2}(2s_4-d_4+d_8) & 0 \\ \hline
d_1 & \frac{\eta}{2}(2s_1+d_1-d_2) & \frac{\eta}{2}(2s_2+d_1-d_2) & \frac{\eta}{2}(2s_3+d_1-d_5) & \frac{\eta}{2}(2s_4+d_1-d_5) & d_1 & \eta(d_1+d_2-s_1-s_2) & 0 & \frac{\eta}{2}(2d_1+2d_4-u) & \eta(d_1+d_5-s_3-s_4) & 0 & \frac{\eta}{2}(2d_1+2d_7-u) & 0 & \eta(2d_1+u-d_4-d_7) \\ \hline
d_2 & \frac{\eta}{2}(2s_1-d_1+d_2) & \frac{\eta}{2}(2s_2-d_1+d_2) & \frac{\eta}{2}(2s_3+d_2-d_6) & \frac{\eta}{2}(2s_4+d_2-d_6) & \eta(d_1+d_2-s_1-s_2) & d_2 & \frac{\eta}{2}(2d_2+2d_3-u) & 0 & 0 & \eta(d_2+d_6-s_3-s_4) & 0 & \frac{\eta}{2}(2d_2+2d_8-u) & \eta(2d_2+u-d_3-d_8) \\ \hline
d_3 & \frac{\eta}{2}(2s_1+d_3-d_4) & \frac{\eta}{2}(2s_2+d_3-d_7) & \frac{\eta}{2}(2s_3+d_3-d_4) & \frac{\eta}{2}(2s_4+d_3-d_7) & 0 & \frac{\eta}{2}(2d_2+2d_3-u) & d_3 & \eta(d_3+d_4-s_1-s_3) & \frac{\eta}{2}(2d_3+2d_5-u) & 0 & \eta(d_3+d_7-s_2-s_4) & 0 & \eta(2d_3+u-d_2-d_5) \\ \hline
d_4 & \frac{\eta}{2}(2s_1-d_3+d_4) & \frac{\eta}{2}(2s_2+d_4-d_8) & \frac{\eta}{2}(2s_3-d_3+d_4) & \frac{\eta}{2}(2s_4+d_4-d_8) & \frac{\eta}{2}(2d_1+2d_4-u) & 0 & \eta(d_3+d_4-s_1-s_3) & d_4 & 0 & \frac{\eta}{2}(2d_4+2d_6-u) & 0 & \eta(d_4+d_8-s_2-s_4) & \eta(2d_4+u-d_1-d_6) \\ \hline
d_5 & \frac{\eta}{2}(2s_1+d_5-d_6) & \frac{\eta}{2}(2s_2+d_5-d_6) & \frac{\eta}{2}(2s_3-d_1+d_5) & \frac{\eta}{2}(2s_4-d_1+d_5) & \eta(d_1+d_5-s_3-s_4) & 0 & \frac{\eta}{2}(2d_3+2d_5-u) & 0 & d_5 & \eta(d_5+d_6-s_1-s_2) & 0 & \frac{\eta}{2}(2d_5+2d_8-u) & \eta(2d_5+u-d_3-d_8) \\ \hline
d_6 & \frac{\eta}{2}(2s_1-d_5+d_6) & \frac{\eta}{2}(2s_2-d_5+d_6) & \frac{\eta}{2}(2s_3-d_2+d_6) & \frac{\eta}{2}(2s_4-d_2+d_6) & 0 & \eta(d_2+d_6-s_3-s_4) & 0 & \frac{\eta}{2}(2d_4+2d_6-u) & \eta(d_5+d_6-s_1-s_2) & d_6 & \frac{\eta}{2}(2d_6+2d_7-u) & 0 & \eta(2d_6+u-d_4-d_7) \\ \hline
d_7 & \frac{\eta}{2}(2s_1+d_7-d_8) & \frac{\eta}{2}(2s_2-d_3+d_7) & \frac{\eta}{2}(2s_3+d_7-d_8) & \frac{\eta}{2}(2s_4-d_3+d_7) & \frac{\eta}{2}(2d_1+2d_7-u) & 0 & \eta(d_3+d_7-s_2-s_4) & 0 & 0 & \frac{\eta}{2}(2d_6+2d_7-u) & d_7 & \eta(d_7+d_8-s_1-s_3) & \eta(2d_7+u-d_1-d_6) \\ \hline
d_8 & \frac{\eta}{2}(2s_1-d_7+d_8) & \frac{\eta}{2}(2s_2-d_4+d_8) & \frac{\eta}{2}(2s_3-d_7+d_8) & \frac{\eta}{2}(2s_4-d_4+d_8) & 0 & \frac{\eta}{2}(2d_2+2d_8-u) & 0 & \eta(d_4+d_8-s_2-s_4) & \frac{\eta}{2}(2d_5+2d_8-u) & 0 & \eta(d_7+d_8-s_1-s_3) & d_8 & \eta(2d_8+u-d_2-d_5) \\ \hline
u & 0 & 0 & 0 & 0 & \eta(2d_1+u-d_4-d_7) & \eta(2d_2+u-d_3-d_8) & \eta(2d_3+u-d_2-d_5) & \eta(2d_4+u-d_1-d_6) & \eta(2d_5+u-d_3-d_8) & \eta(2d_6+u-d_4-d_7) & \eta(2d_7+u-d_1-d_6) & \eta(2d_8+u-d_2-d_5) & u \\ \hline
\end{tabu}$}
\caption{The new 13-dimensional algebra}\label{t:algebra D4-2}
\endgroup
\end{center}
\end{sidewaystable}

\begin{sidewaystable}  
\begin{center}
\begingroup
\setlength{\tabcolsep}{10pt} %
\renewcommand{\arraystretch}{4}
\scalebox{.25}{
$\begin{tabu}[h!]{|c||c|c|c|c|c|c|c|c|c|c|c|c|c|c|c|c|c|c|c|c|}
\hline
& s_1 & s_2 & s_3 & s_4 & s_5 & s_6 & d_1 & d_2 & d_3 & d_4 & d_5 & d_6 & d_7 & d_8 & d_9 & d_{10} & d_{11} & d_{12} & d_{13} & u \\ \hline 
s_1 & s_1 & \frac{\eta}{2}(s_1+s_2-s_3) & \frac{\eta}{2}(s_1-s_2+s_3) & 0 & 0 & 0 & \frac{\eta}{2}(2s_1+d_1-d_2) & \frac{\eta}{2}(2s_1-d_1+d_2) & \frac{\eta}{2}(2s_1+d_3-d_4) & \frac{\eta}{2}(2s_1-d_3+d_4) & \frac{\eta}{2}(2s_1+d_5-d_6) & \frac{\eta}{2}(2s_1-d_5+d_6) & \frac{\eta}{2}(2s_1+d_7-d_8) & \frac{\eta}{2}(2s_1-d_7+d_8) & \frac{\eta}{2}(2s_1+d_9-d_{10}) & \frac{\eta}{2}(2s_1-d_9+d_{10}) & \frac{\eta}{2}(2s_1+d_{11}-d_{12}) & \frac{\eta}{2}(2s_1-d_{11}+d_{12}) & 0 & 0 \\ \hline
s_2 & \frac{\eta}{2}(s_1+s_2-s_3) & s_2 & \frac{\eta}{2}(s_2-s_1+s_3) & 0 & 0 & 0 & \frac{\eta}{2}(2s_2+d_1-d_5) & \frac{\eta}{2}(2s_2+d_2-d_9) & \frac{\eta}{2}(2s_2+d_3-d_7) & \frac{\eta}{2}(2s_2+d_4-d_{11}) & \frac{\eta}{2}(2s_2-d_1+d_5) & \frac{\eta}{2}(2s_2+d_6-d_{10}) & \frac{\eta}{2}(2s_2-d_3+d_7) & \frac{\eta}{2}(2s_2+d_8-d_{12}) & \frac{\eta}{2}(2s_2-d_2+d_9) & \frac{\eta}{2}(2s_2-d_6+d_{10}) & \frac{\eta}{2}(2s_2-d_4+d_{11}) & \frac{\eta}{2}(2s_2-d_8+d_{12}) & 0 & 0 \\ \hline
s_3 & \frac{\eta}{2}(s_1-s_2+s_3) & \frac{\eta}{2}(s_2-s_1+s_3) & s_3 & 0 & 0 & 0 & \frac{\eta}{2}(2s_3+d_1-d_{10}) & \frac{\eta}{2}(2s_3+d_2-d_6) & \frac{\eta}{2}(2s_3+d_3-d_{12}) & \frac{\eta}{2}(2s_3+d_4-d_8) & \frac{\eta}{2}(2s_3+d_5-d_9) & \frac{\eta}{2}(2s_3-d_2+d_6) & \frac{\eta}{2}(2s_3+d_7-d_{11}) & \frac{\eta}{2}(2s_3-d_4+d_8) & \frac{\eta}{2}(2s_3-d_5+d_9) & \frac{\eta}{2}(2s_3-d_1+d_{10}) & \frac{\eta}{2}(2s_3-d_7+d_{11}) & \frac{\eta}{2}(2s_3-d_3+d_{12}) & 0 & 0 \\ \hline
s_4 & 0 & 0 & 0 & s_4 & \frac{\eta}{2}(s_4+s_5-s_6) & \frac{\eta}{2}(s_4-s_5+s_6) & \frac{\eta}{2}(2s_4+d_1-d_2) & \frac{\eta}{2}(2s_4-d_1+d_2) & \frac{\eta}{2}(2s_4+d_3-d_7) & \frac{\eta}{2}(2s_4+d_4-d_8) & \frac{\eta}{2}(2s_4+d_5-d_9) & \frac{\eta}{2}(2s_4+d_6-d_{10}) & \frac{\eta}{2}(2s_4-d_3+d_7) & \frac{\eta}{2}(2s_4-d_4+d_8) & \frac{\eta}{2}(2s_4-d_5+d_9) & \frac{\eta}{2}(2s_4-d_6+d_{10}) & \frac{\eta}{2}(2s_4+d_{11}-d_{12}) & \frac{\eta}{2}(2s_4-d_{11}+d_{12}) & 0 & 0 \\ \hline
s_5 & 0 & 0 & 0 & \frac{\eta}{2}(s_4+s_5-s_6) & s_5 & \frac{\eta}{2}(s_5-s_4+s_6) & \frac{\eta}{2}(2s_5+d_1-d_5) & \frac{\eta}{2}(2s_5+d_2-d_6) & \frac{\eta}{2}(2s_5+d_3-d_4) & \frac{\eta}{2}(2s_5-d_3+d_4) & \frac{\eta}{2}(2s_5-d_1+d_5 & \frac{\eta}{2}(2s_5-d_2+d_6) & \frac{\eta}{2}(2s_5+d_7-d_{11}) & \frac{\eta}{2}(2s_5+d_8-d_{12}) & \frac{\eta}{2}(2s_5+d_9-d_{10}) &  \frac{\eta}{2}(2s_5-d_9+d_{10}) & \frac{\eta}{2}(2s_5-d_7+d_{11}) & \frac{\eta}{2}(2s_5-d_8+d_{12}) & 0 & 0 \\ \hline
s_6 & 0 & 0 & 0 & \frac{\eta}{2}(s_4-s_5+s_6) & \frac{\eta}{2}(s_5-s_4+s_6) & s_6 & \frac{\eta}{2}(2s_6+d_1-d_{10}) & \frac{\eta}{2}(2s_6+d_2-d_9) & \frac{\eta}{2}(2s_6+d_3-d_{12}) & \frac{\eta}{2}(2s_6+d_4-d_{11}) & \frac{\eta}{2}(2s_6+d_5-d_6) & \frac{\eta}{2}(2s_6-d_5+d_6) & \frac{\eta}{2}(2s_6+d_7-d_8) & \frac{\eta}{2}(2s_6-d_7+d_8) & \frac{\eta}{2}(2s_6-d_2+d_9) & \frac{\eta}{2}(2s_6-d_1+d_{10}) & \frac{\eta}{2}(2s_6-d_4+d_{11}) & \frac{\eta}{2}(2s_6-d_3+d_{12}) & 0 & 0 \\ \hline
d_1 & \frac{\eta}{2}(2s_1+d_1-d_2) & \frac{\eta}{2}(2s_2+d_1-d_5) & \frac{\eta}{2}(2s_3+d_1-d_{10}) & \frac{\eta}{2}(2s_4+d_1-d_2) & \frac{\eta}{2}(2s_5+d_1-d_5) & \frac{\eta}{2}(2s_6+d_1-d_{10}) & d_1 & \eta(d_1+d_2-s_1-s_4) & 0 & \frac{\eta}{2}(2d_1+2d_4-u) & \eta(d_1+d_5-s_2-s_5) & \frac{\eta}{2}(d_1+d_6-d_9) &  \frac{\eta}{2}(2d_1+2d_7-u) & 0 & \frac{\eta}{2}(d_1-d_6+d_9) & \eta(d_1+d_{10}-s_3-s_6) & 0 & \eta(d_1+d_{12}-d_{13}) & \eta(d_1-d_{12}+d_{13}) & \eta(2d_1-d_4-d_7+u) \\ \hline
d_2 & \frac{\eta}{2}(2s_1-d_1+d_2) & \frac{\eta}{2}(2s_2+d_2-d_9) & \frac{\eta}{2}(2s_3+d_2-d_6) & \frac{\eta}{2}(2s_4-d_1+d_2) & \frac{\eta}{2}(2s_5+d_2-d_6) & \frac{\eta}{2}(2s_6+d_2-d_9) & \eta(d_1+d_2-s_1-s_4) & d_2 & \frac{\eta}{2}(2d_2+2d_3-u) & 0 & \frac{\eta}{2}(d_2+d_5-d_{10}) & \eta(d_2+d_6-s_3-s_5) & 0 & \frac{\eta}{2}(2d_2+2d_8-u) & \eta(d_2+d_9-s_2-s_6) & \frac{\eta}{2}(d_2-d_5+d_{10}) & \eta(d_2+d_{11}-d_{13}) & 0  & \eta(d_2-d_{11}+d_{13}) & \eta(2d_2-d_3-d_8+u) \\ \hline
d_3 & \frac{\eta}{2}(2s_1+d_3-d_4) & \frac{\eta}{2}(2s_2+d_3-d_7) & \frac{\eta}{2}(2s_3+d_3-d_{12}) & \frac{\eta}{2}(2s_4+d_3-d_7) & \frac{\eta}{2}(2s_5+d_3-d_4) & \frac{\eta}{2}(2s_6+d_3-d_{12}) & 0 & \frac{\eta}{2}(2d_2+2d_3-u) & d_3 & \eta(d_3+d_4-s_1-s_5) & \frac{\eta}{2}(2d_3+2d_5-u) & 0 & \eta(d_3+d_7-s_2-s_4) & \frac{\eta}{2}(d_3+d_8-d_{11}) & 0 & \eta(d_3+d_{10}-d_{13}) & \frac{\eta}{2}(d_3-d_8+d_{11}) & \eta(d_3+d_{12}-s_3-s_6) & \eta(d_3-d_{10}+d_{13}) & \eta(2d_3-d_2-d_5+u) \\ \hline
d_4 & \frac{\eta}{2}(2s_1-d_3+d_4) & \frac{\eta}{2}(2s_2+d_4-d_{11}) & \frac{\eta}{2}(2s_3+d_4-d_8) & \frac{\eta}{2}(2s_4+d_4-d_8) & \frac{\eta}{2}(2s_5-d_3+d_4) & \frac{\eta}{2}(2s_6+d_4-d_{11}) & \frac{\eta}{2}(2d_1+2d_4-u) & 0 & \eta(d_3+d_4-s_1-s_5) & d_4 & 0 & \frac{\eta}{2}(2d_4+2d_6-u) & \frac{\eta}{2}(d_4+d_7-d_{12}) & \eta(d_4+d_8-s_3-s_4) & \eta(d_4+d_9-d_{13}) & 0 & \eta(d_4+d_{11}-s_2-s_6) & \frac{\eta}{2}(d_4-d_7+d_{12}) & \eta(d_4-d_9+d_{13}) & \eta(2d_4-d_1-d_6+u) \\ \hline
d_5 & \frac{\eta}{2}(2s_1+d_5-d_6) & \frac{\eta}{2}(2s_2-d_1+d_5) & \frac{\eta}{2}(2s_3+d_5-d_9) & \frac{\eta}{2}(2s_4+d_5-d_9) & \frac{\eta}{2}(2s_5-d_1+d_5) & \frac{\eta}{2}(2s_6+d_5-d_6) & \eta(d_1+d_5-s_2-s_5) & \frac{\eta}{2}(d_2+d_5-d_{10}) & \frac{\eta}{2}(2d_3+2d_5-u) & 0 & d_5 & \eta(d_5+d_6-s_1-s_6) & 0 & \eta(d_5+d_8-d_{13}) & \eta(d_5+d_9-s_3-s_4) & \frac{\eta}{2}(d_5-d_2+d_{10}) & \frac{\eta}{2}(2d_5+2d_{11}-u) & 0 & \eta(d_5-d_8+d_{13}) & \eta(2d_5-d_3-d_{11}+u) \\ \hline
d_6 & \frac{\eta}{2}(2s_1-d_5+d_6) & \frac{\eta}{2}(2s_2+d_6-d_{10}) & \frac{\eta}{2}(2s_3-d_2+d_6) & \frac{\eta}{2}(2s_4+d_6-d_{10}) & \frac{\eta}{2}(2s_5-d_2+d_6) & \frac{\eta}{2}(2s_6-d_5+d_6) & \frac{\eta}{2}(d_1+d_6-d_9) & \eta(d_2+d_6-s_3-s_5) & 0 & \frac{\eta}{2}(2d_4+2d_6-u) & \eta(d_5+d_6-s_1-s_6) & d_6 & \eta(d_6+d_7-d_{13}) & 0 & \frac{\eta}{2}(d_6+d_9-d_1) & \eta(d_6+d_{10}-s_2-s_4) & 0 & \frac{\eta}{2}(2d_6+2d_{12}-u) & \eta(d_6-d_7+d_{13}) & \eta(2d_6-d_4-d_{12}+u) \\ \hline
d_7 & \frac{\eta}{2}(2s_1+d_7-d_8) & \frac{\eta}{2}(2s_2-d_3+d_7) & \frac{\eta}{2}(2s_3+d_7-d_{11}) & \frac{\eta}{2}(2s_4-d_3+d_7) & \frac{\eta}{2}(2s_5+d_7-d_{11}) & \frac{\eta}{2}(2s_6+d_7-d_8) & \frac{\eta}{2}(2d_1+2d_7-u) & 0 & \eta(d_3+d_7-s_2-s_4) & \frac{\eta}{2}(d_4+d_7-d_{12}) & 0 & \eta(d_6+d_7-d_{13}) & d_7 & \eta(d_7+d_8-s_1-s_6) & \frac{\eta}{2}(2d_7+2d_9-u) & 0 & \eta(d_7-s_3-s_5+d_{11}) & \frac{\eta}{2}(d_7-d_4+d_{12}) & \eta(d_7-d_6+d_{13}) & \eta(2d_7-d_1-d_9+u) \\ \hline
d_8 & \frac{\eta}{2}(2s_1-d_7+d_8) & \frac{\eta}{2}(2s_2+d_8-d_{12}) & \frac{\eta}{2}(2s_3-d_4+d_8) & \frac{\eta}{2}(2s_4-d_4+d_8) & \frac{\eta}{2}(2s_5+d_8-d_{12}) & \frac{\eta}{2}(2s_6-d_7+d_8) & 0 & \frac{\eta}{2}(2d_2+2d_8-u) & \frac{\eta}{2}(d_3+d_8-d_{11}) & \eta(d_4+d_8-s_3-s_4) & \eta(d_5+d_8-d_{13}) & 0 & \eta(d_7+d_8-s_1-s_6) & d_8 & 0 & \frac{\eta}{2}(2d_8+2d_{10}-u) & \frac{\eta}{2}(d_8-d_3+d_{11}) & \eta(d_8+d_{12}-s_2-s_5) & \eta(d_8+d_{13}-d_5) & \eta(2d_8-d_2-d_{10}+u) \\ \hline
d_9 & \frac{\eta}{2}(2s_1+d_9-d_{10}) & \frac{\eta}{2}(2s_2-d_2+d_9) & \frac{\eta}{2}(2s_3-d_5+d_9) & \frac{\eta}{2}(2s_4-d_5+d_9) & \frac{\eta}{2}(2s_5+d_9-d_{10}) & \frac{\eta}{2}(2s_6-d_2+d_9) & \frac{\eta}{2}(d_1-d_6+d_9) & \eta(d_2+d_9-s_2-s_6) & 0 & \eta(d_4+d_9-d_{13}) & \eta(d_5+d_9-s_3-s_4) & \frac{\eta}{2}(d_6+d_9-d_1) & \frac{\eta}{2}(2d_7+2d_9-u) & 0 & d_9 & \eta(d_9+d_{10}-s_1-s_5) & 0 & \frac{\eta}{2}(2d_9+2d_{12}-u) & \eta(d_9-d_4+d_{13}) & \eta(2d_9-d_7-d_{12}+u) \\ \hline
d_{10} & \frac{\eta}{2}(2s_1-d_9+d_{10}) & \frac{\eta}{2}(2s_2-d_6+d_{10}) & \frac{\eta}{2}(2s_3-d_1+d_{10}) & \frac{\eta}{2}(2s_4-d_6+d_{10}) & \frac{\eta}{2}(2s_5-d_9+d_{10}) & \frac{\eta}{2}(2s_6-d_1+d_{10}) & \eta(d_1+d_{10}-s_3-s_6) & \frac{\eta}{2}(d_2-d_5+d_{10}) & \eta(d_3+d_{10}-d_{13}) & 0 & \frac{\eta}{2}(d_5-d_2+d_{10}) & \eta(d_6+d_{10}-s_2-s_4) & 0  & \frac{\eta}{2}(2d_8+2d_{10}-u) & \eta(d_9+d_{10}-s_1-s_5) & d_{10} & \frac{\eta}{2}(2d_{11}+2d_{10}-u) & 0 & \eta(d_{10}-d_3+d_{13}) & \eta(2d_{10}-d_8-d_{11}+u) \\ \hline
d_{11} & \frac{\eta}{2}(2s_1+d_{11}-d_{12}) & \frac{\eta}{2}(2s_2-d_4+d_{11}) & \frac{\eta}{2}(2s_3-d_7+d_{11}) & \frac{\eta}{2}(2s_4+d_{11}-d_{12}) & \frac{\eta}{2}(2s_5-d_7+d_{11}) & \frac{\eta}{2}(2s_6-d_4+d_{11}) & 0 & \eta(d_2+d_{11}-d_{13}) & \frac{\eta}{2}(d_3-d_8+d_{11}) & \eta(d_4+d_{11}-s_2-s_6) & \frac{\eta}{2}(2d_5+2d_{11}-u) & 0 & \eta(d_7+d_{11}-s_3-s_5) & \frac{\eta}{2}(d_8-d_3+d_{11}) & 0 & \frac{\eta}{2}(2d_{11}+2d_{10}-u) & d_{11} & \eta(d_{11}+d_{12}-s_1-s_4) & \eta(d_{11}-d_2+d_{13}) & \eta(2d_{11}-d_5-d_{10}+u) \\ \hline
d_{12} & \frac{\eta}{2}(2s_1-d_{11}+d_{12}) & \frac{\eta}{2}(2s_2-d_8+d_{12}) & \frac{\eta}{2}(2s_3-d_3+d_{12}) & \frac{\eta}{2}(2s_4-d_{11}+d_{12}) & \frac{\eta}{2}(2s_5-d_8+d_{12}) & \frac{\eta}{2}(2s_6-d_3+d_{12}) & \eta(d_1+d_{12}-d_{13}) & 0 & \eta(d_3+d_{12}-s_3-s_6) & \frac{\eta}{2}(d_4-d_7+d_{12}) & 0 & \frac{\eta}{2}(2d_6+2d_{12}-u) & \frac{\eta}{2}(d_7+d_{12}-d_4) & \eta(d_8+d_{12}-s_2-s_5) & \frac{\eta}{2}(2d_9+2d_{12}-u) & 0 & \eta(d_{11}+d_{12}-s_1-s_4) & d_{12} & \eta(d_{12}-d_1+d_{13}) & \eta(2d_{12}-d_6-d_9+u) \\ \hline
d_{13} & 0 & 0 & 0 & 0 & 0 & 0 & \eta(d_1-d_{12}+d_{13}) & \eta(d_2-d_{11}+d_{13}) & \eta(d_3-d_{10}+d_{13}) & \eta(d_4-d_9+d_{13}) & \eta(d_5-d_8+d_{13}) & \eta(d_6-d_7+d_{13}) & \eta(d_7-d_6+d_{13}) & \eta(d_8-d_5+d_{13}) & \eta(d_9-d_4+d_{13}) & \eta(d_{10}-d_3+d_{13}) & \eta(d_{11}-d_2+d_{13}) & \eta(d_{12}-d_1+d_{13}) & d_{13} & {\eta}d_{13}  \\ \hline
u & 0 & 0 & 0 & 0 & 0 & 0 & \eta(2d_1-d_4-d_7+u) & \eta(2d_2-d_3-d_8+u) & \eta(2d_3-d_2-d_5+u) & \eta(2d_4-d_1-d_6+u)  & \eta(2d_5-d_3-d_{11}+u) & \eta(2d_6-d_4-d_{12}+u) & \eta(2d_7-d_1-d_9+u) & \eta(2d_8-d_2-d_{10}+u) & \eta(2d_9-d_7-d_{12}+u) & \eta(2d_{10}-d_8-d_{11}+u) & \eta(2d_{11}-d_5-d_{10}+u) & \eta(2d_{12}-d_6-d_9+u) & {\eta}d_{13} & (\eta+1)u-{\eta}d_{13} \\ \hline
\end{tabu}$}
\caption{Table of products for the new 20-dimensional algebra}\label{t:algebra D4-3}
\endgroup
\end{center}
\end{sidewaystable}

\subsection{Diagram $D_5$}\label{sec:D5}
In this subsection, we suppose that the diagram on $Z$ is $D_5$.
Since $d$, $c$, $a$, $b$, and $e$ form the Coxeter diagram of type~$A_5$,
we infer that $G=\hat{G}$ is isomorphic to $S_6$, and we may identify support elements with transpositions in the following way: $a=(1,2)$, $b=(1,3)$, $c=(2,4)$, $d=(4,6)$, and $e=(3,5)$. 
Then the following elements belong to $A$:
\begin{center}
$\begin{aligned}
s_1:=&\,x=a=(1,2),\\
d_1:=&\,y=b+c=(1,3)+(2,4),\\
s_2:=&\,x^{\tau_y}=(3,4),\\
d_2:=&\,z=d+e=(3,5)+(4,6),\\
s_3:=&\,s_2^{\tau_z}=(5,6),\\
d_3:=&\,y^{\tau_x}=(1,4)+(2,3),\\
d_4:=&\,z^{\tau_{s_2}}=(3,6)+(4,5),\\
d_5:=&\,y^{\tau_z}=(1,5)+(2,6),\\
d_6:=&\,d_5^{\tau_{s_3}}=(1,6)+(2,5).\\
\end{aligned}$
\end{center}

Consider an involuntary automorphism $\tau=(1,2)(3,4)(5,6)$ of $G$.
Then it is easy to see that our nine elements are exactly orbit elements for the action of $H=\langle\tau\rangle$ on the set of points $a^G$. Therefore, $A$ coincides with the fixed subalgebra $M_H$.
Since there are no extras among orbit vectors, we infer that the flip subalgebra $A(\tau)$ is the whole $M_H$. This flip subalgebra is denoted by $Q_3(\eta)$ in~\cite{gjmss} and it has a generalization for arbitrary symmetric group $S_{2k}$ with $k\geq2$ \cite[Prop.~8.1 and Cor.~8.2]{gjmss}. Finally, observe that the written basis is exactly the same as that of $M_H$ in \cite[Prop.~8.1]{gjmss} for $k=3$.


\subsection{Diagram $D_7$}\label{sec:D7}
In this subsection, we suppose that the diagram on $Z$ is $D_7$.
Step by step we add relations from the presentation A18a of the Appendix in \cite{hs}, replacing $a$ by $b$ and vice versa. All relations come from the 3-transposition property of $G$. 

{\bf Case~1}. Assume that $(a^cd)^2=1$ and $(a^be)^2=1$. Clearly, elements $b$, $a$, and $c$ generate a subgroup isomorphic to $S_4$. Since $|ec|=|db|=2$, we find that $|a^{bc}e|=|a^be^c|=|a^be|=2$ and $|a^{bc}d|=|a^{cb}d|=|a^cd^b|=|a^cd|=2$.
Consequently, we see that $e$, $b$, $a^{bc}$, $c$, and $d$ form the Coxeter diagram of type $A_5$ and hence
$\hat{G}$ is a homomorphic image of $S_6$. On the other hand, if we take $a:=(1,2)$, $b:=(1,3)$, $c:=(2,4)$, $d:=(2,5)$, $e:=(1,6)$,
then the elements satisfy all relations of the diagram $D_7$ and, moreover, $a^cd=(1,4)(2,5)$ and $a^be=(2,3)(1,6)$ are elements of order two. Therefore, in this case $G=S_6$ and we can assume that $a$, $b$, $c$, $d$, and $e$ are as chosen. It is easy to see that elements $a^{bc}=(3,4)$, $b$, $c$, $d$, and $e$ form the diagram $D_5$. Note that $a^{bc}=x^{\tau_y}\in A$. Since $x\in\dla x^{\tau_y}, y\dra$ and $x^{\tau_y}\in\dla x, y\dra$, we infer that $A$ is isomorphic to the 9-dimensional algebra $Q_3(\eta)$ which is the unique algebra arising for the diagram $D_5$ in Subsection~\ref{sec:D5}.

\vspace{10pt}
Note that $y^{\tau_x}=b^a+c^a=a^b+a^c$ is a double axis in $A$. Now $a^b+a^c$ and $d+e$ must generate a primitive subalgebra in $A$.
It follows from~\cite[Theorem~5.5]{gjmss} that $|a^cd|=|a^be|$.
Therefore, we further assume that $(a^cd)^3=(a^be)^3=1$. 
Then \cite[Prop.~6.3]{hs} implies that $\langle a,b,c,d\rangle\simeq F_4(18) \simeq W_3(\Tilde{A}_4)$.

{\bf Case~2}. Assume that $e\in H=\langle a,b,c,d\rangle$, and so $G=H$. 
We use the description of $\Gamma(W_3(\Tilde{A}_4))$ introduced in Subsection~\ref{Wp(An)}. Identify elements of the diagram as follows:
$$a:=b_{1,2}, b:=b_{1,3}, c:=b_{2,4}, \text{ and } d:=c_{2,4}.$$
Since $e$ commutes with $d$ and $e\neq b$, we infer that $e\in\{c_{1,3},c_{3,1}\}$.
First, we assume that $e=c_{1,3}=b^{(acd)^2}$.

We proceed by listing elements that lie in $A$:
\begin{center}
$\begin{aligned}
s_1:=&\,x=a=b_{1,2},\\
s_2:=&\,x^{\tau_y}=b_{1,2}^{b_{1,3}b_{2,4}}=b_{3,4},\\ 
d_1:=&\,y=b+c=b_{1,3}+b_{2,4},\\
d_2:=&\,d_1^{\tau_x}=b_{1,3}^{b_{1,2}}+b_{2,4}^{b_{1,2}}=b_{2,3}+b_{1,4},\\ 
d_3:=&\,z=d+e=c_{2,4}+c_{1,3},\\
d_4:=&\,d_3^{\tau_x}=c_{2,4}^{b_{1,2}}+c_{1,3}^{b_{1,2}}=c_{1,4}+c_{2,3},\\ 
d_5:=&\,d_2^{\tau_z}=b_{2,3}^{c_{2,4}c_{1,3}}+b_{1,4}^{c_{2,4}c_{1,3}}=c_{3,4}^{c_{1,3}}+c_{2,1}^{c_{1,3}}=c_{4,1}+c_{3,2},\\ 
d_6:=&\,d_5^{\tau_x}=c_{4,1}^{b_{1,2}}+c_{3,2}^{b_{1,2}}=c_{4,2}+c_{3,1},\\ 
u:=&\,\frac{2}{\eta}(\eta(d_1+d_4)-d_1\cdot d_4)=2(b_{1,3}+b_{2,4}+c_{1,4}+c_{2,3}) \\
   &\,-\frac{2}{\eta}(b_{1,3}\cdot c_{1,4}+b_{1,3}\cdot c_{2,3}+b_{2,4}\cdot c_{1,4}+b_{2,4}\cdot c_{2,3})\\ 
   &\,=c_{3,4}+c_{1,2}+c_{2,1}+c_{4,3}.\\ 
\end{aligned}$
\end{center}

These nine elements include all 18 points of the Fischer space
and have disjoint supports, so they are linearly independent. 
A straightforward calculation shows that this set of elements of $A$ is closed under the multiplication. 
We list the results of the pairwise products in Table~\ref{t:algebra D7-1}. Consequently, we infer that $A$ is 9-dimensional. We provide several examples of calculations:
\begin{multline*}
s_1 \cdot d_3= b_{1,2} \cdot (c_{2,4}+c_{1,3}) = \frac{\eta}{2}[(b_{1,2} + c_{2,4} - b_{1,2}^{c_{2,4}})+
(b_{1,2} + c_{1,3} - b_{1,2}^{c_{1,3}})]\\= \frac{\eta}{2} (2s_1+d_3-c_{1,4}-c_{2,3})=\frac{\eta}{2} (2s_1+d_3-d_4)
\end{multline*}
\begin{multline*}
s_1 \cdot u = b_{1,2} \cdot (c_{3,4}+c_{1,2}+c_{2,1}+c_{4,3}) = 
b_{1,2} \cdot (c_{1,2}+c_{2,1})\\ = \frac{\eta}{2}[2b_{1,2} + c_{1,2}+c_{2,1} - (b_{1,2}^{c_{1,2}} +  b_{1,2}^{c_{2,1}}) ]= \frac{\eta}{2}  [2b_{1,2} + c_{1,2}+c_{2,1} - (c_{2,1} +  c_{1,2}) ] = \eta s_1.
\end{multline*}
\begin{multline*}
d_1 \cdot d_3=(b_{1,3}+b_{2,4}) \cdot (c_{2,4}+c_{1,3}) = b_{1,3} \cdot c_{1,3} + b_{2,4} \cdot c_{2,4}\\=\frac{\eta}{2}  [(b_{1,3}+c_{1,3}-c_{3,1}) + (b_{2,4}+c_{2,4}-c_{4,2})]=\frac{\eta}{2} (d_1+d_3-d_6)
\end{multline*}

\begin{multline*}
d_1 \cdot d_5 = (b_{1,3}+b_{2,4}) \cdot (c_{4,1}+c_{3,2}) = \frac{\eta}{2} [2d_1+2d_5 - (b_{1,3}^{c_{4,1}} +
b_{1,3}^{c_{3,2}} + b_{2,4}^{c_{4,1}} + b_{2,4}^{c_{3,2}})]\\=
\frac{\eta}{2} [2d_1+2d_5 - (c_{4,3} +c_{1,2} + c_{2,1} + c_{3,4})] = \frac{\eta}{2} (2d_1+2d_5-u).
\end{multline*}
\begin{multline*}
d_1 \cdot u = (b_{1,3}+b_{2,4}) \cdot (c_{3,4}+c_{1,2}+c_{2,1}+c_{4,3})\\ = \frac{\eta}{2} [4d_1 + 2u
- (b_{1,3}^{c_{3,4}} + b_{1,3}^{c_{1,2}} + b_{1,3}^{c_{2,1}} + b_{1,3}^{c_{4,3}} +
b_{2,4}^{c_{3,4}} + b_{2,4}^{c_{1,2}} + b_{2,4}^{c_{2,1}} + b_{2,4}^{c_{4,3}})]\\=
\frac{\eta}{2} (4d_1 + 2u - (c_{1,4}+c_{3,2}+c_{2,3}+c_{4,1}+c_{3,2}+c_{1,4}+c_{4,1}+c_{2,3})=
\eta(2d_1+u-d_4-d_5).
\end{multline*}
\begin{multline*}
u\cdot u=c_{1,2}\cdot c_{1,2}+2c_{1,2}\cdot c_{2,1}+c_{2,1}\cdot c_{2,1}+c_{3,4}\cdot c_{3,4}+2c_{3,4}\cdot c_{4,3}+c_{4,3}\cdot c_{4,3}\\=u+\eta(c_{1,2}+c_{2,1}-b_{1,2}+c_{3,4}+c_{4,3}-b_{3,4})=
(\eta+1)u-\eta(s_1+s_2).
\end{multline*}

Denote $\tau=b_{1,2}b_{3,4}$ and $H=\langle\tau\rangle$. 
By Proposition~\ref{basis}, $\tau$ induces an involutive automorphism of the Fischer space and $\tau$-orbit vectors are two singles $s_1$ and $s_2$, six doubles $d_1,\ldots,d_6$, and two extras $c_{1,2}+c_{2,1}$ and $c_{3,4}+c_{4,3}$. 
In particular, the fixed subalgebra $M_H$ is 10-dimensional.
Since $A$ includes all doubles of $M_H$, we infer that the flip algebra $A(\tau)$ coincides with $A$.

Using GAP, we find that the determinant of the Gram matrix for $A$ with respect to the basis $s_1,s_2,d_1,\ldots,d_6$, and $u$ equals $-2^5\cdot(\eta-2)^5(\eta-\frac{1}{2})^2(7\eta+1)$.
Hence, the algebra is simple except for the cases when 
$\eta\neq2$ or $\operatorname{char}\mathbb{F}\neq7$ and $\eta\neq-\frac{1}{7}$.

Note that there are two known 9-dimensional algebras arising in a similar manner, that is 
$Q_3(\eta)$, found by Joshi in \cite{Joshi-phd}, and $Q^3(\eta)$, found by Alsaeedi in \cite{alsaeedi-paper}, however in our case the basis includes a vector which is the sum of two extras, moreover, critical values for $\eta$ differ from ones for these two algebras.

\begin{table}[ht]
\begin{center}
\begingroup
\setlength{\tabcolsep}{20pt}
\scalebox{.5}{
$\begin{tabu}[h!]{|c||c|c|c|c|c|c|c|c|c|}
\hline
 & s_1 & s_2 & d_1 & d_2 & d_3 & d_4 & d_5 & d_6 & u \\ \hline
s_1 & s_1   & 0 & \frac{\eta}{2}(2s_1+d_1-d_2)  & \frac{\eta}{2}(2s_1-d_1+d_2) & \frac{\eta}{2}(2s_1+d_3-d_4) & \frac{\eta}{2}(2s_1-d_3+d_4) & \frac{\eta}{2}(2s_1+d_5-d_6) & \frac{\eta}{2}(2s_1-d_5+d_6)  & \eta{s}_1  \\ \hline 
s_2 & 0 & s_2 & \frac{\eta}{2}(2s_2+d_1-d_2)  & \frac{\eta}{2}(2s_2-d_1+d_2) & \frac{\eta}{2}(2s_2+d_3-d_4) & \frac{\eta}{2}(2s_2-d_3+d_4) & \frac{\eta}{2}(2s_2+d_5-d_6) & \frac{\eta}{2}(2s_2-d_5+d_6) & \eta{s}_2 \\ \hline 
d_1 & \frac{\eta}{2}(2s_1+d_1-d_2) & \frac{\eta}{2}(2s_2+d_1-d_2) & d_1 & \begin{tabu}{@{}c@{}} \eta(d_1+d_2 \\ -s_1-s_2) \end{tabu} & \frac{\eta}{2}(d_1+d_3-d_6) & \frac{\eta}{2}(2d_1+2d_4-u) & \frac{\eta}{2}(2d_1+2d_5-u) & \frac{\eta}{2}(d_1-d_3+d_6) & \begin{tabu}{@{}c@{}} \eta(2d_1+u \\ -d_4-d_5) \end{tabu} \\ \hline 
d_2 & \frac{\eta}{2}(2s_1-d_1+d_2) & \frac{\eta}{2}(2s_2-d_1+d_2) & \begin{tabu}{@{}c@{}} \eta(d_1+d_2 \\ -s_1-s_2) \end{tabu} & d_2 & \frac{\eta}{2}(2d_2+2d_3-u) & \frac{\eta}{2}(d_2+d_4-d_5) & \frac{\eta}{2}(d_2-d_4+d_5) & \frac{\eta}{2}(2d_2+2d_6-u) & \begin{tabu}{@{}c@{}} \eta(2d_2+u \\ -d_3-d_6) \end{tabu}  \\ \hline 
d_3 & \frac{\eta}{2}(2s_1+d_3-d_4) & \frac{\eta}{2}(2s_2+d_3-d_4) & \frac{\eta}{2}(d_1+d_3-d_6) & \frac{\eta}{2}(2d_2+2d_3-u) & d_3 & \begin{tabu}{@{}c@{}} \eta(d_3+d_4 \\ -s_1-s_2) \end{tabu} & \frac{\eta}{2}(2d_3+2d_5-u) & \frac{\eta}{2}(d_3+d_6-d_1) & \begin{tabu}{@{}c@{}} \eta(2d_3+u \\-d_2-d_5) \end{tabu} \\ \hline 
d_4 & \frac{\eta}{2}(2s_1-d_3+d_4)  & \frac{\eta}{2}(2s_2-d_3+d_4) & \frac{\eta}{2}(2d_1+2d_4-u) & \frac{\eta}{2}(d_2+d_4-d_5) & \begin{tabu}{@{}c@{}} \eta(d_3+d_4 \\ -s_1-s_2) \end{tabu} & d_4 & \frac{\eta}{2}(d_4-d_2+d_5) & \frac{\eta}{2}(2d_4+2d_6-u) & \begin{tabu}{@{}c@{}} \eta(2d_4+u \\ -d_1-d_6) \end{tabu} \\ \hline 
d_5 & \frac{\eta}{2}(2s_1+d_5-d_6) & \frac{\eta}{2}(2s_2+d_5-d_6) & \frac{\eta}{2}(2d_1+2d_5-u) & \frac{\eta}{2}(d_2-d_4+d_5) & \frac{\eta}{2}(2d_3+2d_5-u) & \frac{\eta}{2}(d_5-d_2+d_4) & d_5 & \begin{tabu}{@{}c@{}} \eta(d_5+d_6 \\ -s_1-s_2) \end{tabu} & \begin{tabu}{@{}c@{}} \eta(2d_5+u \\ -d_1-d_3) \end{tabu} \\ \hline 
d_6 & \frac{\eta}{2}(2s_1-d_5+d_6) & \frac{\eta}{2}(2s_2-d_5+d_6) & \frac{\eta}{2}(d_1-d_3+d_6) & \frac{\eta}{2}(2d_2+2d_6-u) & \frac{\eta}{2}(d_3-d_1+d_6) & \frac{\eta}{2}(2d_4+2d_6-u) & \begin{tabu}{@{}c@{}} \eta(d_5+d_6 \\ -s_1-s_2) \end{tabu} & d_6 & \begin{tabu}{@{}c@{}} \eta(2d_6+u \\ -d_2-d_4) \end{tabu}  \\ \hline 
u & \eta{s}_1  & \eta{s}_2  & \begin{tabu}{@{}c@{}} \eta(2d_1+u \\ -d_4-d_5) \end{tabu} & \begin{tabu}{@{}c@{}} \eta(2d_2+u \\ -d_3-d_6) \end{tabu} & \begin{tabu}{@{}c@{}} \eta(2d_3+u \\ -d_2-d_5) \end{tabu} & \begin{tabu}{@{}c@{}} \eta(2d_4+u \\ -d_1-d_6) \end{tabu} & \begin{tabu}{@{}c@{}} \eta(2d_5+u \\ -d_1-d_3) \end{tabu} & \begin{tabu}{@{}c@{}} \eta(2d_6+u \\ -d_2-d_4) \end{tabu} & (\eta+1)u-\eta{s}_1-\eta{s}_2  \\ \hline 
\end{tabu}$}
\caption{The new $9$-dimensional algebra}\label{t:algebra D7-1}
\endgroup
\end{center}
\end{table}

Suppose now that $e=c_{3,1}=b^{(acd)^{-2}}=b^{(dca)^2}$. We claim that the algebra $A$ coincides with the 12-dimensional algebra $3Q_2(\eta)$ which was obtained during the classification of the primitive algebras generated by two single axes and one double axis \cite[Table~11]{gjmss}. 
According to \cite[Section 6.1]{gjmss}, elements $b_{1,2}=a$, 
$c_{3,4}$, and $b_{1,3}+b_{2,4}=y$ generate the algebra $3Q_2(\eta)$. Observe that $z=d+e$ belongs to the basis of $3Q_2(\eta)$ provided in \cite{gjmss}. This implies that $A$ is a subalgebra in $3Q_2(\eta)$. It remains to show that $c_{3,4}$ is in $A$. Note that $x^{\tau_z}=b_{1,2}^{c_{3,1}c_{2,4}}=c_{3,2}^{c_{2,4}}=c_{3,4}$. Since, $x^{\tau_z}\in A$, we infer that $A=3Q_2(\eta)$. 
Finally, we note that this construction have been generalized to a series of algebras $3Q_k(\eta)$ for arbitrary $k\geq1$ \cite[Proposition 9.3]{gjmss}.

Further we assume that $e\in G\setminus H$. 
Recall that elements $a$, $b$, $c$, $d$, and $e$ satisfy the set of relations $R_0=R(a,b,c,d,e)\cup\{(a^cd)^3,(a^be)^3\}$. 
Now we continue to extend the list of defining relations of $G$.

\textbf{Case 3.} Assume that $(b^{(acd)^2}e)^2=1$. 
Since $e\not\in H$, we have $|b^{(acd)^2}e|=2$.
As above, we see that $b^{(acd)^2}=c_{1,3}$.
Denote $A=c_{1,3}$, $B=a=b_{1,2}$, $C=d=c_{2,4}$, $D=c=b_{2,4}$ and $E=e$.
Then $|AB|=|c_{1,3}b_{1,2}|=3$, $|AC|=|c_{1,3}c_{2,4}|=2$, $|AD|=|c_{1,3}b_{2,4}|=2$, and $|AE|=|b^{(acd)^2}e|=2$. Similarly, we see that $|BC|=|b_{1,2}c_{2,4}|=3$, $|BD|=|b_{1,2}b_{2,4}|=3$, 
$|BE|=|ae|=3$, $|CD|=|c_{2,4}b_{2,4}|=3$, $|CE|=|de|=2$, and $|DE|=|ce|=2$. Since $|a^dc|=3$, we have $|B^CD|=3$.
By definition, $A=b^{(acd)^2}=b^{(BDC)^2}$ and hence $b=A^{(CDB)^2}$. Since $A$ commutes with $C$ and $D$, we find that $b=A^{BCDB}$. Finally, $b$ commutes with $C$ and $D$,
so $b=A^{BCDBCD}=A^{(BCD)^2}$. Then $|EA^{(BCD)^2}|=|eb|=3$.
It follows from the presentation A17b in \cite{hs}
that $G=\langle a,b,c,d,e\rangle=\langle A,B,C,D,E\rangle$ is a homomorphic image of $F(5,72)\simeq 2^{1+8}:3^3:S_4$. 
This can be obtained directly enumerating cosets of $\hat G / \langle R_0\cup\{(b^{(acd)^2}e)^2\} \rangle$ by $H$. Inspecting the homomorphic images of $F(5,72)$, we find that $G\simeq F(5,72)$ or $G\simeq F(5,72)/Z(F(5,72))$, where $|Z(F(5,72)|=2$.
According to discussion in \cite[Example~17]{hs},
$F(5,72)$ has the same central type as $Wr(A_4,4)$. Therefore, we can assume that $G=Wr(A_4,4)$. We use the description of its Fischer space from Lemma~\ref{l:Wr(A4,n)}. Identify the elements of $D_7$-diagram with the following points of $\Gamma(G)$:
$$a:=(1,2),b:=(1,3),c:=(2,4),d:=(1,4,2).(2,4),e:=(1,2,3).(1,3)$$
Then $a^c=(1,4)$ and $a^b=(2,3)$, so $|a^cd|=|a^be|=3$.
Moreover, $b^{(acd)^2}=(1,4,2).(1,3)$ and hence $|b^{(acd)^2}e|=2$.
Since the chosen elements satisfy all the relations of this case and generate $Wr(A_4,4)$, we can use them for calculations in $A$.

As in the most of the previous cases, we see that the following four elements belong to $A$ and generate the subalgebra $Q_2(\eta)$:
$$ s_1:=x=a=(1,2), s_2:=x^{\tau_y}=(3,4),$$ 
$$d_1:=y=b+c=(1,3)+(2,4), d_2:=y^{\tau_x}=(1,4)+(2,3).$$
Denote $d_9:=z=d+e=(1,4,2).(2,4)+(1,2,3).(1,3)$.

Consider the involutive automorphism $\tau$ of $G$ as in Proposition~\ref{p:fixed:Wr(A_4,n)}
and denote $H=\langle\tau\rangle$.
By Proposition~\ref{p:fixed:Wr(A_4,n)}, the fixed subalgebra $M_H$ has dimension $40$ and a basis with $8$ singles, $24$ doubles, and $8$ extras with respect to $\tau$. Consider a permutation $\sigma=(1,2)(3,4)$. If $\pi_1\in A_4$ and $\pi_2$ is a transposition in  
$S_4$, then we see that $(\pi_1.\pi_2)^{\tau}=\pi_1^\sigma.\pi_2^\sigma$. 
Applying these rules, we find that $\tau$ fixes $a$, switches $b$ with $c$ and $d$ with $e$. Therefore, $A$ is a subalgebra of $M_H$. Now we prove that $A$ 
is $39$-dimensional with a basis including $6$ singles and all $24$ doubles of $M_H$.

We claim that six singles corresponding to $\tau$ are in $A$, and two remaining singles form a double axis $(1,2)(3,4).(1,2)+(1,2)(3,4).(3,4)$. 
Applying Miyamoto involutions, we find that the following elements belong to $A$: 
$$s_3:=x^{\tau_z}=a^{de}=((1,4,2).(1,4))^e=(1,3)(2,4).(3,4), s_4:=s_3^{\tau_y}=(1,3)(2,4).(1,2),$$ $$s_5:=s_4^{\tau_z}=(1,3)(2,4).(1,2)^{de}=((1,3,4).(1,4))^{(1,2,3).(1,3))}=  (1,4)(2,3).(3,4),$$ and $s_6:=s_5^{\tau_y}=(1,4)(2,3).(1,2)$.

Consider 24 doubles of $M_H$ in three steps. 
First, we look at eight elements of shapes $t.(1,3)+t.(2,4)$ and $t.(1,4)+t.(2,3)$, where $t\in K_4$. We claim that all these doubles are in $A$. Observe that two of them are $d_1$ and $d_2$. Then
\begin{multline*}
d_3:=d_1^{\tau_{s_5}}=(1,3)^{(1,4)(2,3).(2,4)}+(2,3)^{(1,4)(2,3).(3,4)}\\=
(1,4)(2,3).(1,4)+(1,4)(2,3).(2,3)
\end{multline*}
lies in $A$. Similarly, we find that 
$d_4:=d_1^{\tau_{s_3}}=(1,3)(2,4).(1,4)+(1,3)(2,4).(2,3)$ belongs to $A$. 
Using Lemma~\ref{l:lines}, we get the following double in $A$:
\begin{multline*}
d_5:=(d_3^{\tau_{s_3}})^{\tau_x}=((1,4)(2,3).(1,4))^{\tau_{s_3}\tau_x}+((1,4)(2,3).(2,3))^{\tau_{s_3}\tau_x}\\=((1,2)(3,4).(1,3))^{\tau_x}+((1,2)(3,4).(2,4))^{\tau_x}=(1,2)(3,4).(2,3)+(1,2)(3,4).(1,4).\end{multline*}
Applying $\tau_x$, we find the remaining three doubles:
$$d_6:=d_3^{\tau_x}=(1,4)(2,3).(2,4)+(1,4)(2,3).(1,3),$$
$$d_7:=d_4^{\tau_x}=(1,3)(2,4).(2,4)+(1,3)(2,4).(1,3), \text{ and}$$
$$d_8:=d_5^{\tau_x}=(1,2)(3,4).(2,4)+(1,2)(3,4).(1,3).$$

Now consider $16$ doubles of the shapes $t.(1,3)+t^{\tau}.(2,4)$ and $t.(1,4)+t^{\tau}.(2,3)$, where $t$ is a 3-cycle in $A_4$. We claim that all these elements belong  to $A$. Since $\tau_x$ switches elements of these two shapes, it is sufficient to show that the elements of the first shape are in $A$. Suppose that $s\in K_4$ and $t$ is a $3$-cycle in $A_4$. We know that $t.(1,3)+t^\sigma.(2,4)$ is a double in $M_H$ and $s.(1,4)+s.(2,3)$ belongs to $A$. Then
\begin{multline*}
(t.(1,3)+t^\sigma.(2,4))^{\tau_{s.(1,4)+s.(2,3)}}=
\left((t.(1,3))^{s.(1,4)}+((t^\sigma).(2,4))^{s.(1,4)}\right)^{\tau_{s.(2,3)}}\\=
\left(t^{-1}s.(3,4)+s(t^{-1})^\sigma.(1,2)\right)^{\tau_{s.(2,3)}}=
(t^{-1})^s.(2,4)+(t^{-1})^{s\sigma}.(1,3).
\end{multline*}
If $t=(1,2,3)$, then the corresponding double is $(1,2,3).(1,3)+(1,4,2).(2,4)=d_9=z\in A$.
Applying the equality for this $t$ and $s\in\{(1,2)(3,4)$, $(1,3)(2,4)$, $(1,4)(2,3)\}$,
we obtain the following elements of $A$: 
$d_{10}:=d_9^{\tau_{d_5}}=(1,2,4).(2,4)+(1,3,2).(1,3)$,
$d_{11}:=d_9^{\tau_{d_4}}=(1,4,3).(2,4)+(2,3,4).(1,3)$, and $d_{12}:=d_9^{\tau_{d_3}}=(2,3,4).(2,4)+(1,4,3).(1,3)$.
Now we apply $\tau_{s_1}\tau_{s_2}$:
\begin{multline*}
d_{13}:=d_{9}^{\tau_{s_1}\tau_{s_2}}=((1,4,2).(2,4))^{(1,2)(3,4)}+((1,2,3).(1,3))^{(1,2)(3,4)}\\=(1,4,2).(1,3)+(1,2,3).(2,4),
\end{multline*}
$$d_{14}:=d_{10}^{\tau_{s_1}\tau_{s_2}}=(1,3,2).(2,4)+(1,2,4).(1,3).$$
Similarly, we find that doubles
$$d_{15}:=d_9^{\tau_{s_1}\tau_{s_4}}=((1,2,3).(2,3)+(1,4,2).(1,4))^{(1,3)(2,4).(1,2)}=
(2,4,3).(1,3)+(1,3,4).(2,4)$$ 
and $d_{16}:=d_{15}^{\tau_{s_1}\tau_{s_2}}=(2,4,3).(2,4)+(1,3,4).(1,3)$ belong to $A$.
Consequently, we prove that all eight required doubles are in $A$.
As was mentioned above, applying $\tau_x$ to them, we obtain the eight remaining doubles of the second shape: $d_{i+8}:=d_{i}^{\tau_x}$, where $9\leq i\leq 16$.

Since 
$$d_1\cdot d_5=((1,3)+(2,4)(\sigma.(1,4)+\sigma.(2,3))=\eta(d_1+d_5-\sigma.(1,2)-\sigma.(3.4)),$$
we infer that $d_{25}:=\frac{1}{\eta}(\eta(d_1+d_5)-d_1\cdot d_5)=(1,2)(3,4).(1,2)+(1,2)(3,4).(3,4)$ belongs to $A$.

Now we find two elements in $A$ using products $d_2\cdot d_9$ and $d_1\cdot d_{18}$:
\begin{multline*}
q_1:=\frac{2}{\eta}(-d_2\cdot d_9+\eta(d_2+d_9))=\frac{2}{\eta}[-((1,4)+(2,3))\cdot((1,4,2).(2,4)+(1,2,3).(1,3))\\+
\eta((1,4)+(2,3)+(1,4,2).(2,4)+(1,2,3).(1,3))]\\=(1,4,2).(2,1)+(1,2,3).(4,3)+(1,4,2).(3,4)+(1,2,3).(1,2)\\=
(1,2,3).(1,2)+(1,2,4).(1,2)+(1,3,2).(3,4)+(1,4,2).(3,4)
\end{multline*}
and
\begin{multline*}
q_2:=\frac{2}{\eta}(-d_1\cdot d_{18}+\eta(d_1+d_{18}))=\frac{2}{\eta}[-((1,3)+(2,4))\cdot((1,2,4).(1,4)+(1,3,2).(2,3))\\+
\eta((1,3)+(2,4)+(1,2,4).(1,4)+(1,3,2).(2,3))]\\=(1,2,4).(3,4)+(1,3,2).(2,1)+(1,2,4).(1,2)+(1,3,2).(4,3)\\=
(1,2,3).(1,2)+(1,2,4).(1,2)+(1,2,3).(3,4)+(1,2,4).(3,4).
\end{multline*}

Since
\begin{multline*}
q_1\cdot q_2=\newline
((1,2,3).(1,2))^2+2(1,2,3).(1,2)\cdot(1,2,4).(1,2)+((1,2,4).(1,2))^2\\+
(1,2,4).(3,4)\cdot(1,4,2).(3,4) + (1,2,3).(3,4) \cdot(1,3,2).(3,4)\\+(1,3,2).(3,4) \cdot(1,2,4).(3,4)+(1,4,2).(3,4)\cdot (1,2,3).(3,4)\\
=(1+\eta)[(1,2,3).(1,2)+(1,2,4).(1,2)]-\eta (1,2)(3,4).(1,2)\\ +
\frac{\eta}{2}[(1,4,2).(3,4)+(1,2,4).(3,4)+(1,2,3).(3,4)+(1,3,2).(3,4)]-\eta (3,4)
\\=[(1,2,3).(1,2)+(1,2,4).(1,2)]+\frac{\eta}{2}(q_1+q_2)-\eta [(1,2)(3,4).(1,2)+s_2],
\end{multline*}
we infer that 
$$m_1:=q_1\cdot q_2-\frac{\eta}{2}(q_1+q_2)+\eta s_2=(1,2,3).(1,2)+(1,2,4).(1,2)-\eta (1,2)(3,4).(1,2)$$
belongs to $A$.
Applying Miyamoto involutions, we find that the following seven elements lie in $A$:
$$m_2:=m_1^{\tau_a}=[(1,3,2).(1,2)+(1,4,2).(1,2)]-\eta (1,2)(3,4).(1,2),$$
$$m_3:=m_1^{\tau_{s_4}}=[(1,4,3).(1,2)+(2,4,3).(1,2)]-\eta (1,2)(3,4).(1,2),$$
$$m_4:=m_1^{\tau_{s_6}}=[(2,3,4).(1,2)+(1,3,4).(1,2)]-\eta (1,2)(3,4).(1,2),$$
$$m_5:=m_1^{\tau_y}=[(1,2,3).(3,4)+(1,2,4).(3,4)]-\eta (1,2)(3,4).(3,4),$$
$$m_6:=m_2^{\tau_y}=[(1,3,2).(3,4)+(1,4,2).(3,4)]-\eta (1,2)(3,4).(3,4),$$
$$m_7:=m_3^{\tau_y}=[(1,4,3).(3,4)+(2,4,3).(3,4)]-\eta (1,2)(3,4).(3,4),$$
$$m_8:=m_4^{\tau_y}=[(2,3,4).(3,4)+(1,3,4).(3,4)]-\eta (1,2)(3,4).(3,4).$$

We find 39 elements in $A$: $s_1,\ldots, s_6$, $d_1,\ldots, d_{25}$, and $m_1,\ldots, m_8$.
Each of 72 points of the Fischer space except $(1,2)(3,4).(1,2)$ and $(1,2)(3,4).(3,4)$ is a summand in exactly one of these 39 elements. Therefore, these 39 elements are linearly independent.
Recall that $A$ is a subalgebra in the 40-dimensional algebra $M_H$. Denote $w=(1,2)(3,4).(1,2)-(1,2)(3,4).(3,4)$. Now we prove that $A$ lies in the orthogonal complement to $\langle w\rangle$ in $M_H$ with respect to the Frobenius form.
First we find the products of $y$ and $z$ with $t$:
\begin{multline*}
yw=(b+c)\cdot((1,2)(3,4).(1,2)-(1,2)(3,4).(3,4))=\frac{\eta}{2}(b+(1,2)(3,4).(1,2)
\\-(1,2)(3,4).(2,3))-\frac{\eta}{2}(b+(1,2)(3,4).(3,4)-(1,2)(3,4).(1,4))
\\+\frac{\eta}{2}(c+(1,2)(3,4).(1,2)-(1,2)(3,4).(1,4))-\frac{\eta}{2}(c+(1,2)(3,4).(3,4)
\\-(1,2)(3,4).(2,3))=\eta{w};
\end{multline*}
and 
\begin{multline*}
zw=(d+e)\cdot((1,2)(3,4).(1,2)-(1,2)(3,4).(3,4))=\frac{\eta}{2}(d+(1,2)(3,4).(1,2)
\\-(2,4,3).(1,4))-\frac{\eta}{2}(d+(1,2)(3,4).(3,4)-(1,3,4).(2,3) ])
\\+\frac{\eta}{2}(e+(1,2)(3,4).(1,2)-(1,3,4).(2,3)])-\frac{\eta}{2}(e+(1,2)(3,4).(3,4)
\\-(2,4,3).(1,4))=\eta{w}.
\end{multline*}
Now we prove that products $x_1\cdot x_2\cdots x_n$, where $n\geq1$ and $x_i\in\{x,y,z\}$, are orthogonal to $w$ by induction on $n$. 
Since $xw=0$, we have $(x,w)=0$. We know that $yw=zw=\eta{w}$, so 
$(y,w)=(y^2,w)=(y,yw)=\eta(y,w)$ and hence $(y,w)=0$. Similarly, we get that $(z,w)=0$. To finish the induction 
we verify that if $(t,w)=0$, then elements $xt$, $yt$, and $zt$ are orthogonal to $w$. Clearly, $(xt,w)=(t,xw)=0$. Now we see that $(yt,w)=(y,tw)=(yw,t)=(\eta{w},t)=\eta(w,t)=0$
and $(zt,w)=(z,tw)=(zw,t)=(\eta{w},t)=\eta(w,t)=0$.
So we get that all elements of $A$ are orthogonal to $w$. Since $(w,w)=2\neq0$,
the orthogonal complement $w^{\perp}$ is 39-dimensional and hence the dimension of $A$ is not greater than 39. On the other hand, we know that $A$ is at least 39-dimensional. So $A$ is 39-dimensional and
elements $s_1,\ldots, s_6$, $d_1,\ldots, d_{25}$, and $m_1,\ldots, m_8$ form a basis of $A$.

Using GAP, we find that the determinant of the Gram matrix for $A$ with respect to the basis equals 
$2^{49}\cdot(\eta-\frac{1}{2})^6\cdot(\eta-\frac{1}{8})^2\cdot(7\eta+\frac{1}{4})\cdot(\eta+\frac{1}{4})$. Hence the algebra is simple unless $\eta=\frac{1}{8}$, $-\frac{1}{28}$, $-\frac{1}{4}$, or $\operatorname{char}\mathbb{F}\neq7$ and $\eta=-\frac{1}{28}$ .

Finally, note that in this case the flip subalgebra $A(\tau)$ includes $x$, $y$, $z$, and $(1,2)(3,4).(1,2)$, so it coincides with whole $M_H$. However, one can verify that $A$ coincides with subalgebra of $M_H$ generated by all doubles.

\vspace{10px}
We further assume the set of relations $R_2= R_1\cup\{(b^{(acd)^2}e)^3\}$.

\textbf{Case 4.} Assume that $(b^{(acd)^{-2}}e)^2=1$. 
If $a$, $b$, $c$, and $d$ are elements of $H$ as above, then $b^{(acd)^2}=c_{1,3}$ and $b^{(acd)^{-2}}=c_{3,1}$.
Therefore, we infer that $|ec_{1,3}|=3$ and $|ec_{3,1}|=2$. 
Denote as in the previous case $B=a$, $C=d$, $D=c$, $E=e$, and set $A=c_{3,1}$.
Then $|AB|=3$, $|AC|=|AD|=|AE|=3$, $|BC|=|BD|=|BE|=|CD|=3$, $|CD|=3$, and $|CE|=|DE|=2$.
Moreover, $|B^CD|=|c_{1,4}b_{2,4}|=3$. Now $A^{(BCD)^2}=c_{3,1}^{(b_{1,2}c_{2,4}b_{2,4})^2}=c_{2,3}^{b_{1,2}c_{2,4}b_{2,4}}=c_{1,3}$ and hence $|EA^{(BCD)^2}|=|ec_{1,3}|=3$. Therefore, as in the previous case, it follows from the presentation A17b from \cite{hs} that $G$ is a homomorphic image of $F(5,72)$.
Inspecting the homomorphic images of $F(5,72)$, we find that $G\simeq F(5,72)$ or $G\simeq F(5,72)/Z(F(5,72))$. Since $F(5,72)$ has the same central type as $Wr(A_4,4)$,
we can assume that $G=Wr(A_4,4)$. Identify the elements of $D_7$-diagram with the following points of $\Gamma(G)$:
$$a:=(1,2),b:=(1,3),c:=(2,4),d:=(1,4,2).(2,4),e:=(1,3,2).(1,3)$$
Then $a^c=(1,4)$ and $a^b=(2,3)$, so $|a^cd|=|a^be|=3$.
As in the previous case, we have $b^{(acd)^2}=(1,4,2).(1,3)$.
Since $(1,4,2)(1,3,2)^{-1}=(1,4,3)$, Lemma~\ref{l:lines} implies that $|b^{(acd)^2}e|=3$.
Finally, we find that 
\begin{multline*}
b^{(acd)^{-2}}=b^{(dca)^2}=(1,3)^{[((1,4,2).(2,4))(2,4)(1,2)]^2}=(2,3)^{[((1,4,2).(2,4))(2,4)(1,2)]}\\
=((1,4,2).(3,4))^{(2,4)(1,2)}=(1,4,2).(3,1)=(1,2,4).(1,3)
\end{multline*}
and hence $|eb^{(acd)^{-2}}|=2$. Therefore, elements $a$, $b$, $c$, $d$, and $e$ satisfy all relations from $R_2$. We use them for calculations in $A$.
Consider the involutive automorphism $\tau$ of $G$ as in Proposition~\ref{p:fixed2:Wr(A_4,n)}. 
If $\pi_1\in A_4$ and $\pi_2$ is a transposition in $S_4$, then $(\pi_1.\pi_2)^{\tau}=\pi_1^{(3,4)}.\pi^{(1,2)(3,4)}$. Therefore, $a^\tau=a$, $b^\tau=c$, and $d^\tau=e$.
So $A$ is a subalgebra in the 42-dimensional fixed algebra $M_H$. We claim that $A$ coincides with $M_H$.
By Proposition~\ref{p:fixed2:Wr(A_4,n)}, $M_H$ has a basis comprising of 12 singles and 30 double with respect to 
$\tau$. We prove that all of them belong to $A$.
Starting with $a$ and applying Miyaomoto involutions, we find that all singles belongs to $A$:
\begin{center}
$\begin{aligned}
s_1:=&\,a=(1,2), \quad s_2:=s_1^{\tau_y}=(3,4),\\
s_3:=&\,s_1^{\tau_z}=(1,2)^{((1,4,2).(2,4))((1,3,2).(1,3))}=((1,4,2).(1.4))^{(1,3,2).(1,3)}=(2,3,4).(3,4),\\ 
s_4:=&\,s_3^{\tau_{s_2}}=((2,3,4).(3,4))^{(3,4)}=(2,4,3).(3,4),\\
s_5:=&\,s_3^{\tau_y}=((2,3,4).(3,4))^{(1,3)(2,4)}=(2,3,4).(1,2),\\ 
s_6:=&\,s_5^{\tau_x}=((2,3,4).(1,2))^{(1,2)}=(2,4,3).(1,2),\\
s_7:=&\,s_2^{\tau_z}=(3,4)^{((1,4,2).(2,4))((1,3,2).(1,3))}=((1,4,2).(2,3))^{(1,3,2).(1,3)}=(1,3,4).(1,2),\\
s_8:=&\,s_7^{\tau_x}=(1,4,3).(1,2),\\
s_9:=&\,s_7^{\tau_y}=(1,3,4).(3,4),\\
s_{10}:=&\,s_9^{\tau_{s_2}}=(1,4,3).(3,4),\\
s_{11}:=&\,s_6^{\tau_z}=((2,4,3).(1,2))^{((1,4,2).(2,4))((1,3,2).(1,3))}=((1,4,3).(1,4))^{(1,3,2).(1,3)}=(1,2)(3,4).(3,4)\\
s_{12}:=&\,s_{11}^{\tau_y}=(1,2)(3,4).(1,2).
\end{aligned}$
\end{center}

Consider a double $s.(2,4)+s^{(3,4)}.(1,3)$ and a single $t.(1,2)$, where $s,t\in A_4$.
Then
$$(s.(2,4)+s^{(3,4)}.(1,3))^{\tau_{t.(1,2)}}=(ts).(1,4)+(t^{-1}s^{(3,4)}).(2,3).$$
Applying this formula for $y=(2,4)+(1,3)$ and different singles, we obtain the following doubles:
\begin{center}
$\begin{aligned}
d_1:=&\,y^{\tau_x}=(1,4)+(2,3),\\
d_2:=&\,y^{\tau_{s_5}}=(2,3,4).(1,4)+(2,4,3).(2,3),\\
d_3:=&\,y^{\tau_{s_6}}=(2,4,3).(1,4)+(2,3,4).(2,3),\\
d_4:=&\,y^{\tau_{s_7}}=(1,3,4).(1,4)+(1,4,3).(2,3),\\
d_5:=&\,y^{\tau_{s_8}}=(1,4,3).(1,4)+(1,3,4).(2,3),\\
d_6:=&\,y^{\tau_{s_{12}}}=(1,2)(3,4).(1,4)+(1,2)(3,4).(2,3).
\end{aligned}$
\end{center}
Now we use $z=(1,4,2).(2,4)+(1,3,2).(1,3)$ instead of $y$ and get the following elements:
\begin{center}
$\begin{aligned}
d_7:=&\,z^{\tau_x}=(1,4,2).(1,4)+(1,3,2).(2,3),\\
d_8:=&\,z^{\tau_{s_5}}=((2,3,4)(1,4,2)).(1,4)+((2,4,3)(1,3,2)).(2,3)\\&\,=(1,4)(3,2).(1,4)+(1,3)(2,4).(2,3),\\
d_9:=&\,z^{\tau_{s_7}}=((1,3,4)(1,4,2)).(1,4)+((1,4,3)(1,3,2)).(2,3)\\&\,=(1,3,2).(1,4)+(1,4,2).(2,3).\\
\end{aligned}$
\end{center}

Applying $\tau_{d_1}$, we see that
\begin{multline*}
d_{10}:=d_7^{\tau_{d_1}}=((1,4,2).(1,4))^{(1,4)(2,3)}+((1,3,2).(2,3))^{(1,4)(2,3)}
\\=(1,2,4).(1,4)+(1,2,3).(2,3)
\end{multline*}
and
\begin{multline*}
d_{11}:=d_9^{\tau_{d_1}}=((1,3,2).(1,4))^{(1,4)(2,3)}+((1,4,2).(2,3))^{(1,4)(2,3)}
\\=(1,2,3).(1,4)+(1,2,4).(2,3).
\end{multline*}
belong to $A$. Finally, we find that 
$$d_{12}:=d_8^{\tau_{s_1}\tau_{s_2}}=(1,4)(2,3).(2,3)+(1,3)(2,4).(1,4)
$$
lies in $A$. Therefore, all 12 doubles of the shape $t.(1,4)+t^{(3,4)}.(2,3)$, where $t\in A_4$,
are in $A$. Applying $\tau_x$ to them, we obtain all 12 doubles of the shape $t.(2,4)+t^{(3,4)}.(1,3)$:
$d_{i+12}:=d_i^{\tau_x}$, where $i\in\{1,\ldots,12\}$.
It remains to prove that six doubles of shape $t.s+(t^{-1})^{(3,4)}.s$,
where $$t\in\{(1,2,3),(1,3,2),(1,3)(2,4)\} \text{ and } s\in\{(1,2),(3,4)\},$$ belong to $A$. 

Consider two doubles $d_s=s.(2,4)+s^{(3,4)}.(1,3)$ and $d_t=t.(1,4)+t^{(3,4)}.(2,3)$, where $s,t\in A_4$.
Using Lemma~\ref{l:3trans}, we find that their product equals 
$$\eta(d_t+d_s)-\frac{\eta}{2}(ts^{-1}.(1,2)+(st^{-1})^{(3,4)}.(1,2)+
(t^{-1})^{(3,4)}s.(3,4)+(s^{-1})^{(3,4)}t.(3,4)).$$
Substituting $s=()$, $t=(1,3)(2,4)$ and $s=(2,4,3), t=(1,3,4)$, 
we infer that elements 
$$q_1:=(1,3)(2,4).(1,2)+(1,4)(2,3).(1,2)+(1,3)(2,4).(3,4)+(1,4)(2,3).(3,4)$$ 
$$\text{and }q_2:=(1,3)(2,4).(1,2)+(1,4)(2,3).(1,2)+(1,2,4).(3,4)+(1,3,2).(3,4)$$ lie in A.
Then
\begin{multline*}
q_1\cdot q_2=((1,3)(2,4).(1,2))^2+((1,4)(2,3).(1,2))^2+(1,3)(2,4).(3,4)\cdot(1,2,4).(3,4)
\\+(1,3)(2,4).(3,4)\cdot(1,3,2).(2,4)+(1,4)(2,3).(3,4)\cdot(1,2,4).(3,4)
\\+(1,3)(2,4).(3,4)\cdot(1,3,2).(2,4)=(1,3)(2,4).(1,2)+(1,4)(2,3).(1,2)
\\+\eta[(1,3)(2,4).(3,4)+(1,4)(2,3).(3,4)+(1,2,4).(3,4)+(1,3,2).(2,4)]\\
-\eta[(2,4,3).(3,4)+(1,3,4).(3,4)].
\end{multline*}
Denote $d_{25}:=(1,3)(2,4).(1,2)+(1,4)(2,3).(1,2)$.
Then 
$$q_1\cdot q_2=d_{25}+\eta(q_1-d_{25}+q_2-d_{25})-\eta(s_4+s_9)=(1-2\eta)d_{25}+\eta(q_1+q_2-s_4-s_9)$$ 
and hence $(1-2\eta)d_{25}\in A$.
Since $\eta\neq1/2$, we infer that $d_{25}\in A$.
Therefore, $d_{26}:=d_{25}^{\tau_{s_5}}=(1,4,2).(1,2)+(1,2,3).(1,2)$
and $d_{27}:=d_{26}^{\tau_{x}}=(1,2,4).(1,2)+(1,3,2).(1,2)$ belong to $A$.
Finally, applying $\tau_y$ to $d_{25}$, $d_{26}$, and $d_{27}$, we obtain the remaining three doubles of $M_H$. Thus, all singles and doubles of $M_H$ lie in $A$, and hence $A=M_H$, as claimed.

Using GAP, we find that the determinant of the Gram matrix for $A$ with respect to the basis $s_1,\ldots,s_{12},d_1,\ldots, d_{30}$ equals 
$$ 2^{48}\cdot(\eta-\frac{1}{2})^8\cdot(\eta-\frac{1}{8})^2\cdot(7\eta+\frac{1}{4})\cdot(\eta+\frac{1}{4})
.$$ 
Hence the algebra is simple unless $\eta=\frac{1}{8}$, $-\frac{1}{4}$, or $\operatorname{char}\mathbb{F}\neq7$ and $\eta=-\frac{1}{28}$,.

\vspace{15px}
Further assume that elements $a$, $b$, $c$, $d$, and $e$ satisfy the set of relations $R_3:=R_2\cup\{(b^{(acd)^{-2}}e)^3\}$.
 
\textbf{Case 5.} Assume $(a^{(bacda)^2}e)^2=1$. Then $G$ is a homomorphic image of $\hat G / \langle R_3 \cup \{ (a^{(bacda)^2}e)^2 \} \rangle$. Computations show that the last group is of order~2.

Therefore, we can assume that $a$, $b$, $c$, $d$, and $e$ satisfy the set of relations $R_4:=R_3\cup\{(a^{(bacda)^2}e)^3\} $.

\textbf{Case~6.} Assume $(b^{a(acd)^2}e)^2=1$. Similarly to the previous case, the set of relations $R_4\cup\{ (b^{a(acd)^2}e)^2\}$ defines a group of order $2$, which is not possible.

\textbf{Case~7.} Therefore, we can assume that $a$, $b$, $c$, $d$, and $e$ satisfy the set of relations $R_5:= R_4\cup\{(b^{a(acd))^2}e)^3\}$. Then $G$ is a homomorphic image of $\hat G / \langle R_5 \rangle$. 
Denote $A:=b, B:=a, C:=c, D:=D, E:=e.$
Since $|ab|=|ac|=|ad|=|ae|=3$ and $|cd|=|be|=3$, we have  
$|BA|=|AC|=|CD|=|DA|=|AE|=|EB|=3$.
The set $R_5$ includes the following relations: $$(a^cd)^3=(a^be)^3=(b^{(acd)^2}e)^3=(b^{(acd)^{-2}}e)^3=(a^{(bacda)^2}e)^3=(b^{a(acd)^2}e)^3=1.$$
Therefore, it is true that
\begin{multline*}
(B^CD)^3=(B^AE)^3=(A^{(BCD)^2}E)^3=(A^{(BCD)^{-2}}E)^3
\\=(B^{(ABCDB)^2}E)^3=(A^{B(BCD)^2}E)^3=1.
\end{multline*}

Now the presentation A18a of $\cite{hs}$ implies that 
$G$ is a homomorphic image of $F(5,162)$, namely 
the index of $\hat G / \langle R_5 \rangle$ by $H=\langle a,b,c,d\rangle$ equals 19683. 
According to \cite[Example~18]{hs},
the group F(5,162) has the elementary abelian center of order $3^3$ and the same central type as $Wr(3^{1+2},4)$.
Moreover, it is mentioned after \cite[Example~19]{hs} that any 
quotient of $F(5,162)$, which requires five transposition generators, has central type $F(5,162)$ or $F(5,54)$. In this section we consider the first type and section Case~8
is devoted to the second one.

We use the description of $\Gamma(Wr(3^{1+2},4))$ from Subsection~\ref{s:Wr(3^1+2,n)}. 
So we suppose that $u$ and $v$ generate an extraspecial group $P$ of order $3^3$ and of exponent 3. Denote $w=[u,v]=u^{-1}v^{-1}uv$. Then the following equalities are true:
$$u^3=v^3=w^3=1, wu=uw, \text{ and } wv=vw.$$ 
In particular, each element of $P$ 
can be uniquely written as a product
$u^iv^jw^k$, where $1\leq i,j,k\leq3$. By Lemma~\ref{l:3trans}, each point of the Fischer space has form $u^iv^jw^k.(s,t)$, where $u^iv^jw^k\in P$ and $1\leq s<t\leq4$. Identify the elements of $D_7$-diagram with the following group elements. 
$$a:=(1,2),b:=(1,3),c:=(2,4),d:=v.(2,4),e:=u.(1,3).$$

It is easy to see that $|ab|=|ac|=|ad|=|ae|=|cd|=|be|=3$ and
$|bc|=|de|=|bd|=|ce|=2$.
Now we verify the relations of $R_5$:
$|a^cd|=|(1,4)(v.(2.4))|=3$, $|a^be|=|(2,3)(u.(1,3))|=3$.
Observe that $|b^{(acd)^2}e|=|b^{acda}e^{cd}|=|b^{acda}e|$.
Applying Lemma~\ref{l:lines}, we find that $b^{acda}=(2,3)^{cda}=(3,4)^{da}=(v.(2,3))^a=v.(1,3)$
and hence $|b^{(acd)^2}e|=3$.
Since 
$$|b^{(acd)^{-2}}e|=|b^{(dca)^2}e|=|b^{dacada}e|=|b^{dacdad}e|=|b^{dacda}e|=
|b^{adaca}e|=|b^{adcac}e|=|b^{adca}e|$$
and $b^{adca}=(2,3)^{dca}=(v.(3,4))^{ca}=(v^2.(2,3))^a=v^2.(1.3)$, we infer that $|b^{(acd)^{-2}}e|=3$.
Using the transpositions "on top" of $a$, $b$, $c$, $d$, and $e$, we can find the transpositions "on top" of $a^{(bacda)^2}$ and $b^{a(acd)^2}$,
namely, they equal to $(1,2)^{((1,3)(1,2)(2,4)(2,4)(1,2))^2}=(1,2)$ and $(1,3)^{(1,2)((1,2)(2,4)(2,4))^2}=(2,3)$, respectively.
By Lemma~\ref{l:lines}, we infer that $|a^{(bacda)^2}e|=|b^{a(acd)^2}e|=3$.
Therefore, the elements $a$, $b$, $c$, $d$, and $e$ satisfy all the relations of this case. Since they generate $Wr(3^{1+2},4)$,
we can use them for calculations in $A$.

Consider the automorphism $\tau$ of $G$
as in Proposition~\ref{p:fixed:Wr(3^1+2,n)}. It induces an automorphism of $\Gamma(G)$ as follows:
$(u^iv^jw^k.(s,t))^\tau=v^iu^jw^{-k}.(s,t)^\sigma=u^jv^iw^{-k-ij}.(s,t)^\sigma$. 
Then $a^\tau=a$, $b^\tau=c$, and $d^\tau=e$.
Denote $H=\langle\tau\rangle$. Therefore, $A$ is a subalgebra of 90-dimensional algebra $M_H$. We claim that $A=M_H$ of $\eta\neq2$ and prove this by showing that all elements of the basis from Proposition~\ref{p:fixed:Wr(3^1+2,n)} belong to $A$.

Suppose that $g.(1,2)$ is a single in $A$. Then 
$$(g.(1,2))^{\tau_y\tau_z}=(g.(3,4))^{{\tau_z}}=(g.(3,4))^{(v.(2,4))(u.(1,3))}=(vg^{-1}.(2,3))^{u.(1,3)}=ugv^{-1}.(1,2).$$
Denote $s_1:=x=(1,2)$. Then 
$s_2:=s_1^{\tau_y\tau_z}=uv^2.(1,2)$ and $s_3:=s_2^{\tau_y\tau_z}=u^2v.(1,2)$ belong to $A$.
Now 
$s_4:=s_2^{\tau_x}=uv^2.(2,1)=vu^2.(1,2)=u^2vw.(1,2)$
and $s_5:=s_3^{\tau_x}=u^2v.(2,1)=v^2u.(1,2)=uv^2w.(1,2)$
are elements of $A$. 
Using $s_4$, we find the following elements:
$s_6:=s_4^{\tau_y\tau_z}=w.(1,2)$, 
$s_7:=s_6^{\tau_x}=w^2.(1,2)$,
$s_8:=s_7^{\tau_y\tau_z}=uv^2w^2.(1,2)$, and $s_9:=s_8^{\tau_y\tau_z}=u^2vw^2.(1,2)$.
Applying $\tau_y$ to these nine singles,
we find the remaining nine singles of $M_H$: 
$s_{i+9}:=s_i^{\tau_y}$, where $i\in\{1,\ldots,9\}$.

If $g.(1,2)$ is a single, then $y^{\tau_{g.(1,2)}}=
g^{-1}.(2,3)+g.(1,4)$.
Denote $d_i:=y^{\tau_{s_i}}$, where $i\in\{1,\ldots,9\}$.
Therefore, all nine doubles of $M_H$ of shape 
$u^iv^jw^k.(1,4)+u^{2j}v^{2i}w^{2k-i-j}.(2,3)$, where $i+j\equiv0\pmod{3}$ and $0\leq k\leq 2$, are elements of $A$.
Similarly, if $g.(1,2)$ is a single, then $z^{\tau_{g.(1,2)}}=
gv.(1,4)+g^{-1}u.(2,3)$.
Denote $d_{i+9}:=z^{\tau_{s_i}}$, where $i\in\{1,\ldots,9\}$.
Therefore, all nine doubles of the shape 
$u^iv^jw^k.(1,4)+u^{2j}v^{2i}w^{2k-ij}.(2,3)$, where $i+j\equiv1\pmod{3}$ and $0\leq k\leq 2$, are elements of $A$.
Observe that 
$$z^{\tau_y}=(v.(2,4))^{((1,3)(2,4)}+(u.(1,3))^{(1,3)(2,4)}=v^2.(2,4)+u^2.(1,3).$$
Then $d_{i+18}:=z^{\tau_y\tau_{s_i}}$, where $i\in\{1,\ldots, 9\}$, are all nine doubles of $M_H$ of the shape 
$u^iv^jw^k.(1,4)+u^{2j}v^{2i}w^{2k-ij}.(2,3)$ with $i+j\equiv2\pmod{3}$ and $0\leq k\leq 2$.
Applying $\tau_x$ to doubles $d_1,\ldots, d_{27}$,
we get the remaining 27 doubles of $M_H$: $d_{i+{27}}:=d_i^{\tau_x}$, where $i\in\{1,\ldots,27\}$.

It remains to prove that 18 extras of shapes
$u^iv^jw^k.(1,2)+u^{2j}v^{2i}w^k.(1,2)$ and
$u^iv^jw^k.(3,4)+u^{2j}v^{2i}w^k.(3,4)$, where 
$i+j\not\equiv0\pmod{3}$ and $0\leq k\leq 2$, belong to $A$.

Since 
\begin{multline*}
y\cdot d_{11}=((1,3)+(2,4))\cdot(u.(1,4)+v.(2,3))\\=\eta(y+d_{11})-\frac{\eta}{2}(u.(3,4)+v^2.(3,4)+v^2.(1,2)+u.(1,2)),
\end{multline*}
we infer that
$q_1:=\frac{2}{\eta}(\eta(y+d_{11})-y\cdot d_{11})$ belong to $A$.
Then 
$$q_2:=q_1^{\tau_x}=u.(3,4)+v^2.(3,4)+v.(1,2)+u^2.(1.2)$$
and
\begin{multline*}
q_3:=q_2^{\tau_{s_9}}=u.(3,4)+v^2.(3,4)+(v.(1,2))^{u^2vw^2.(1,2)}+(u^2.(1,2))^{u^2vw^2.(1,2)}\\=
u.(3,4)+v^2.(3,4)+u^2vw^2v^2u^2vw^2.(1,2)+
u^2vw^2uu^2vw^2.(1,2)
\\=u.(3,4)+v^2.(3,4)+uvw.(1,2)+
u^2v^2w.(1,2)
\end{multline*}
are elements of $A$.
Now we find that
\begin{multline*}
q_1\cdot q_2=(u.(3,4)+v^2.(3,4))^2+(v^2.(1,2)+u.(1,2))\cdot(v.(1,2)+u^2.(1,2))
\\=u.(3,4)+v^2.(3,4)+2\cdot\left(u.(3,4)\cdot v^2.(3,4)\right)+
\eta\cdot\left(v^2.(1,2)+u.(1,2)+v.(1,2)+u^2.(1,2)\right)
\\-\frac{\eta}{2}\cdot\left((v^2.(1,2))^{v.(1,2)}+(u.(1,2))^{v.(1,2)}+
(v^2.(1,2))^{u^2.(1,2)}+(u.(1,2))^{u^2.(1,2)}\right)
\\=(1+\eta)\cdot\left(u.(3,4)+v^2.(3,4)\right)-\eta\cdot u^2vw^2.(3,4)
\\+\eta\cdot\left(v^2.(1,2)+u.(1,2)+v.(1,2)+u^2.(1,2)\right)-
\eta\cdot(1,2)-\frac{\eta}{2}\cdot(uvw.(1,2)+u^2v^2w.(1,2)).
\end{multline*}
Denote $d_{55}:=u.(3,4)+v^2.(3,4)$.
Then 
\begin{multline*}
q_1\cdot q_2=(1+\eta)d_{55}-\eta s_1-\eta s_{18}+\eta(q_1-d_{55}+q_2-d_{55})-\frac{\eta}{2}(q_3-d_{55})\\=
(1-\frac{\eta}{2})d_{55}+\eta(q_1+q_2-\frac{1}{2}q_3-s_1-s_{18}).
\end{multline*}
Since $\eta\neq2$, we infer that $d_{55}\in A$.
Applying $\tau_{s_i}$ with $10\leq18$ to $d_{55}$, we get nine extras with (3,4) "on top", 
including $d_{55}$ for $s_{18}=u^2vw^2.(3,4)$:  $d_{55}^{\tau_{s_{18}}}=
u.(3,4)^{s_{18}}+v^2.(3,4)^{s_{18}}=((u^2vw^2)u^2(u^2vw^2)).(3,4)+((u^2vw^2)v(u^2vw^2)).(3,4)=v^2.(3,4)+u.(3,4)=d_{55}$. Applying $\tau_y$ to these nine extras, we obtain the remaining nine extras with $(1,2)$ "on top". Thus, $A$ coincides with $M_H$ and as a consequence we find that the flip subalgebra 
$A(\tau)$ coincides with $M_H$.

Using GAP, we find that the determinant of the Gram matrix for $A$ with respect to the basis equals 
$$\frac{1}{2^{10}}(\eta-2)^{86}\cdot(14\eta-1)^2\cdot(67\eta+1)\cdot(13\eta+1).
$$

Suppose that $\eta=2$. In this case we show that $A$ has dimension 89. We know that 
$s_1,\ldots, s_{18}$ and $d_1,\ldots, d_{54}$ belong to $A$. Denote 18 extras with 
$(1,2)$ "on top" by $e_1,\ldots, e_9$ and 18 extras with 
$(3,4)$ "on top" by  $f_1,\ldots,f_9$. Now set $g_{ij}=e_i+f_j$, where $1\leq i,j\leq 9$.
We claim that $A$ is the span of $s_1,\ldots, s_{18}$, $d_1,\ldots, d_{54}$, and $g_{ij}$ with 
$1\leq i,j\leq 9$. To show this, we need to verify that this set is closed under multiplication.
Consider two doubles $d'=s.(1,3)+s^\sigma.(2,4)$ and $d''=t.(1,4)+t^\sigma.(2,3)$ of different kinds.
Then their product is equal to $$\eta(d'+d'')-\frac{\eta}{2}\left(s^{-1}t.(3,4)+t^{-\sigma}s^\sigma.(3,4)+st^{-\sigma}.(1,2)+ts^{-\sigma}.(1,2)\right).$$
 Note that $s^{-1}t.(3,4)+t^{-\sigma}s^\sigma.(3,4)\in\{ f_i~|~1\leq i\leq 9\}$ and 
 $st^{-\sigma}.(1,2)+ts^{-\sigma}.(1,2)\in\{e_i~|~1\leq i\leq 9\}$. Other products of single and doubles can be treated similarly. Now we deal with products involving elements $g_{ij}$. 
Applying $\tau_{s_i}$ with $1\leq i\leq 9$ and $10\leq i\leq 18$, we can transform any $g_{ij}$ to $q_1$. First, find $q_1^2$:
\begin{multline*}
q_1^2=q_1+2\cdot\left( u.(3,4)\cdot v^2.(3,4)+v^2.(1,2)\cdot u.(1.2) \right)\\=(1+\eta)\cdot q_1-\eta\cdot(u^2vw^2.(3,4)+u^2vw^2.(1,2))=(1+\eta)q_1-\eta(s_9+s_{18}).
\end{multline*}

Multiply now $q_1$ on $q=s.(3,4)+s^{-\sigma}.(3,4)+t.(1,2)+t^{-\sigma}.(1,2)\in\{g_{i,j}~|~1\leq i,j\leq 9\}$. If $q_1$ and $q$ have common summands, then $q_1\cdot q$ can be calculated as $q_1\cdot q_2$ above. So we can assume that all summands in $q_1$ and $q$ are distinct.
Note that $$q_1\cdot q=\left(v^2.(1,2)\cdot u.(1.2)\right)\cdot\left(t.(1,2)+t^{-\sigma}.(1,2)\right)+\left(v^2.(3,4)\cdot u.(3,4)\right)\cdot\left(s.(3,4)+s^{-\sigma}.(3,4)\right).$$
Let $t=u^iv^jw^k$. Then $t.(1,2)+t^{-\sigma}.(1,2)=u^iv^jw^k.(1,2)+u^{2j}v^{2i}w^k.(1,2)$.
We can assume that $i+j\equiv1\pmod 3$. Then 
\begin{multline*}
\left(v^2.(1,2)\cdot u.(1.2)\right)\cdot\left(t.(1,2)+t^{-\sigma}.(1,2)\right)=
\eta\left(v^2.(1,2)\cdot u.(1.2)+t.(1,2)+t^{-\sigma}.(1,2)\right)\\-
\frac{\eta}{2}\left(v^2t^{-1}v^2.(1,2)+v^2t^{\sigma}v^2.(1,2)+ut^{-1}u.(1,2)+ut^{\sigma}u.(1,2)\right)\\=\eta\left(v^2.(1,2)\cdot u.(1.2)+t.(1,2)+t^{-\sigma}.(1,2)\right)
-\frac{\eta}{2}( u^{-i}v^{1-j}w^{-k-i(1+j)}.(1,2)\\+u^jv^{i+1}w^{-k+j^2}.(1,2)+u^{-i-1}v^{-j}w^{-k+j^2}.(1,2)+u^{j-1}v^iw^{-k-i(j+1)}.(1.2))
\end{multline*}
Note that $-i+1-j\equiv j-1+i\equiv0\pmod3$ and hence $u^{-i}v^{1-j}w^{-k-i(1+j)}.(1,2)$ and 
$u^{j-1}v^iw^{-k-i(j+1)}.(1.2)$ are singles. On the other hand, we see that 
$u^jv^{i+1}w^{-k+j^2}.(1,2)+u^{-i-1}v^{-j}w^{-k+j^2}.(1,2)$ belongs to $\{e_n~|~1\leq n\leq 9\}$.
Similarly, we get that the product $\left(v^2.(3,4)\cdot u.(3,4)\right)\cdot\left(s.(3,4)+s^{-\sigma}.(3,4)\right)$ is written in terms of its factors, two singles, and an element from
$\{f_n~|~1\leq n\leq 9\}$. Therefore, $q_1\cdot q$ lies in the space.
Products of $q_1$ with singles and doubles can be found easily, we skip these calculations.
Finally, observe that $\langle f_1+e_1, f_1+e_2,\ldots f_1+e_9, e_1+f_2, e_1+f_3,\ldots, e_1+f_9 \rangle=\langle\{g_{ij}~|~1\leq i,j\leq 9\}\rangle$ and hence $A$ has dimension 89 and as a basis one can take 18 singles of $M_H$, 54 doubles of $M_H$ and these 17 elements of $\{g_{ij}~|~1\leq i,j\leq 9\}$.

\textbf{Case~8.} Now suppose that relations for $G$ are the same as in the previous case, but $G$ has the same central type as $F(5,54)$. According to \cite{hs}, the group $F(5,54)$ has the same central type as $Wr(3^2,4)$. We use the description of the Fischer space of $Wr(3^2,4)$ from Subsection~\ref{s:Wr(3^2,n)}. Fix generators $u$ and $v$ of the elementary abelian group $P$ of order nine. Each point of the Fischer space has form $u^rv^s.(i,j)$,
where $0\leq r,s\leq 2$ and $(i,j)\in S_4$.

We identify the elements of $D_7$-diagram with the following group elements: 
$$a:=(1,2),b:=(1,3),c:=(2,4),d:=v.(2,4), \text{ and }e:=u.(1,3)$$

It is easy to see that they satisfy the relations of the diagram.
The relations in $R_5$ can be verified in the same way as in the previous case. Since $a$, $b$, $c$, $d$, and $e$ generate $Wr(3^2,4)$,
we can use them for calculations in $A$.

Consider the automorphism $\tau$ and permutation $\pi=(1,2)(3,4)$
as in Proposition~\ref{p:fixed:Wr(3^2,n)}: then 
$(u^sv^r.(i,j))^\tau=v^su^r.(i,j)^\pi=u^rv^s.(i^\pi,t^\pi)$. 
Hence $a^\tau=a$, $b^\tau=c$, and $d^\tau=e$.
Denote $H=\langle\tau\rangle$. Then $A$ is a subalgebra of 30-dimensional algebra $M_H$. We claim that $A=M_H$ if $\eta\neq2$ and prove this by showing that all elements of the basis from Proposition~\ref{p:fixed:Wr(3^2,n)} belong to $A$.

Denote $s_1:=x=(1,2)$. Then $s_2:=s_1^{\tau_y}=(3,4)$ and $s_3:=s_1^{\tau_z}=(1,2)^{(v.(2,4))(u.(1,3))}=(v.(1,4))^{u.(1,3)}=u^2v.(3,4)$ belong to $A$.
Now we find that $s_4:=s_3^{\tau_{s_2}}=uv^2.(3,4)\in A$. This implies that
$s_5:=s_4^{\tau_y}=uv^2.(1,2)$ and $s_6:=s_4^{\tau_z}=(uv^2.(3,4))^{(v.(2,4))(u.(1,3))}=(u^2v^2.(2,3))^{u.(1,3)}=u^2v.(1,2)$
are elements of $A$. Therefore, all six singles $s_1,\ldots, s_6$ of $M_H$ belong to $A$.

Applying the Miyamoto involutions of singles to $y$, we obtain doubles 
\begin{multline*}
d_1:=y^{\tau_x}=(2,3)+(1,4), d_2:=y^{\tau_{s_5}}=(1,3)^{uv^2.(1,2)}+(2,4)^{uv^2.(1,2)}=u^2v.(2,3)+uv^2.(1,4),\\\text{and }d_3:=y^{\tau_{s_6}}=(1,3)^{u^2v.(1,2)}+(2,4)^{u^2v.(1,2)}=uv^2.(2,3)+u^2v.(1,4).
\end{multline*}
Now we use $z$ instead of $y$ and find three other doubles in $A$:
\begin{multline*}
d_4:=z^{\tau_x}=u.(2,3)+v.(1,4), d_5:=z^{\tau_{s_5}}=v.(2,3)+u.(1,4),\\\text{and }d_6:=z^{\tau_{s_6}}=u^2v^2.(2,3)+u^2v^2.(1,4).
\end{multline*}
Finally, using $z^{\tau_y}$, we get the remaining three doubles with $(1,4)$ and $(2,3)$: 
\begin{multline*}
d_7:=z^{\tau_y\tau_x}=(u^2.(1,3)+v^2.(2,4))^{\tau_x}=u^2.(2,3)+v^2.(1,4), d_8:=z^{\tau_y\tau_{s_5}}=uv.(2,3)+uv.(1,4),\\\text{and }d_9:=z^{\tau_y\tau_{s_6}}=v^2.(2,3)+u^2.(1,4).
\end{multline*}
Applying $\tau_x$ to $d_1,\ldots,d_9$, we obtain the remaining nine doubles of $M_H$:
$d_{i+9}:=d_i^{\tau_x}$, where $i\in\{1,\ldots,9\}$.

It remains to show that six extras of $M_H$ belong to $A$.
Since 
\begin{multline*}
y\cdot d_5=((1,3)+(2,4))\cdot(u.(1,4)+v.(2,3))\\=\eta(y+d_5)-\frac{\eta}{2}(u.(3,4)+v^2.(3,4)+v^2.(1,2)+u.(1,2)),
\end{multline*}
we infer that
$q_1:=\frac{2}{\eta}(\eta(y+d_5)-y\cdot d_5)$ belong to $A$.
Then 
$$q_2:=q_1^{\tau_x}=u.(3,4)+v^2.(3,4)+v.(1,2)+u^2.(1,2)$$
and
\begin{multline*}
q_3:=q_2^{\tau_{s_6}}=u.(3,4)+v^2.(3,4)+v.(1,2)^{u^2v.(1,2)}+u^2.(1.2)^{u^2v.(1,2)}\\=
u.(3,4)+v^2.(3,4)+u^2vv^2u^2v.(1,2)+
u^2vuu^2v.(1,2)
\\=u.(3,4)+v^2.(3,4)+uv.(1,2)+
u^2v^2.(1,2)
\end{multline*}
are elements of $A$.
Now we find that
\begin{multline*}
q_1\cdot q_2=(u.(3,4)+v^2.(3,4))^2+(v^2.(1,2)+u.(1,2))\cdot(v.(1,2)+u^2.(1,2))
\\=u.(3,4)+v^2.(3,4)+2\cdot u.(3,4)\cdot v^2.(3,4)+
\eta\cdot\left(v^2.(1,2)+u.(1,2)+v.(1,2)+u^2.(1,2)\right)
\\-\frac{\eta}{2}\cdot\left((v^2.(1,2))^{v.(1,2)}+(u.(1,2))^{v.(1,2)}+
(v^2.(1,2))^{u^2.(1,2)}+(u.(1,2))^{u^2.(1,2)}\right)
\\=(1+\eta)\cdot\left(u.(3,4)+v^2.(3,4)\right)-\eta\cdot u^2v.(3,4)
\\+\eta\cdot\left(v^2.(1,2)+u.(1,2)+v.(1,2)+u^2.(1,2)\right)-
\eta\cdot(1,2)-\frac{\eta}{2}\cdot\left(uv.(1,2)+u^2v^2.(1,2)\right).
\end{multline*}
Denote $d_{19}:=u.(3,4)+v^2.(3,4)$.
Then $q_1\cdot q_2=(1+\eta)d_{19}-\eta s_1-\eta s_3+\eta(q_1-d_{19}+q_2-d_{19})-\frac{\eta}{2}(q_3-d_{19})=
(1-\frac{\eta}{2})d_{19}+\eta(q_1+q_2-\frac{1}{2}q_3-s_1-s_3)$.

If $\eta\neq2$, then we infer that $d_{19}\in A$. Now $d_{20}:=d_{19}^{\tau_{s_2}}=u^2.(3,4)+v.(3,4)$ and $d_{21}:=d_{19}^{\tau_{s_4}}=uv.(3,4)+u^2v^2.(3,4)$ belong to $A$.
Applying $\tau_y$ to these three elements, we get the three remaining extras: $d_{22}:=d_{19}^{\tau_y}$,
$d_{23}:=d_{20}^{\tau_y}$, and $d_{24}:=d_{21}^{\tau_y}$.
Thus, $A$ coincides with $M_H$ and as a consequence we find that the flip subalgebra $A(\tau)$ coincides with $M_H$.

Using GAP, we find that the determinant of the Gram matrix for $A$ with respect to the basis equals 
$$(\eta-2)^{26}\cdot(5\eta-1)^2\cdot(22\eta+1)\cdot(\eta+1/4).$$

Suppose that $\eta=2$. We claim that $A$ has dimension 29. 
Denote $f_1:=d_{22}$,  $f_2:=d_{23}$, and $f_4:=d_{24}$. 
Consider elements $g_{ij}=d_i+f_j$, where $19\leq i\leq 21$ and $1\leq j\leq3$.
Then $A$ is the span of elements $s_1,\ldots,s_6$, $d_1,\ldots,d_{18}$ and $\{g_{ij}~|~19\leq i\leq 21$ and $1\leq j\leq 3\}$. This can be verified similarly to Case~7. So as a basis of $A$ one can take 6 singles, 18 doubles, and elements $d_{19}+f_1$, $d_{19}+f_2$, $d_{19}+f_3$, $d_{20}+f_1$, $d_{21}+f_1$.






\Addresses

\end{document}